\documentclass[reqno]{amsart}
\usepackage{color}
\usepackage{amsmath,amssymb,amsfonts,amsthm,bm}
\usepackage{mathrsfs}
\usepackage{graphicx}
\usepackage{enumerate}
\usepackage[numbers, sort&compress]{natbib}
\usepackage[shortlabels]{enumitem}

 \usepackage[reversemp,paperwidth=155mm,paperheight=235mm,top={22mm},headheight={5.5pt},headsep={5.6mm},text={121mm,195mm},marginparsep=5mm,marginparwidth=12mm, bindingoffset=6mm, footskip=10mm]{geometry}
 
\numberwithin{equation}{section}
 
\usepackage{times} 
\usepackage[varvw]{newtxmath}

 \makeatletter

 \newbox \abstractbox
\renewenvironment{abstract}{\global\setbox\abstractbox=\vbox\bgroup
 \hsize=\textwidth
  \vskip 1.2cm
  \noindent\unskip \textbf{Abstract.}
 }
 {
 \egroup}

\@namedef{subjclassname@2020}{%
	\textup{2020} Mathematics Subject Classification}

\def\@settitle{%
  \bgroup
  \centering
  \vglue1cm
  \fontsize{12}{15}\fontseries{b}\selectfont
  \@title
  \vskip 20pt plus 6pt minus 8pt
  \egroup
}

\def\@setauthors{%
  \begingroup
  \trivlist
  \centering \bfseries
 \normalsize\@topsep30\p@\relax
  \advance\@topsep by -\baselineskip
  \item\relax
  \andify\authors
 {\rmfamily\authors}%
  \endtrivlist
  \endgroup
}

\def\@setaddresses{\par
  \nobreak \begingroup
\normalsize
  \def\author##1{\nobreak\addvspace\bigskipamount}%
  \def\\{\unskip, \ignorespaces}%
  \interlinepenalty\@M
  \def\address##1##2{\begingroup
    \par\addvspace\bigskipamount\noindent
    \@ifnotempty{##1}{(\ignorespaces##1\unskip) }%
    {\ignorespaces##2}\par\endgroup}%
  \def\curraddr##1##2{\begingroup
    \@ifnotempty{##2}{\nobreak\indent{\itshape Current address}%
      \@ifnotempty{##1}{, \ignorespaces##1\unskip}\/:\space
      ##2\par}\endgroup}%
  \def\email##1##2{\begingroup
    \@ifnotempty{##2}{\nobreak\noindent{\itshape E-mail address}%
      \@ifnotempty{##1}{, \ignorespaces##1\unskip}\/:
       ##2\par}\endgroup}%
   \def\urladdr##1##2{\begingroup
    \@ifnotempty{##2}{\nobreak\indent{\itshape URL}%
      \@ifnotempty{##1}{, \ignorespaces##1\unskip}\/:\space
      \ttfamily##2\par}\endgroup}%
  \addresses
  \endgroup
}

  \renewcommand\section{\@startsection{section}{1}{\z@}%
 {27pt plus 6pt minus 8pt}{14pt plus 6pt minus 8pt}
 {\center\normalfont\large\bfseries}}

 \def\subsection{\@startsection{subsection}{2}%
  \z@{.5\linespacing\@plus.7\linespacing}{.2\linespacing}%
  {\normalfont\bfseries}}
\def\subsubsection{\@startsection{subsubsection}{3}%
  \z@{.5\linespacing\@plus.7\linespacing}{-.5em}%
  {\normalfont\bfseries}}

\makeatother

\newcommand{\eps}{\varepsilon}
\newcommand{\norm}[1]{\Vert#1\Vert}
\newcommand{\abs}[1]{\left\vert#1\right\vert}

\newcommand{\inner}[1]{\left(#1\right)}
\newcommand{\comi}[1]{\left<#1\right>}

\newcommand{\normm}[1]{{ \vert\kern-0.25ex \vert\kern-0.25ex \vert #1
		\vert\kern-0.25ex \vert\kern-0.25ex \vert}}

\def\inte#1{
\displaystyle\mathop{#1\kern0pt}^\circ }

\let\pa=\partial
\let\al=\alpha

\let\f=\frac

\let\om=\omega

\let\ka=\theta

\def\cP{{\mathcal P}}

\def\pa{\partial}

\def\virgp{\raise 2pt\hbox{,}}
\def\cdotpv{\raise 2pt\hbox{;}}

 \def\si{\sigma}
\def\C{\mathop{\mathbb C\kern 0pt}\nolimits}
\def\DD{\mathop{\mathbb D\kern 0pt}\nolimits}
\def\EE{\mathop{{\mathbb E \kern 0pt}}\nolimits}
\def\K{\mathop{\mathbb K\kern 0pt}\nolimits}
\def\N{\mathop{\mathbb N\kern 0pt}\nolimits}
\def\Q{\mathop{\mathbb Q\kern 0pt}\nolimits}
\def\R{\mathop{\mathbb R\kern 0pt}\nolimits}
\def\SS{\mathop{\mathbb S\kern 0pt}\nolimits}

\def\<{\langle}
\def\>{\rangle}
\def\S{\mathbb{S}}
\def\R{\mathbb{R}}
\def\T{\mathbb{T}}
\def\Z{\mathbb{Z}}
\def\N{\mathbb{N}}

\def\si{\sigma}
\def\th{\theta}
\def\ga{\gamma}
\def\al{\alpha}
\def\be{\beta}
\def\de{\delta}
\def\vphi{\varphi}

\def\pa{\partial}
\def\vep{\varepsilon}
\def\ZZ{\mathop{\mathbb Z\kern 0pt}\nolimits}
\def\TT{\mathop{\mathbb T\kern 0pt}\nolimits}
\def\P{\mathop{\mathbb P\kern 0pt}\nolimits}

\def\th{\theta}

\newcommand{\beq}{\begin{equation}}
\newcommand{\eeq}{\end{equation}}
\newcommand{\ben}{\begin{eqnarray}}
\newcommand{\een}{\end{eqnarray}}
\newcommand{\beno}{\begin{eqnarray*}}
\newcommand{\eeno}{\end{eqnarray*}}

\newtheorem{thm}{Theorem}[section]
\newtheorem{lem}[thm]{Lemma} 
\newtheorem{corollary}[thm]{Corollary} 
\newtheorem{prop}[thm]{Proposition} 
\theoremstyle{assumption}

\newtheorem{definition}[thm]{Definition}
 \theoremstyle{remark}
\newtheorem{rmk}[thm]{Remark}

\begin{document}

\title[Existence and sharp smoothing effect for the Boltzmann]
{On the Boltzmann equation with soft potentials: Existence, uniqueness and  smoothing  effect of mild solutions}

\author[]{ Ling-Bing He, Jie Ji and Wei-Xi Li}
\address[L.-B. He]{Department of Mathematical Sciences, Tsinghua University\\
Beijing 100084,  P. R.  China.} \email{hlb@tsinghua.edu.cn}
\address[J. Ji]{School of Mathematics, Nanjing University  of Aeronautics and Astronautics\\
	Nanjing 211106,  P. R.  China.} \email{jij\_24@nuaa.edu.cn}
\address[W.-X. Li]{School of Mathematics and Statistics, Wuhan University,
	Wuhan 430072,  China   \&  Hubei Key Laboratory of Computational Science, Wuhan University, Wuhan 430072,  P. R. China } \email{wei-xi.li@whu.edu.cn}



\begin{abstract}
We consider the spatially inhomogeneous Boltzmann equation without angular cutoff  for soft potentials.  For any given  initial datum such that the mass, energy and entropy densities are bounded and  the mass is   away from vacuum,  we establish the local-in-time existence and uniqueness of mild solutions,  and further provide the first result on sharp smoothing  effect in  analytic space or Gevrey space for soft potentials. 
\end{abstract}

\maketitle

\setcounter{tocdepth}{1}
\tableofcontents


\renewcommand{\theequation}{\thesection.\arabic{equation}}
\setcounter{equation}{0}

\section{Introduction}

The spatially inhomogeneous Boltzmann equation reads 
\ben\label{Bolt}
\pa_t f+v\cdot \partial_xf=Q(f,f),\quad f|_{t=0}=f_0,
\een
where $f(t,x,v)$ stands for  the distribution function for particles at position $x$,  time $t\geq0$ with velocity $v$. The Boltzmann collision operator $Q$ is a bilinear operator which acts   on the velocity variable $v$, that is
\beno
Q(f,g)=\int_{\R^3}\int_{\mathbb{S}^2}B(v-v_*,\si)(f'_*g'-f_*g)d\si dv_*,
\eeno
where here and below we use the standard shorthand $f=f(v),g_*=g(v_*),f'=f(v'),g'_*=g(v'_*)$, with $v',v'_*$  given by
	\beno
	v'=\frac{v+v_*}{2}+\frac{|v-v_*|}{2}\sigma \ \textrm{and}\ v_*'=\frac{v+v_*}{2}-\frac{|v-v_*|}{2}\sigma,\quad \si\in\mathbb{S}^2.
	\eeno
	This representation is consistent with the physical conservation laws of moment and energy for   elastic collisions:
	\begin{equation}\label{conservation}
	v+v_*=v'+v'_*,  \quad
	|v|^2+|v_*|^2=|v'|^2+|v'_*|^2.
	\end{equation}
The nonnegative function $B(v-v_*,\si)$ in the collision operator is called the Boltzmann collision kernel.  Physically, it is  assumed to depend only on $|v-v_*|$ and the deviation angle $\theta$ with
\begin{equation}\label{cotheta}
	 \cos\theta= \frac{v-v_*}{|v-v_*|}\cdot\si.
\end{equation} 
In the present work,  the { basic assumptions on the kernel} ${B}$  can be concluded as follows:
\begin{enumerate}[label= $(\boldsymbol{A\arabic*})$ , leftmargin=*, widest=ii]
\item  	 The Boltzmann kernel $B$ takes the form: 
		\begin{equation}
        \label{gamma}
			B(v-v_*,\si)=|v-v_*|^\ga b  (\cos\theta),
		\end{equation}
		 where   $\cos\theta$ is defined by \eqref{cotheta} and $b$ is a nonnegative function.
		 \item  Without loss of generality, we assume  $B(v-v_*,\si)$ is supported on the set $(v-v_*)\cdot\sigma\geq 0$  which corresponds to $ \theta \in[0,\pi/2]$, since  otherwise $B$ may be replaced by its symmetrized form:		 
		  		\beno
		\overline{B}(v-v_*,\si)=B(v-v_*,\sigma)+ B(v-v_*, -\sigma).		\eeno
		 
		 \item We consider the non-cutoff Boltzmann equation and  assume the angular function $b$ in \eqref{gamma}  is not locally integrable and  has the following singular behavior:
		\begin{equation}
			\label{s}
			  \sin\theta b(\cos\theta)\sim   \theta^{-1-2s},\quad 0\leq\th\leq\f\pi2, 
		\end{equation}
		where  $0<s<1$.
		\item  The parameters $\ga$ and $s$ in \eqref{gamma} and \eqref{s} satisfy that  $\ga+2s>-1$ and $-3<\ga< 0$.
\end{enumerate}

\begin{rmk} 
 Note that the Boltzmann kernels  considered here  include the potential of the inverse power law as a typical physical model.
For inverse repulsive potential $\f1{r^{p-1}}$, it holds that $\gamma = \frac {p-5} {p-1}$ and $s = \frac 1 {p-1}$ with $p > 2$. It is easy to check that $\gamma + 4s = 1$ which implies that  assumption $(\boldsymbol{(A4)})$  is fulfilled  for the full range of the inverse power law model with soft potentials (i.e., $\ga<0$). In addition, the cases $\gamma > 0$ and $\gamma = 0$ correspond to so-called hard and Maxwellian potentials, respectively.
\end{rmk}

 \subsection{Function  spaces and statement of the main results}
 
 We will use  $\norm{\cdot}_{L^2}$ and $\inner{\cdot, \cdot}_{L^2}$ to denote the norm and inner product of  $L^2=L^2(\mathbb T_x^3\times \mathbb R_v^3)$   and use the notation   $\norm{\cdot}_{L_v^2}$ and $\inner{\cdot, \cdot}_{L_v^2}$  when the variable $v$ is specified. Similar notation  will be used for $L^\infty$.  Moreover, denote by  $L^p_{x}L^q_v=L^p(\mathbb T_x^3; L^q(\mathbb{R}_v^3))$  the classical Lebesgue space, and  similarly for the Sobolev space $H^p_{x}H^q_v.$

In this paper we will work with solutions which are away from vacuum and  admit bounded mass, energy  and entropy densities.  Precisely,  suppose the initial datum $f_{in}$  in \eqref{Bolt} satisfies that $f_{in}\geq  0$ and that
 \begin{equation}\label{finite}
\forall\ x\in\mathbb T_x^3,\qquad \left\{
\begin{aligned}
	&0<m_0\leq \int_{\mathbb R^3} f_{in}(x,v)dv \leq M_0, \\
& \int_{\mathbb R^3}  f_{in}(x,v) \abs v^2 dv\leq E_0,\\
&  \int_{\mathbb R^3}   f_{in}(x,v)\log \big(1+ f_{in}(x,v)\big )dv\leq H_0,
\end{aligned}
	\right. 
	\end{equation}
where here and throughout the paper $m_0, M_0, E_0$ and $H_0$ are given constants. 	We will prove solutions to the Boltzmann equation \eqref{Bolt} will preserve the above property, saying $f\geq 0$ and  
 \begin{equation}\label{aat}
\forall\ (t,x)\in [0,T]\times \mathbb T_x^3,\qquad  \left\{
\begin{aligned}
	&0<\frac{m_0}{2} \leq \int_{\mathbb R^3} f(t,x,v)dv\leq 2M_0,\\
&\int_{\mathbb R^3} f(t,x,v)	 |v|^2dv\leq 2E_0,\\
&\int_{\mathbb R^3} f(t,x,v)	 \log \big (1+f(t,x,v)\big )dv\leq 2H_0.\\
\end{aligned}
\right.
\end{equation}

Before stating the main result, we first recall the definition of the Gevrey spaces which are denoted by $G^r=G^r(\mathbb T_x^3\times \mathbb R_v^3)$ with $r>0$. 
We say a function $f(x,v)$ belongs to $G^r$, if there exists a constant  $C>0$ such that  
\beno
\forall\ \al,\be\in\Z^3_+,\quad \|\pa^\al_x\pa^\be_vf\|_{L^2}\leq C^{|\al|+|\be|+1}[(|\al|+|\be|)!]^r,
\eeno
or equivalently,   \begin{equation}\label{eqgev}
	e^{c(-\Delta_x-\Delta_v)^{\frac{1}{2r}}} f\in L^2_{x,v}  
\end{equation}
for some constant $c>0.$
Here   $ e^{c(-\Delta_x-\Delta_v)^{\frac{1}{2r}}} f$ is defined by 
\begin{align*}
	\mathcal F_{x,v} \Big(e^{c(-\Delta_x-\Delta_v)^{\frac{1}{2r}}} f\Big)(\xi,\eta)=e^{c(|\xi|^2+|\eta|^2)^{\frac{1}{2r}}}  \mathcal F_{x,v} f(\xi,\eta),
\end{align*}
where $\mathcal F_{x,v}$ represents  the full Fourier transform with respect to $(x,v)$ and $(\xi,\eta)$ are the Fourier dual variable of $(x,v)$. 
In particular, $f$ is real analytic if $r=1$ and   ultra-analytic if $0<r<1$.

\begin{thm}\label{thm:Gevrey}
	Assume the non-cutoff collision kernel $B$  satisfies $\mathbf{(A1)}-\mathbf{(A3)}$ above with the numbers $s,\gamma$ therein satisfying $\mathbf{(A4)}$.  Suppose that the initial datum $f_{in}$ in \eqref{Bolt} is non-negative,  satisfying  condition \eqref{finite} and  that 
	\begin{equation}\label{inquant}
		\norm{e^{a_0(1+|v|^2) } f_{in}}_{H_x^3 L_v^2 }<+\infty
	\end{equation}
	for some constant $a_0>0.$  Then
	the following assertions hold true.
	\begin{enumerate}[label=(\roman*), leftmargin=*, widest=ii]
		\item The Boltzmann equation \eqref{Bolt} admits a unique local solution 
 $f\in L^\infty([0,T],H^3_xL^2_v)$ for some $T>0,$ which is non-negative and satisfies condition \eqref{aat}.  
\item  For  any $0<t\leq T,$ the solution $f(t, x,v)$ is of  Gevrey regular in $x,v$,  that is,      
 \begin{equation*} 
	\forall \  0<t\leq T, \quad  f(t,\cdot)\in G^{\max   \{(2\tau)^{-1}, 1 \}}, 	 	
\end{equation*}
 where here and throughout the paper,
  \begin{equation}\label{tau}
	 	\tau:=\frac{2s}{2-\gamma}.
	\end{equation}
  Moreover,  for  any   $\lambda>\max\big \{1, \frac{1}{2\tau} \big \}$   we can find a constant $C_*>0$,  depending only on $\lambda,$ the quantity in \eqref{inquant}    and the constants $m_0, M_0, E_0$ and $H_0$ in \eqref{finite}, such that the following quantitative estimate
	\begin{equation}\label{alpha1}
	\quad  	\sup_{t\leq T}t^{(\lambda+1)\abs\alpha+ \lambda \abs\beta} \norm{\partial_x^{\alpha}\partial_{v}^{\beta}f(t)}_{H^3_xL^2_v} \leq    C_*^{|\alpha|+|\beta|+1} [(|\alpha|+|\beta|)!]^{\max \{(2\tau)^{-1}, 1\}}
	\end{equation}
	holds true for any $\al,\be\in\Z^3_+$. 
	\end{enumerate} 
\end{thm}

\begin{rmk}
	To overcome  the nonlinearity, it seems that  \(L^\infty_xL^2_v\) is a reasonable space for the local well-posedness theroy, and  it is natural to work with $H^{\frac{3}{2}+}_x L_v^2$ if energy method applies.  In this paper, we  restrict the spatial component to \(H^3_x\) rather than \(H^{\frac{3}{2}+}_x\), in order  to control the lower bound of the density function (see the proof of Proposition \ref{lem:mc0} for detail).   However, if only the smoothing effect is concerned with, then we can reduce the function setting to  $H^{\frac{3}{2}+}_x L_v^2$ and conclude that for any solution $f\in L^\infty([0,T], H^{\frac{3}{2}+}_x L_v^2)$ to \eqref{Bolt} with condition \eqref{aat}  fulfilled,  the quantitative estimate \eqref{alpha1} holds true.   
\end{rmk}

\begin{rmk}
	  It is widely believed that the local well-posedness property is more difficult  in the case of hard potentials; however, the method presented in this text is  applicable for \(0 \leq \gamma \leq 1\).
\end{rmk}

\begin{rmk}Note that the  quantitative estimate \eqref{alpha1} yields  
  the global analytic or Gevrey  regularity, that is, the radius of the convergence is independent of $(x,v)$. For global regularity, the Gevrey index   \(\max \{(2\tau)^{-1}, 1\} = \max \big \{\frac{2-\gamma}{4s}, 1\big \}\) seems to be sharp for initial data with the exponential weight \(e^{a_0 (1+\abs v^2)}\). Interested readers may refer to 
 Subsection \ref{sharpG} for the discussion on the sharpness of the Gevrey index.  More generally,   if we replace the above exponential weight by the sub-exponential weights   \(e^{a_0 (1+\abs v^2)^{\frac{\beta}{2}}}\) with \(\beta \in\big  (\frac{\gamma}{2} + s, 2\big )\), then the Gevrey regularity with index \(\max \big \{\frac{\beta-\gamma }{2\beta s}, 1\big \}\) can also be achieved  though with more intricate calculations.
\end{rmk}

\begin{rmk}
When the grazing collisions are considered, the Landau equation can be derived as the grazing limit of the Boltzmann equation.  The same results to that in Theorem \ref{thm:Gevrey}  with $s$ therein replaced by $1$,  still hold for the Landau equation, including  the Coulomb potential case.
\end{rmk}

\subsection{Brief review} 
Since the well-posedness and smoothing effect theory are well-explored for the non-cutoff Boltzmann equations, in what follows, we will provide a brief overview of the previous works closely related to our result.

 $\bullet$ \textbf{Local-wellposedness:} There are several existence results  for the inhomogeneous non-cutoff Boltzmann equation.    In \cite{MR1851391, MR2679369,  MR2765735, MR3177640}, the initial data are in higher order Sobolev spaces and  required to exhibit Gaussian decay in velocity. In \cite{MR4704643,MR3376931, MR4112183,MR4179789,MR4416998}, local well-posedness was established under the condition that the initial data have polynomial decay.  Recently,   a local solution was constructed by Henderson-Snelson-Tarfulea \cite{CSA2} with initial data in a  weighted \(L^\infty_{x,v}\)-space. We believe that a reasonable  function space for existence and uniqueness  may be   \(L^\infty_xL^2_v\) with polynomial decay in velocity. Our main objective in this work is to minimize the regularity requirement in the velocity space, specifically targeting \(L^2_v\). However, Sobolev regularity in the spatial variable and exponential decay in the  velocity variable are still required.

$\bullet$ \textbf{Gevrey regularity:} Due to the diffusive properties of the Boltzmann collision operator, as discussed in \cite{MR1765272, MR2863853, MR1949176, MR1639275, MR2784329, MR3942041}, the equation has attracted significant attention in the study of regularity issues. Extensive research has been conducted on the $C^\infty$ or $H^{+\infty}$ smoothing effects for the non-cutoff Boltzmann equation and related models, seeing e.g. \cite{MR2885564, MR2038147, MR3325244, MR2679369,MR4433077}. Additionally, Gevrey or Gelfand smoothing properties have been explored for these or related models, as discussed in \cite{MR2679746, MR4356815,  MR2467026, MR3665667, MR4612704} and their references. Different from the heat   equation,  the spatially inhomogeneous  Boltzmann equation is a degenerate parabolic equation.  It is highly non-trivial to improve the  Gevrey regularity to analyticity for degenerate elliptic operators.   In fact for the inhomogeneous Boltzmann equations,  so far very few analytic solutions  are available.  Recently, in \cite{chenlixu2024}, the author obtained  analytic or sharp Gevrey regularity for the inhomogeneous non-cutoff Boltzmann equation with hard potentials in the perturbation setting around  the normalized global Maxwellian.   In this work, we aim to investigate the analytic and sharp Gevrey class regularity effects for the spatially inhomogeneous non-cutoff Boltzmann equation with soft potentials. To our knowledge, this is the first work on analytic or Gevrey regularity for soft potentials.

\subsection{Difficulties and methodology} 

To investigate the well-posedness property of the Boltzmann equation,  
one main difficulty  lies  in  the  low  regularity.  In fact,  the data considered in this text only belong to $L^2_v$ for the velocity variable. So to achieve  the coercivity estimate, it is crucial  to continuously ensure both the lower and upper bounds of the  density, as well as the upper bounds on the energy and the entropy (see \eqref{finite} for definitions).  Another difficulty is 
  the lack of smallness on the  data, different from the perturbation setting. This may prevent us to close the energy when estimating the   nonlinear term $Q(f, f)$ with  the top  derivatives involved in. To ensure the  lower or upper bounds in \eqref{finite} are preserved at positives times, we   impose    \(H^3_x\)-regularity in the spatial variable, taking the advantage of the fact that  \(\|\partial_x f\|_{L^\infty} \leq C \|f\|_{H^3_x}\).  Moreover, to overcome the absence of the smallness condition,  we will work with the Sobolev space   \(\mathcal{H}^3_{x,N}L^2_v\) (see \eqref{mathch2} below for the definition) which is  equivalent to the classical Sobolev space  \(H^3_{x}L^2_v\). More details can be found in Lemma \ref{lem:j1}.
     
   For soft potentials, a major difficulty is that the collision operator \(Q\) is not uniformly elliptic and degenerates at infinity in the velocity variable (see \eqref{lower} and \eqref{trinorm}). To overcome this difficulty, a typical approach is to use some well-chosen weights to compensate for the degeneration, or conversely, to accept a trade-off in regularity for additional weights. Previous works \cite{CSA2, HJ2, MR4433077} demonstrate that solutions may  become  \(C^\infty\) at positive times for initial data with any order weights.  In \cite{HJ2}, it was shown that for the initial data with only polynomial weights in the velocity variable, the solution can only gain finite Sobolev regularity when \(\gamma + 2s < 0\). This implies that exponential weights are necessary to achieve analytic or Gevrey regularity.  
   
    It is well-known that the spatially inhomogeneous Boltzmann equation without cut-off admits a hypoelliptic property,  which can transfer regularity from the velocity variable to the spatial variable (see  e.g.\cite{MR1949176}).  However, this indirect method can only improve  $C^\infty$-regularity  to the Gevrey class of index \(\frac{1+2s}{2s}\) (see \cite{MR4356815}), even in the case of hard potentials. Thus, new methods are needed to address the regularity of both \(x\) and \(v\) at the same time. For this purpose  we introduce auxiliary vector fields that enable us to obtain regularity for both spatial and temporal variables simultaneously. As with hard potentials, the analytic or sharp Gevrey class smoothing effect depends on a quantitative estimate of directional derivatives with respect to the auxiliary vector fields \(H_\delta\), introduced in \cite{chenlixu2024}. These fields are defined by:
\begin{equation}
   	\label{vecM}
   H_\delta= \frac{1}{\delta+1}t^{\delta+1} \partial_{x_1}+ t^{\delta} \partial_{v_1},
   \end{equation}
where \(\delta \geq 1\) is an arbitrary constant. The advantage of these vector fields is that the spatial derivatives do not appear in the commutator between \(H_\delta\) and the transport operator, as observed from:
\begin{equation*} 
  	[\partial_t+v\,\cdot\,\partial_x, \,\,   H_\delta]=\delta t^{\delta-1}\partial_{v_1},
  	   \end{equation*}
where \([\cdot, \cdot]\) denotes the commutator between two operators. More generally,
 \begin{equation}\label{kehigher}
\forall\ k\geq 1,\quad 	[\partial_t+v\,\cdot\,\partial_x, \,\,   H_\delta ^k]=\delta kt^{\delta-1}\partial_{v_1} H_\delta^{k-1}.
\end{equation}
Note that \eqref{kehigher} can be derived using induction on \(k\) (see \cite{chenlixu2024} for details). This, combined with uniform elliptic estimates in  Corollary \ref{corollary:coer}, allows us to apply the diffusion in the velocity variable to obtain a crucial estimate of the directional derivatives \(H^k_\delta f\). Moreover, classical derivatives can be generated as linear combinations of \(H_\delta\) (see \eqref{generate}), which provides the desired quantitative estimate on classical derivatives.

 In this text, let $\lambda>\max\big \{1, \frac{1}{2\tau} \big \}$ be an arbitrary  given number.  We define   $\delta_1$ and $\delta_2$   by setting
\begin{equation}\label{de1de2}
	\delta_1 :=\lambda, \quad 	\delta_2:= \left\{
	\begin{aligned}
	&1,  \  \textrm{ if } \  {\ga\over 2}+2s \geq 1, \\
	&	 1+(1-2\tau)\lambda, \  \textrm{ if } \ {\ga\over 2}+2s < 1.
	\end{aligned}
	\right.
\end{equation}
Direct verification shows that
\begin{equation*}
\delta_1>\delta_2\geq 1.	
\end{equation*}
Accordingly, 
  let $H_{\delta_1}$ and  $H_{\delta_2}$ be  defined by \eqref{vecM}:
\begin{equation*}
	 H_{\delta_1}=\frac{1}{\delta_1+1} t^{\delta_1+1}\partial_{x_1}+t^{\delta_1}\partial_{v_1} \ \textrm{ and } \  H_{\delta_2}=\frac{1}{\delta_2+ 1} t^{\delta_2+1}\partial_{x_1}+t^{\delta_2}\partial_{v_1}.
\end{equation*}
Then  $\partial_{x_1}$ and $\partial_{v_1}$ can be generated by the linear combination of $H_{\delta_j}, j=1,2$, that is,
\begin{equation}
	\label{generate}
\left\{
\begin{aligned}
&t^{\lambda+ 1}\partial_{x_1}=	t^{\delta_1+ 1}\partial_{x_1}=\frac{(\delta_2+ 1)(\delta_1+1)}{\delta_2-\delta_1}  H_{\delta_1}-\frac{(\delta_2+ 1)(\delta_1+1)}{\delta_2-\delta_1} t^{\delta_1-\delta_2}H_{\delta_2},\\
&t^{\lambda}\partial_{v_1}=t^{\delta_1}\partial_{v_1}=-\frac{\delta_1+ 1}{\delta_2-\delta_1} H_{\delta_1}+\frac{\delta_2+ 1}{\delta_2-\delta_1}t^{\delta_1-\delta_2}H_{\delta_2} .
\end{aligned}
\right.
\end{equation}
This enables to control the classical  derivatives in terms of the directional derivatives in $H_{\delta_1} $ and $H_{\delta_2}.$
More generally, we define 
\begin{equation}\label{hdel12}
	H_{\delta_i, j}=\frac{1}{\delta_i+1} t^{\delta_i+1}\partial_{x_j}+t^{\delta_i}\partial_{v_j}.
\end{equation}
with $\delta_1, \delta_2$ defined in \eqref{de1de2} and $1\leq j\leq 3.$

Finally, to establish the well-posedness of the nonlinear Boltzmann equation, we use an approximation method involving linear equations. More details on this approximation process can be found in Sections \ref{sec:hinfty}
and \ref{sec:sobsmo}.

 \subsection{Notations and quantitative estimate for the collision operator}
We first list some notations used throughout the paper. 
  For a vector-valued  function or operator $A=(A_1,A_2, \ldots, A_n)$,  
  we used the convention that 
  \begin{equation*}
  	\norm{A} =\bigg (\sum\limits_{1\leq k\leq n}\norm{A_k}^2\bigg )^{\frac12}
  \end{equation*}
  for some given norm $\norm{\cdot}$.  For two operators  $T$ and $S,$   the commutator $[T,S]$ is defined by $[T,S]=TS-ST$.

We use the notation that $\comi{\cdot}=(1+\abs{\cdot}^2)^{1\over2}$. So for $v\in\mathbb R^3$ we have
$\comi v=(1+\abs v^2)^{1\over 2}$. Meanwhile $\comi {D_x}^r$ and  $\comi {D_v}^r, r\in\mathbb R,$ stand  for two Fourier multipliers with symbols 
 $
\comi {\xi }^r $ and $\comi {\eta}^r,
$
respectively, that is,
\begin{equation*}
\mathcal F_x\big(	\comi {D_x}^r h\big)(\xi, v)=\comi \xi ^r  (\mathcal F_x h)(\xi, v) \textrm{ and } \mathcal F_v\big(	\comi {D_v}^r h\big)(x, \eta)=\comi \eta ^r (\mathcal F_v h)(x, \eta).
\end{equation*}
Here and below, $\mathcal F_x$ and $\mathcal F_v$ stand  for the partial Fourier transform with respect to $x$ and $v, $ respectively, and $(\xi,\eta)\in\mathbb Z^3\times\mathbb R^3 $ is the Fourier dual variable of $(x,v)\in \mathbb T^3\times\mathbb R^3.$  Similarly,  $\comi{D_{x,v}}^{r}$ is the Fourier multiplier with symbol $(1+\abs\xi^2+\abs\eta^2)^{\frac{r}{2}}.$

We recall some facts on the  lower and upper bounds established by \cite{MR3942041,MR2959943} for the collision operator $Q$. The first one is concerned with the trilinear estimate.  If $\ga<0$, then it holds that (cf. \cite[Theorem 1.4]{MR3942041} for instance)
	\begin{equation}\label{uppbound}
	|(Q(g,h),f)_{L_v^2}| \leq C \|\comi v^7g\|_{L^2_v}\big (\normm{h}+\|\comi v^{\frac{\gamma}{2}+s} h\|_{L^2_{v}}\big )\normm{f},
	\end{equation}
where here and below the triple-norm $\normm{\cdot}$ is defined by     
\begin{equation}\label{trinorm}
	\normm{f}^2:=\norm{\comi v^{\frac{\gamma}{2}}\comi{D_v}^sf}_{L^2_v}^2+\norm{\comi v^{\frac{\gamma}{2}}(-\triangle_{\SS^2})^{s/2} f}_{L_v^2}^2.
\end{equation}
Note \eqref{uppbound} is just a specific case of \cite[Theorem 1.4]{MR3942041}  by choosing $(a,b)=(s,s)$, $(a_1,b_1)=(0,s)$ and $(w_1, w_2)=\big (\frac{\gamma}{2}+s, \frac{\gamma}{2}\big)$ therein.  Next we recall the coercivity estimate for the Boltzmann collision operator. 
Suppose $g=g(t,x,v)$ is a non-negative function  satisfying   that
	\begin{equation*}
	\|g(t,x,\cdot)\|_{L_v^1}\geq m_0, \ \,  \|\comi v^2 g(t,x,\cdot)\|_{L^1_v}\leq E_0 \ \textrm{ and }\    \|g(t,x,\cdot)\|_{L\log L}<H_0,
	\end{equation*}
	with $m_0, E_0, H_0>0$ three positive constants. Then there exist two constants $C_1,   C_2>0$, depending only on   $m_0, E_0$ and $ H_0$ above, such that	(cf.  \cite[Theorem 1.2]{MR3942041} or \cite{MR2959943} for instance)
	\begin{equation}\label{lower+}
			(-Q(g,f),f)_{L^2_v}\geq C_1	\normm{f}^2-  C_2\big(1+ \|\comi v^6 g\|_{L^2_v}\big)\|\comi v^{\frac{\gamma}{2}+s}f\|^2_{L^2_{v}}.
	\end{equation}
Interested readers may refer to \cite{MR3942041} for the proof of 	\eqref{uppbound} and \eqref{lower+}. 
  Note that  the constants $C_1$ and $C_2$ above are independent of $t,x$, and thus we integrate \eqref{lower+} with respect to $x\in \mathbb T^3$ to conclude that
  \begin{equation}
  	\label{lower}
  		(-Q(g,f),f)_{L^2}\geq C_1	\normm{f}_{L^2_x}^2-  C_2\big(1+ \|\comi v^6 g\|_{L_x^\infty L^2_v}\big)\|\comi v^{\frac{\gamma}{2}+s}f\|^2_{L^2},
  \end{equation}    
  recalling $L^2=L^2_{x,v}$.

\subsection{Dyadic decomposition and sharpness of Gevrey index}\label{sharpG}

Before discussing the sharpness of Gevrey index,  we first recall  some basic facts on the dyadic decomposition (cf. \cite{MR2768550} for more details).    Define  a ball $\mathcal B$ and a ring $\mathcal C$ by setting
\begin{equation*}
	\mathcal B:=\big \{\xi\in\R^3;\  |\xi|\leq 4/3\big \}, \quad \mathcal C:=\big \{\xi\in\R^3;\  3/4\leq|\xi|\leq 8/3\big \}.
\end{equation*}
Accordingly, one may introduce two radial functions $\psi\in C_0^\infty(\mathcal B)$ and $\vphi\in C_0^\infty(\mathcal C)$ which satisfy that
\begin{equation*}
\left\{
\begin{aligned}
& 0\leq  \psi,\vphi\leq1,~\textrm{ and }\psi(\xi)+\sum_{j\geq0}\vphi(2^{-j}\xi)=1 \textrm{ for any } \xi\in\R^3,\\
& |j-k|\geq2 \Rightarrow  \textrm{Supp}~\vphi(2^{-j}\cdot)\cap \textrm{Supp}~\vphi(2^{-k}\cdot)=\emptyset,\\
&  j\geq1 \Rightarrow  \textrm{Supp}~\psi\cap \textrm{ Supp } \vphi(2^{-j}\cdot)=\emptyset.
\end{aligned}
\right.
\end{equation*}
 Moreover, we define two dyadic operators $\mathcal P_j$ and $\Delta_j$  in the phase space and its frequency  space, respectively, that is, 
\ben\label{Defcpj}
\mathcal P_{-1}f(v):=\psi(v)f(v),~\mathcal P_jf(v):=\vphi(2^{-j}v)f(v),~j\geq0,
\een
and
\ben\label{Deffj}
\Delta_{-1}f(\xi):=\psi(D)f(\xi),~\Delta_jf(\xi):=\vphi(2^{-j}D)f(\xi),~j\geq0,
\een 
recalling $\psi(D)$ is the Fourier multiplier with symbol $\psi(\xi)$, and similarly for $\vphi(2^{-j}D).$
Then for any  $u\in \mathcal{S}'$,  the dual of Schwartz space $\mathcal{S}$,   we have 
\begin{equation*}
	u=\mathcal P_{-1}u+\sum\limits_{j\geq0}\mathcal P_ju=\Delta_{-1}u+\sum_{j\geq0}\Delta_ju, 
\end{equation*}
which holds in the sense of distribution. 
We may use  the dyadic operator $\Delta_j$, to characterize weighted  Sobolev spaces (cf. \cite[Section 2.7]{MR2768550}) as follows:
\begin{equation*}
\|u\|^2_{H_v^m}\sim	\sum_{j=-1}^\infty2^{2jm}\|\Delta_j u\|^2_{L_v^2},\quad m\in\mathbb R,
	\end{equation*}
where here and below  by  $A\sim B$ we mean $C^{-1}A\leq B\leq CA$ for some generic constant $C$.   
More generally, we have the following characterization of   weighted Sobolev spaces:
\begin{equation}\label{chara}
	\norm{\comi v^p \comi{D_v}^m u}^2_{L_v^2}\sim \norm{\comi{D_v}^m\comi v^p  u}^2_{L_v^2}\sim   \sum_{j, k=-1}^\infty2^{2kp}  2^{2mj}\|\Delta_j\mathcal P_ku\|^2_{L^2}.
\end{equation}
We refer to  \cite[Lemmas 4.11 and 4.15]{MR4704643} for the proof of \eqref{chara}.

Induced by \eqref{lower}, it suffices to consider the following simplified toy model of the Boltzmann equation to show the sharpness of Gevrey index: 
\beno
\pa_t f+v\cdot\partial_xf+ \<v\>^{\ga}(-\Delta_v)^sf+ \<v\>^{\ga} (-\Delta_{\S^2})^{s}f=0.
\eeno
Furthermore, if we restrict the spatially  homogeneous case   and radial case (i.e., $f=f(t,v)=f(t, |v|)$),  it is reasonable  to  consider the following  toy model:
\ben\label{toymodel}
\pa_tf +\<v\>^\gamma \<D_v\>^{2s}f=0.
\een
We will demonstrate  the sharp Gevrey regularity for the above toy model  via localized method.  
Recall that  the dyadic operators $\mathcal P_j$ and $\Delta_j$ in phase space and frequency space are defined in \eqref{Defcpj} and \eqref{Deffj}.  By virtue of \eqref{chara}, formally  we may rewrite    \eqref{toymodel} as 
\beno
\pa_t \Delta_j\mathcal P_kf+2^{2sj}2^{\ga k}\Delta_j\mathcal P_kf \approx  0.
\eeno
Thus, at $t=1$ we have 
\beno
\Delta_j\mathcal P_kf(1,v) \approx e^{-2^{2sj}2^{\ga k}}\Delta_j\mathcal P_kf_{in}(v) \approx e^{-2^{2sj}2^{\ga k}}e^{-a_0 2^{2k}}\Delta_j\mathcal P_k\big(e^{a_0\<v\>^2}f_{in}(v)\big),
\eeno
or equivalently, 
\begin{equation*}
	 e^{ 2^{j\frac{1}{r}}} \Delta_j\mathcal P_kf(1,v)\approx e^{2^{j\frac{1}{r}}} e^{-2^{2sj}2^{\ga k}}e^{-a_0 2^{2k}}\Delta_j\mathcal P_k\big(e^{a_0\<v\>^2}f_{in}(v)\big). 
\end{equation*}
In view of \eqref{eqgev}, 
if the index of the Gevrey function $\Delta_j\mathcal P_kf(1,v)$ is $r$, then 
%
%
%
 it is necessary that 
\beno
\forall \ j\in\mathbb Z_+,\quad \inf_{k\in \mathbb Z_+}\Big(2^{2sj}2^{\ga k}+  a_02^{2k}\Big) \gtrsim  2^{j\frac{1}{r}}.
\eeno
Direct verification shows
that
\begin{equation*}
	\inf_{k\in\mathbb Z_+}\Big(	2^{2sj}2^{\ga k}+  a_02^{2k}\Big)\gtrsim   2^{2sj}2^{2sj\frac{\gamma}{2-\gamma}}\gtrsim 2^{\frac{4s}{2-\gamma}j}.
\end{equation*}
Combining the above estimate we have   $r=\frac{2-\gamma}{4s}.$ 

On the other hand, note the coefficient $\comi v^\gamma=(1+\abs v^2)^{\gamma/2}$ in \eqref{toymodel} is only (locally) analytic but not ultra-analytic for $ \gamma<0$.  Then heuristically it seems reasonable that the ultra-analyticity could not be achievable and the analyticity  should be  the best regularity setting  we may expect for the toy model \eqref{toymodel}. 
Thus  $\max\{\frac{2-\gamma}{4s},\ 1\}$ can be regarded as the sharp   Gevrey index.

\subsection{Organization of the paper} The rest of this paper is arranged as follows. In Section 2, we introduce some auxiliary lemmas, including estimates and applications related to the collision operator. Sections 3, 4, and 5 focus on deriving \emph{a priori} estimates for the analyticity and Gevrey regularity of solutions to the linear Boltzmann equation. In Sections 6 and 7, we establish the well-posedness and Gevrey regularity of local solutions for both linear and nonlinear Boltzmann equations. Additional supplementary material is included in the Appendix.

\section{Preliminaries}
In this part we list some facts to be used frequently when proving the main result, which are concerned with the uniform elliptic estimate in velocity, commutator estimates and  quantitative estimates on the collision operator.  To lighten the notations, we will use $C$ in the following discussion to denote some generic constant that may vary from line to line. 

\subsection{Regularity in velocity} We begin with the uniform elliptic estimate in velocity variable.  The main result can be stated as follows, which may  be regarded as a specific interpolation between   weighted Sobolev spaces. Its proof relies on the dyadic decomposition  in Subsection \ref{sharpG}.
 
\begin{lem}\label{lem:velo} Let  $\tau$ be defined in \eqref{tau}, that is, 
	$
		\tau=\frac{2s}{2-\gamma}.
	$ 
	There exists a constant $C>0$ depends on $s$ and $\gamma$ such that  
	\begin{equation*}
  \|\comi {D_v}^\tau u\|_{L_v^2}\leq C\inner{	\norm{\comi v u}_{L_v^2}+	\norm{\comi v^{\gamma/2} \comi {D_v}^s u}_{L_v^2}} 
	\end{equation*}
holds true for any regular function $u$,  where here and below by regular functions we mean that the norms involved of these functions are finite.   	
\end{lem}

 \begin{proof} 
 We recall the characterization \eqref{chara} of   weighted Sobolev spaces in terms of the dyadic decomposition, that is, 
\begin{equation}\label{chara+}
	\norm{\comi v^p \comi{D_v}^m u}^2_{L_v^2}\sim \norm{\comi{D_v}^m\comi v^p  u}^2_{L_v^2}\sim   \sum_{j, k=-1}^\infty2^{2kp}  2^{2mj}\|\Delta_j\mathcal P_ku\|^2_{L^2}.
\end{equation}
For any pairs $(m_j, p_j), j=1,2,3$, satisfying that
\begin{equation}\label{m1m2}
	 m_1=m_2\th +m_3(1-\th) \textrm{ and }  p_1=p_2\th+p_3(1-\th) \textrm{  for  some} \  0<\theta<1, 
\end{equation}
it follows from Cauchy inequality  that 
\begin{multline*}
\sum_{j,k=-1}^\infty 2^{2m_1 j}2^{2p_1k}\|\Delta_j\mathcal P_ku\|^2_{L^2}\\
\leq \Big(\sum_{j,k=-1}^\infty 2^{2m_2 j}2^{2p_2k}\|\Delta_j\mathcal P_ku\|^2_{L^2}\Big)^{\th}\Big(\sum_{j,k=-1}^\infty 2^{2m_3 j}2^{2p_3k}\|\Delta_j\mathcal P_ku\|^2_{L^2}\Big)^{1-\th}.
\end{multline*}
This with \eqref{chara+}  yields that, under condition \eqref{m1m2},
	\begin{align*}
		\norm{\comi v^{p_1}\comi{D_v}^{m_1}u}_{L^2_{v}}\leq 	C \norm{\comi v^{p_2}\comi{D_v}^{m_2}u}_{L^2_{v}}^\theta	\norm{\comi v^{p_3}\comi{D_v}^{m_3}u}_{L^2_{v}}^{1-\theta}.
	\end{align*}
In particular,
	we apply  the above estimate by choosing 
	\begin{equation*}
		(m_2,p_2)=\Big (s,\frac{\gamma}{2}\Big), \ \ (m_3,p_3)= (0,1),\ \  \theta=\frac{2}{2-\gamma},
	\end{equation*}
	this gives $(m_1, p_1)=(\tau, 0)$ with $\tau =\frac{2s}{2-\gamma}$. Hence
	 \begin{align*}
	\|\comi{D_v}^{\tau} u\|_{L^{2}_v}\leq C\|\comi{v}^{\frac{\gamma}{2}}\comi{D_v}^su\|_{L^2_{v}}^{\theta}\|\comi v u\|_{L^2_{v}}^{1-\theta} \  \textrm{  with  } \ \theta=\frac{2}{2-\gamma}. 
\end{align*}	
Using Young's inequality, the proof of Lemma \ref{lem:velo} is thus completed.
\end{proof}

As an immediate consequence of Lemma \ref{lem:velo}, we have the following uniform elliptic estimate in velocity. 

\begin{corollary}\label{corollary:coer}
There exists a constant $C>0$ such that, for any regular function $u$ we have
	\begin{equation*}
	\|\comi {D_v}^{\tau} u\|_{L_v^2}\leq C\inner{	\norm{\comi v u}_{L_v^2}+	\normm{ u}},
	\end{equation*}
	where  $\normm{\cdot}$ is defined in \eqref{trinorm} and  $\tau $ is given in \eqref{tau}.
\end{corollary}

\subsection{Commutator estimates and applicatons}

 To overcome the lack of the uniform ellipticity in  the velocity variable,   we  will  perform estimates in a weighted Sobolev space  with time-dependent weight function $\omega=\omega(t,v)$  defined by 
\begin{equation}\label{weifun}
	\omega=  e^{(a_0-t) \<v\>^2},\quad a_0>0\quad\mbox{and}\quad 0\leq t\leq\frac{a_0}{2}.
\end{equation} 
  The following properties will be used frequently:  for any $(t,v)\in \big[0,\frac{a_0}{2}\big ]\times\mathbb R_v^3,$
\ben\label{vweight}
	\partial_t\omega(t,v)=- \<v\>^2\omega(t,v),\quad |\partial_{v_j}\omega(t,v)|\leq 2a_0 \comi v \omega(t, v) \ \textrm{ with }\  j=1,2,3. 
\een
We first deal with  the commutator between   the collision operator and  weight functions.

\begin{lem}\label{Ighf}
	Define $\mu_\theta$  by  setting $\mu _\theta=e^{-\ka\<v\>^2}$ with $\theta=\theta(t)$ a given function. Then there exists a constant $C>0$ such that for any regular $g,h$ and $f$, we have that
	\begin{multline*}\label{Ig1}
	\big|\big(\mu^{-1}_\theta Q(g,\ h)-Q(g,\ \mu^{-1}_\theta h),\  f \big )_{L^2_v}\big| \\
	\leq  C \|\mu_\th^{-1} g\|_{L_v^2}\big (\normm{\mu_\th^{-1}h} +\|\comi v^{\frac{\ga}{2}+s}\mu_\th^{-1} h\|_{L^2_{v}}\big )\|\comi v^{\frac{\ga}{2}+s}f\|_{L^2_{v}},
	\end{multline*}
 recalling $\normm{\cdot}$ is defined by \eqref{trinorm}. 
\end{lem}

\begin{proof} 
It follows from the energy conservation law in \eqref{conservation} that 	 
\begin{align*}
	\mu_\theta^{-1}=(\mu_\theta^{-1})_*' (\mu_\theta^{-1})' (\mu_\theta)_*,
\end{align*}
recalling  the standard shorthand that $f=f(v),f_*=f(v_*),f'=f(v')$ and $f'_*=f(v'_*)$.  
This enables us to write
\begin{align*}
	\mu_\ka^{-1}Q(g,h) &=	\int_{\R^3}\int_{\mathbb{S}^2}\mu_\ka^{-1}B(v-v_*,\si)(g'_*h'-g_*h)d\si dv_*\\
	&= \int_{\R^3}\int_{\mathbb{S}^2}B(v-v_*,\si)\Big ((\mu_\theta)_* (\mu_\ka^{-1} g)_*'  (\mu_\theta^{-1} h)' - g_*(\mu_\theta^{-1} h) \Big )d\si dv_*.
\end{align*}
On the other hand,  observe 
\begin{equation*}
	Q(g,\ \mu_\ka^{-1}h)  
	= \int_{\R^3}\int_{\mathbb{S}^2}B(v-v_*,\si)\Big ( (\mu_\theta)_*' (\mu_\theta^{-1}g)'_*(\mu_\theta^{-1} h)'- g_*(\mu_\theta^{-1} h)\Big )d\si dv_*.
\end{equation*}
As a result, we have
\begin{multline*}
 	\mu_\ka^{-1}Q(g, h)-Q(g,\mu_\ka^{-1}h)  \\
 = \int_{\mathbb R^3}\int_{\S^2} B(v-v_*,\si)\big [(\mu_\ka)_*-(\mu_\ka)'_*\big](\mu_\ka^{-1}g)'_*(\mu_\ka^{-1}h)'dv_*d\si.
\end{multline*}
Thus
\begin{equation*}
	\begin{aligned}
		& \big(	\mu_\ka^{-1}Q(g,h)-Q(g,\mu_\ka^{-1}h),\ f\big)_{L_v^2}\\
		&\quad=\int_{\R^6\times\S^2}B(v-v_*,\si)\big [(\mu_\ka)_*-(\mu_\ka)'_*\big ](\mu_\ka^{-1}g)'_*(\mu_\ka^{-1}h)' f d\sigma dv_*dv\\
		&\qquad=\int_{\R^6\times\S^2}B(v-v_*,\si)\big[(\mu_\theta)'_*-(\mu_\theta)_* \big ](\mu^{-1}_{\theta} g)_*(\mu^{-1}_{\theta}h) f'd\sigma dv_*dv,
	\end{aligned}
\end{equation*}
where  the second equality follows from the standard   pre-postcollisional change of variables: 
\begin{equation}
	\label{prepost}
	(v,v_*,\sigma) \rightarrow(v',v_*', (v-v_*)/|v-v_*|),
\end{equation}
 which has unit Jacobian (cf. \cite[Subsection 4.5]{MR1942465} for instance).  This with the fact    that  
\begin{equation*}
	 (\mu_\theta)'_*-(\mu_\theta)_*  =[(\mu_\ka^{1\over4})_*'+(\mu_\ka^{1\over4})_*\big]^2\big[(\mu_\ka^{1\over4})_*'-(\mu_\ka^{1\over4})_*\big]^2+2(\mu_\theta^{1\over2})_*\big [(\mu_\ka^{1\over2})_*'-(\mu_\ka^{1\over2})_*\big],
\end{equation*}
yields that
\begin{equation}\label{j1j2j3}
	\big|\big(	\mu_\ka^{-1}Q(g,h)-Q(g,\mu_\ka^{-1}h),\ f\big)_{L_v^2} \big|\leq  	\mathcal{I}_1+	\mathcal{I}_2+	\mathcal{I}_3
\end{equation}
with
\begin{align*}
\mathcal{I}_1&:=\Big| \int_{\R^6\times\S^2} B  \big [(\mu_\ka^{1\over4})_*'+(\mu_\ka^{1\over4})_*\big]^2\big[(\mu_\ka^{1\over4})_*'-(\mu_\ka^{1\over4})_*\big]^2(\mu^{-1}_{\ka} g)_*(\mu^{-1}_{\ka}h) f'd\sigma dv_*dv\Big|,\\
	\mathcal{I}_2&:= 2\Big| \int_{\R^6\times\S^2} B  \big [(\mu_\ka^{1\over2})_*'-(\mu_\ka^{1\over2})_*\big ](\mu^{1\over2}_\ka\mu^{-1}_{\ka} g)_*\big[\mu^{-1}_{\ka}h-(\mu^{-1}_{\ka}h)'\big ]f'd\sigma dv_*dv\Big| ,\\
	\mathcal{I}_3&:= 2\Big| \int_{\R^6\times\S^2} B \big [(\mu_\ka^{1\over2})_*'-(\mu_\ka^{1\over2})_*\big](\mu^{1\over2}_\ka\mu^{-1}_{\ka} g)_*(\mu^{-1}_{\ka}h)'f'd\sigma dv_*dv\Big|,
\end{align*}
where we write $B=B(v-v_*,\si)$ for short. 
The upper bound of $\mathcal {I}_1$ and $\mathcal {I}_3$ just  follows from   \cite[Proposition 2.7]{MR4515094}. In fact  observe that  the estimates  in  \cite[Proposition 2.7]{MR4515094}  are uniform with respect to the parameter $\eps,$ and thus we repeat  the argument in the proof of \cite[Proposition 2.7]{MR4515094} (seeing Step 1 and Step 3     therein)  to conclude   that
\begin{equation}\label{j1j3}
	\mathcal {I}_1  + \mathcal {I}_3 \leq C \|\mu^{-1}_\ka g\|_{L_v^2}\|\comi v^{\frac{\ga}{2}+s}\mu^{-1}_\ka h\|_{L^2_{v}}\|\comi v^{\frac{\ga}{2}+s}f\|_{L^2_{v}}.
\end{equation} 
It remains to deal with  $\mathcal {I}_2$. To do so we write 
\begin{equation}\label{j2}
\begin{aligned}
	 \mathcal{I}_2 & \leq 2  \Big(\int_{\R^6\times\S^2}B(v-v_*,\si) \,\big|(\mu^{1\over2}_\ka\mu^{-1}_\ka g)_*\big|\, \big[ \mu^{-1}_\ka h-(\mu^{-1}_\ka h)'\big ]^2d\si dv_*dv\Big)^{1\over2}\\
& \quad\times \Big(\int_{\R^6\times\S^2}B(v-v_*,\si)\big[(\mu_\ka^{1\over2})_*'-(\mu_\ka^{1\over2})_*\big]^2\,\big|(\mu^{1\over2}_\ka\mu^{-1}_\ka g)_*\big|\,(f')^2d\si dv_*dv\Big)^{1\over2}\\
&:=2 (\mathcal {I}_{2,1})^{1\over2}(\mathcal{I}_{2,2})^{1\over2}.
\end{aligned}
\end{equation}
Using  \cite[Proposition 2.7]{MR4515094} again (seeing Step 2 in the proof therein), we have 
\begin{equation}\label{j22}
	\mathcal {I}_{2,2}\leq C \|\mu^{-1}_\ka g\|_{L_v^2}\|\comi v^{\frac{\ga}{2}+s}f\|_{L^2_{v}}^2.
	\end{equation}
 For $\mathcal {I}_{2,1}$,  we use the relation 
 \begin{align*}
 	\big[ \mu^{-1}_\ka h-(\mu^{-1}_\ka h)'\big ]^2=\big((\mu^{-1}_\th h)'\big)^2-(\mu^{-1}_\ka h)^2+2(\mu^{-1}_\ka h)\big[(\mu^{-1}_\ka h)- (\mu^{-1}_\ka h)'\big],
 \end{align*}
to split that
\beno
\mathcal {I}_{2,1} = \mathcal{I}_{2,1,1} +\mathcal{I}_{2,1,2},
\eeno
where 
\begin{align*}
	\mathcal{I}_{2,1,1} = \int_{\R^6\times\S^2}B(v-v_*,\si) \,\big| (\mu^{-\frac{1}{2}}_\ka g)_*\big|\,\big[\big((\mu^{-1}_\th h)'\big)^2-(\mu^{-1}_\ka h)^2\big]d\si dv_* dv
\end{align*}
 and
 \begin{align*}
 	\mathcal I_{2,1,2}&=2\int_{\R^6\times\S^2}B(v-v_*,\si) \,\big|(\mu^{-{1\over2}}_\ka  g)_*\big|\, (\mu^{-1}_\ka h) \big[ \mu^{-1}_\ka h-(\mu^{-1}_\ka h)'\big ] d\si dv_*dv\\
 	&=-2\big (Q( |\mu^{-{1\over2}}_\th g|,\ \mu^{-1}_\th h),\ \mu^{-1}_\th h\big)_{L^2_v},
 \end{align*}
   the last line holding because of the pre-postcollisional change of variables \eqref{prepost}. Moreover, 
   by cancellation lemma (cf. \cite[Lemma 1 and Remark 6]{MR1765272} for instance), one has 
   \begin{align*}
   	 \int_{\R^3\times\S^2}B(v-v_*,\si)  \big[\big((\mu^{-1}_\th h)'\big)^2-(\mu^{-1}_\ka h)^2\big]d\si   dv=C\int_{\mathbb R^3} \abs{v-v_*}^\gamma (\mu^{-1}_\ka h)^2 dv,
   \end{align*}
  and thus 
  \begin{multline*}
  |	\mathcal{I}_{2,1,1}|=C\int_{\R^6}|v-v_*|^\ga\, \big|(\mu^{-{1\over2}}_\ka g)_*\big|\, (\mu^{-1}_\ka h)^2 dv_*dv\\
  	\leq C \|\comi v^{|\gamma|}\mu^{-{1\over2}}_\ka g\|_{L_v^1}\|\comi v^{\frac{\ga}{2}}\mu^{-1}_\ka h\|^2_{L^2_{v}}\leq C \|\mu^{-1}_\ka g\|_{L_v^2}\|\comi v^{\frac{\ga}{2}}\mu^{-1}_\ka h\|^2_{L^2_{v}}.  \end{multline*}
 As for $\mathcal I_{2,1,2}=-2\big (Q( |\mu^{-{1\over2}}_\th g|,\ \mu^{-1}_\th h),\ \mu^{-1}_\th h\big)_{L^2_v},$  it follows from   \eqref{uppbound} that 
\begin{multline*}
\abs{\mathcal I_{2,1,2}}  
\leq  C\|\comi v^7\mu_\theta^{-{1\over2}}g\|_{L^2_v}\big (\normm{\mu_\theta^{-1}h}+\|\comi v^{\frac{\gamma}{2}+s} \mu_\theta^{-1}h\|_{L^2_{v}}\big )\normm{\mu_\theta^{-1}h}\\
\leq C	\|\mu_\theta^{-1}g\|_{L^2_v}\big (\normm{\mu_\theta^{-1}h}+\|\comi v^{\frac{\gamma}{2}+s} \mu_\theta^{-1}h\|_{L^2_{v}}\big )\normm{\mu_\theta^{-1}h}.
\end{multline*}
As a result,  we 
 combine the above estimates,  to conclude  that
\beno
 \mathcal{I}_{2,1} \leq |\mathcal{I}_{2,1,1}|+|\mathcal{I}_{2,1,2}|\leq C 	\|\mu_\theta^{-1}g\|_{L^2_v}\big (\normm{\mu_\theta^{-1}h}^2+\|\comi v^{\frac{\gamma}{2}+s} \mu_\theta^{-1}h\|_{L^2_{v}}^2\big ) 
.
\eeno
Substituting the above estimate and \eqref{j22} into \eqref{j2}  yields
\begin{equation*}
	\mathcal{I}_2 \leq \|\mu_\theta^{-1}g\|_{L^2_v}\big (\normm{\mu_\theta^{-1}h}+\|\comi v^{\frac{\gamma}{2}+s} \mu_\theta^{-1}h\|_{L^2_{v}}\big ) \|\comi v^{\frac{\gamma}{2}+s}f\|_{L^2_{v}}.
\end{equation*}
Finally we combine the above estimate and  \eqref{j1j3} with \eqref{j1j2j3} to conclude the assertion in Lemma \ref{Ighf}, completing the proof.
\end{proof}

As an immediate consequence of \eqref{uppbound} and Lemma \ref{Ighf}, we have the following  corollary.  
 
   \begin{corollary}\label{cor:upper}
		Let $\om$ be defined in  \eqref{weifun}, then there exists a constant $C>0$, such that for any regular $g,h$ and $f$, we have  
	\begin{multline}
		\label{Ig}
		\notag\big|\big( \om Q(g,h),\om f\big)_{L_v^2} \big| \\
		\leq C \|\omega  g \|_{L_v^2}\big (\normm{\omega h} +\|\comi v^{\frac{\ga}{2}+s} \omega h\|_{L^2_{v}}\big )\big (\normm{\omega f} +\|\comi v^{\frac{\ga}{2}+s} \omega f\|_{L^2_{v}}\big).
	\end{multline}
\end{corollary}

\begin{proof}
We first write
\begin{align*}
		\big(\om Q(g,h),\om f\big)_{L_v^2}=\big( Q(g,  \om h),\om f\big)_{L_v^2}+\big[\big(  Q(g,h),\om f\big)_{L_v^2}-\big(  Q(g, \om h),\om f\big)_{L_v^2}\big].
	\end{align*}
By \eqref{uppbound} as well as the fact that $\norm{\comi v^7 g}_{L_v^2}\leq C\norm{\omega g}_{L_v^2}$, it follows that 
	\begin{align*}
		|(Q(g,\omega h),\  \omega f)_{L_v^2}| \leq C \|\omega g\|_{L^2_v}\big (\normm{\omega h}+\|\comi v^{\frac{\gamma}{2}+s}\omega h\|_{L^2_{v}}\big )\normm{\omega f}.
	\end{align*}
Applying  Lemma \ref{Ighf} with $\theta(t)=a_0-t$, we have
\begin{multline*}
	 \big|\big( \om Q(g,h),\om f\big)_{L_v^2}-\big(  Q(g, \om h),\om f\big)_{L_v^2}\big| \\
	\leq C \|\omega  g \|_{L_v^2}\big (\normm{\omega h} +\|\comi v^{\frac{\ga}{2}+s}\omega h\|_{L^2_{v}}\big )\|\comi v^{\frac{\ga}{2}+s}\omega f\|_{L^2_{v}}.
\end{multline*}
As a result, combining the above estimates yields the assertion in Corollary \ref{cor:upper}.  The proof is thus completed.	
\end{proof}

\subsection{Estimates in the spatial and velocity spaces} 
The estimates in the previous subsection are   with respect to  the velocity variable $v$.  If we further  deal with these estimates therein in the whole space $\mathbb T_x^3\times\mathbb R_v^3,$  then one possible way is to work with $H_x^3 L_v^2$ so that we can take advantage of  the fact that $H_x^3$ is an algebra with respect to arithmetic product of functions.  However, to overcome the absence of the smallness assumption on  solutions,  instead of $H_x^3$ we will perform estimates in the following anisotropic norm  $\norm{\cdot}_{\mathcal H_{x}^3}$ which is equivalent to $\norm{\cdot}_{H_x^3}$ and  defined by 
\begin{equation}
	\label{mathch2}
  	\norm{f}^2_{\mathcal{H}^3_x}=\norm{f}^2_{\mathcal{H}^3_{x,N}}:=\sum_{|\al|=3}\norm{\pa^\al_xf}^2_{L^2_x}+N\norm{f}^2_{L^2_x} ,
\end{equation}
where $N\geq 1$ is a large given number, see Lemma \ref{lem:j1} below  for more detail. In the following discussion, we will write $\mathcal H_{x}^3$ for short, omitting the subscript $N$ without confusion.
Similarly, define
\begin{equation*}
\|f\|^2_{\mathcal{H}^3_xL_v^2}:=\sum_{|\al|=3}\|\pa^\al_xf\|^2_{L^2}+N\|f\|^2_{L^2} \textrm{  and  } \normm{f}^2_{\mathcal{H}^3_x}:=\sum_{|\al|=3}\normm{\pa^\al_xf}^2_{L_x^2}+N\normm{f}^2_{L_x^2},
\end{equation*}
recalling $L^2=L_{x,v}^2$ and $\normm{\cdot}$ is defined in \eqref{trinorm}.
As to be seen in the proof of Lemma \ref{lem:j1}, the large number $N$ in \eqref{mathch2} is used  to overcome
the lack of smallness of  the solution $f$, recalling we consider in this text a general solution.   One may see that if   $\norm{e^{a_0\comi v^2}f_{in}}_{H_x^3 L_v^2} \ll 1$,  then  we can choose $N=1$ in \eqref{mathch2}.

\begin{definition}\label{def:c0mt}
	Throughout the paper we define the constant $C_0$ by setting 
\begin{equation}\label{c0}
C_0:=	2\norm{\omega(0) e^{-2a_0\comi v^2}}_{H_x^3 L_v^2}+4 \norm{e^{a_0\comi v^2}f_{in}}_{H_x^3 L_v^2}  
\end{equation} 
and set
\begin{equation*}
	\mathcal M_T(C_0):=\big\{h\in L^\infty([0,T]; H_x^3 L_v^2); \quad 	\sup_{t\leq T}\norm{\omega h(t)}_{H_x^3 L_v^2}  \leq C_0
\big\}.
\end{equation*}
Moreover, we define $\mathcal A_T=\mathcal A_T(m_0,M_0,E_0,H_0)$ by 
\begin{equation*}
	\mathcal A_T:=\big\{h\in L^\infty([0,T]; H_x^3 L_v^2); \quad 	h\geq 0\  \textrm{ and satisfies condition}\   \eqref{aat}
\big\}.
\end{equation*}	
\end{definition}

\begin{rmk}
	In the rest part of the paper, to lighten the notation we will denote by  $C$   some generic constants,   depending only on  the Sobolev embedding constants and   the constants $C_0, C_1, C_2$ in  \eqref{c0} and \eqref{lower}.  Moreover,   we use $C_\eps$  to stand for some generic constants depending on $\eps $ additionally.  
\end{rmk}
 
\begin{lem} \label{lem:j1}
Recall $\mathcal A_T$ and $\mathcal M_T(C_0)$ are given  in Definition \ref{def:c0mt} and $\omega$ is the weight function defined in \eqref{weifun}.   Suppose $f\in \mathcal A_T\cap \mathcal M_T(C_0)$  for some given $T>0$.   
 Then 
 there eixsts an integer $N\geq 1$ and a constant $\bar C>0$,  both  depending only on the constants $C_0,C_1$ and $C_2$ in \eqref{c0} and \eqref{lower},   such that  for any regular function $g$ we have 
\begin{equation*}
 \big (\omega Q(f, g),\ \omega  g\big )_{\mathcal{H}^3_xL^2_v}  \leq -\frac{C_1}{2}  \normm{\omega   g}_{\mathcal H_x^3}^2 +\frac12\|\comi v \omega g\|_{\mathcal H_x^3 L^2_{v}}^2 +\bar C \|  \omega g\|_{\mathcal H_x^3 L^2_{v}}^2,
   \end{equation*}
   recalling  $\mathcal H_x^3=\mathcal H_{x,N}^3$ is defined in \eqref{mathch2}.
  \end{lem}

\begin{rmk}
	In the following discussion,   we will fix the integer  $N$ constructed in Lemma \ref{lem:j1}. Without confusion we usually write $\mathcal H_x^3$ instead of $\mathcal H_{x,N}^3$  for short, omitting the subscript $N.$
\end{rmk}

\begin{proof}[Proof of Lemma \ref{lem:j1}]
	In view of    definition \eqref{mathch2}  of $\mathcal{H}^3_x$, we write, recalling $L^2=L_{x,v}^2$, 
\begin{equation}\label{s1s2s3}
\begin{aligned}
 \big (\omega Q(f, g),\ \omega  g\big )_{\mathcal{H}^3_xL^2_v} &=N\big (\omega   Q(f, g),\ \omega  g\big )_{L^2}+\sum_{|\al|=3}(\om Q(f,\ \pa^\al_xg),\ \om \pa^\al_x g)_{L^2}\\
&\qquad +\sum_{\stackrel{\beta\leq\alpha}{ |\al|=3, |\beta| \geq 1}} {\alpha\choose\beta} \big (\om Q(\pa^{\beta}_xf, \ \pa^{\alpha-\beta}_xg\big ),\ \om \pa^\al_x g)_{L^2}\\
&:=S_1+S_2+S_{3}.
\end{aligned}
\end{equation}

\noindent\underline{\it Upper bound of $S_1$ and $S_2$}. 
For $S_{1}$,  we write
\begin{equation}\label{s1}
S_{1}=N\big (  Q(f, \omega g),\  \omega  g\big )_{L^2}
+N\Big[ \big (\omega   Q(f, g),\ \omega  g\big )_{L^2}-\big (   Q(f, \omega g),\ \omega  g\big )_{L^2}\Big].
\end{equation}
It follows from \eqref{lower} as well as Definition \ref{c0} of $C_0$  that
\ben\label{comest}
	  \big (  Q(f, \omega g),\ \omega  g\big )_{L^2}\notag&\leq& -C_1 \normm{\om g}^2_{L_x^2} +C_2 \big(1+ \|\comi v^6 f\|_{L_x^\infty L^2_v}\big) \|\comi v^{\frac{\ga}{2}+s}\om g\|^2_{L^2}\\
	  &\leq& -C_1 \normm{\om g}^2_{L_x^2} +\eps C_0 \|\comi v \om g\|^2_{L^2}+C_\eps C_0  \| \om g\|^2_{L^2},	  
\een
the last inequality using  the fact $\f \ga 2+s<1$ and  
  the interpolation inequality that
  \begin{equation}\label{intero}
 \forall\ \eps>0,  \  \ 	\norm{\comi v^{\frac{\gamma}{2}+s}h}_{L^2}\leq \eps \norm{\comi v h}_{L^2}+C_{\eps} \norm{\comi v^{\frac{\gamma}{2}}h}_{L^2}\leq \eps \norm{\comi v h}_{L^2}+C_{\eps} \norm{h}_{L^2}. 
  \end{equation}
On the other hand, using Lemma  \ref{Ighf} with $\mu_\theta^{-1}=\omega$ and the   Sobolev  inequality that $\norm{\cdot}_{L_x^\infty}\leq C\norm{\cdot}_{H_x^3}$,    we have that, for any $\eps>0,$
\begin{multline*}
	\big|  \big (\omega   Q(f, g),\ \omega  g\big )_{L^2}-\big (   Q(f, \omega g),\ \omega  g\big )_{L^2}\big|\\
	  \leq C\norm{\omega f}_{H_x^3L_v^2}\big(\normm{\omega g}_{L_x^2}+\norm{\comi v^{\frac{\gamma}{2}+s}\omega g}_{L^2} \big)\norm{\comi v^{\frac{\gamma}{2}+s}\omega g}_{L^2}\\	\leq \eps C_0 \big( \normm{\omega g}_{L_x^2}^2+\norm{\comi v \omega g}_{L^2}^2\big)+C_\eps C_0 \norm{ \omega g}_{L^2}^2,
\end{multline*}
the last line holding because of   the interpolation inequality \eqref{intero}.  
  Combing the above estimate and \eqref{comest} with \eqref{s1},  we conclude that, for any
$\eps>0,$
\begin{equation*}
	S_1 \leq   -C_1 N\normm{\om g}^2_{L_x^2} +\eps C_0 \big(N \normm{\omega g}_{L_x^2}^2+N\norm{\comi v \omega g}_{L^2}^2\big) 
	+C_\eps  C_0 N\norm{ \omega g}_{L^2}^2.
\end{equation*}
Similarly, 
\begin{multline*}
	S_2 \leq    -C_1\sum_{|\al|=3}\normm{\om \partial_x^\alpha g}^2_{L_x^2}\\
	+ \eps C_0\sum_{|\al|=3}\big( \normm{\omega \partial_x^\alpha g}_{L_x^2}^2+\norm{\comi v \omega \partial_x^\alpha g}_{L^2}^2\big) 
	 +C_\eps  C_0 \sum_{|\al|=3}   \norm{ \omega \partial_x^\alpha g}_{L^2}^2.
\end{multline*}
Combining the two estimates above and recalling definition \eqref{mathch2} of $\mathcal H_x^3$, we have that
\begin{equation}\label{s1s2}
	S_1+S_2 \leq -C_1\normm{\omega g}^2_{\mathcal H_x^3} 
	+\eps C_0 \big( \normm{\omega  g}_{\mathcal H_x^3}^2+\norm{\comi v \omega   g}_{\mathcal H_{x}^3L_v^2}^2\big) 
	 +C_\eps C_0\norm{ \omega g}_{\mathcal H_{x}^3 L_v^2}^2.
\end{equation}

\noindent\underline{\it Upper bound of $S_3$}. 
It remains to deal with $S_3$.  We use Corollary
  \ref{cor:upper} as well as the fact that
  \begin{equation*}
  	\begin{aligned}
  	\norm{fg}_{L_x^2}\leq   C\norm{f }_{L_x^2}	\norm{ g}_{H_x^{\frac74}}, 
  	\end{aligned}
  \end{equation*}
to obtain that  
 \begin{equation*}
 \begin{aligned}
 	S_3&=\sum_{\stackrel{\beta\leq\alpha}{ |\al|=3, |\beta| \geq 1}} {\alpha\choose\beta} \big (\om Q(\pa^{\beta}_xf, \ \pa^{\alpha-\beta}_xg\big ),\ \om \pa^\al_x g)_{L^2}\\
 	& \leq C  \|\omega f \|_{H_x^{3} L_v^2}  \Big (\normm{\omega    g}_{H_x^{\frac{11}{4}}}+\|\comi v^{\frac{\gamma}{2}+s} \omega  g\|_{H_x^{\frac{11}{4}} L^2_{v}}\Big ) 
 	  \big (\normm{\omega   g}_{H_x^3} +\|\comi v^{\frac{\gamma}{2}+s} \omega g\|_{H_x^3 L^2_{v}}\big).
 \end{aligned}
 \end{equation*}
 This,  with  the fact $f\in\mathcal M_T(C_0)$, the interpolation \eqref{intero}  and
 \begin{equation*}
 \forall\  \eps>0,\quad 	\norm{ h}_{H_x^{\frac{11}{4}}}\leq \eps  \norm{ h}_{H_x^{3}}+C_\eps  \norm{ h}_{L_x^{2}}, 
 \end{equation*}
 implies that
 \begin{equation}\label{ss3}
 	S_{3}\leq \eps  C_0 \big (\normm{\omega   g}_{H_x^3}^2+\|\comi v \omega g\|_{H_x^3 L^2_{v}}^2\big)+C_\eps  C_0 \big (\normm{\omega   g}_{L_x^2}^2+\|  \omega g\|_{H_x^3 L^2_{v}}^2\big).
 \end{equation}
   Now we combine the upper bound of  $S_3$ above and   estimate \eqref{s1s2} on $S_1+S_2$ with \eqref{s1s2s3}, to conclude
 \begin{multline}\label{Nlarge}
 \big (\omega Q(f, g),\ \omega  g\big )_{\mathcal{H}^3_xL^2_v}  \leq -C_1 \normm{\omega   g}_{\mathcal H_x^3}^2 +\eps C_0\|\comi v \omega g\|_{\mathcal H_x^3 L^2_{v}}^2 \\
 + 	\Big(\eps C_0+\frac{C_\eps C_0}{N}\Big)  \normm{\omega   g}_{\mathcal H_x^3}^2 +C_\eps C_0  \|  \omega g\|_{\mathcal H_x^3 L^2_{v}}^2.
   \end{multline}
 Now we choose $\eps=\min\{\frac{C_1}{4C_0},\f12\}$ and 
$
 	N=4C_\eps C_0/C_1
$ in the above estimate, this gives
 the assertion in Lemma \ref{lem:j1}. The proof is thus completed. 
\end{proof}

Note in the proof of Lemma \ref{lem:j1}  we only ask $N$ large in estimate \eqref{Nlarge} so that       the term$\normm{\omega   g}_{L_x^2}$ in \eqref{ss3} can be absorbed. So if we  repeat the  the proof  of Lemma \ref{lem:j1} and replace $\mathcal H_x^3 L_v^2$ therein by $H_x^3 L_v^2$(or let $N=1$), we have the following lemma.

\begin{lem}\label{lem:h3}
Under the hypothesis of Lemma \ref{lem:j1}, we can find a constant $\bar C,$   depending only on the constants $C_0,C_1$ and $C_2$ in \eqref{c0}  and \eqref{lower},   such that  for any regular function $g$ we have 
\begin{equation*}
 \big (\omega Q(f, g),\ \omega  g\big )_{H^3_xL^2_v}  \leq -\frac{C_1}{2}  \normm{\omega   g}_{  H_x^3}^2 +\frac12\|\comi v \omega g\|_{ H_x^3 L^2_{v}}^2+\bar C\normm{\omega   g}_{L_x^2}^2 +\bar  C\|  \omega g\|_{  H_x^3 L^2_{v}}^2.
   \end{equation*}
\end{lem}

 The following lemma is concerned with the upper bound of the trilinear term.

\begin{lem}\label{lem:j2}
For any $\eps>0,$ we have
\begin{multline*}
	 \big| \big (\omega Q(g,  f),\ \omega  h\big )_{\mathcal{H}^3_xL^2_v}\big|  \leq \eps  \big(\normm{\omega h}_{\mathcal H_x^3}^2+\norm{\comi v \omega h}_{\mathcal H_x^3L_v^2}^2\big) \\ +C\eps^{-1}  \norm{\omega g}_{\mathcal H_x^3L_v^2}^2  \big(\normm{\omega f}_{H_x^{3}}^2+\norm{\comi v\omega f}_{H_x^{3} L_v^2}^2\big).   
\end{multline*}
\end{lem}

\begin{proof}
The argument is similar to that in  Lemma \ref{lem:j1}.  In fact, as in \eqref{s1s2s3},  we split
\begin{equation*}
	\begin{aligned}
&\big| \big (\omega Q(g,  f),\ \omega  h \big )_{\mathcal{H}^3_xL^2_v} \big| \\
&\qquad\leq N \big| \big (\omega   Q(g,  f),\ \omega  h\big )_{L^2} \big| +\sum_{\stackrel{ |\al|=3 }{\beta\leq\alpha}}\binom{\alpha}{\beta}\big|\big (\om Q(\pa^{\beta}_xg, \ \pa^{\alpha-\beta}_x f\big ),\ \om \pa^\al_x h)_{L^2}\big|.
\end{aligned}
\end{equation*}
Using Corollary \ref{cor:upper} as well as  the fact that $\norm{\cdot}_{L_x^\infty}\leq C\norm{\cdot}_{H_x^3}$,  we have 
\begin{align*}
	& N \big| \big (\omega   Q(g,  f),\ \omega  h\big )_{L^2} \big| \\
	&\qquad  \leq   C N\norm{\omega g}_{L^2} \big(\normm{\omega f}_{H_x^{3}}+\norm{\comi v^{\frac{\gamma}{2}+s}\omega f}_{H_x^{3} L_v^2}\big) \big( \normm{\omega h}_{L_x^2} +\norm{\comi v^{\frac{\gamma}{2}+s}\omega h}_{L^2}\big)
	 \end{align*}
and  
\begin{align*}
	&\sum_{\stackrel{\beta\leq\alpha}{ |\al|=3 }}\binom{\alpha}{\beta}\big|\big (\om Q(\pa^{\beta}_xg, \ \pa^{\alpha-\beta}_x f\big ),\ \om \pa^\al_x h)_{L^2}\big| \\
	&\quad \leq   C   \norm{\omega   g}_{H_x^3L_v^2} \big(\normm{\omega f}_{H_x^{3}}+\norm{\comi v^{\frac{\gamma}{2}+s}\omega f}_{H_x^{3} L_v^2}\big) 	 \big(  \normm{\omega  h}_{H_x^3}+\norm{\comi v^{\frac{\gamma}{2}+s}\omega  h}_{H_x^3L_v^2}\big).
	 \end{align*}
As a result,   we combine the above  estimates   to conclude  the assertion in Lemma \ref{lem:j2}. The proof is thus completed.  
\end{proof}

 \begin{rmk}
 	Let $N=1$ in the the proof of Lemma \ref{lem:j2}, we have that, for any $\eps>0,$  
\begin{multline*}
	 \big| \big (\omega Q(g,  f),\ \omega  h\big )_{H^3_xL^2_v}\big|  \leq \eps  \big(\normm{\omega h}_{  H_x^3}^2+\norm{\comi v \omega h}_{ H_x^3L_v^2}^2\big) \\ +C\eps^{-1}  \norm{\omega g}_{ H_x^3L_v^2}^2  \big(\normm{\omega f}_{H_x^{3}}^2+\norm{\comi v\omega f}_{H_x^{3} L_v^2}^2\big).   
\end{multline*}
 	
 \end{rmk}

\section{\emph{A priori} estimate}

We will adopt the iteration strategy to construct solutions to the Boltzmann equation \eqref{Bolt}, and a key part is to derive the uniform estimate for a regularized  version of equation \eqref{Bolt} and iterative equations.  To clarify the argument, we first prove \emph{a priori} estimate for a linear Boltzmann equation,  seeing   Theorem \ref{thm:apri} below with the proof  postponed to Sections \ref{sec:analy} and \ref{sec:gev}.

Before stating the \emph{a priori} estimate, we first introduce some notations to be used later. 
For a given constant $\rho>0$   we define
\begin{equation}\label{lrhok}
	L_{\rho, k}=\left\{
	\begin{aligned}
		& 1, \    \textrm{ if } \ k=0, \\
		 & \frac{(k+1)^3}{ \rho^{ k-1 }(k!)^{\max \{(2\tau)^{-1}, 1\} }} ,\    \textrm{ if } \  k \geq 1, 
	\end{aligned}
	\right.
\end{equation}
 with $\tau$  defined in \eqref{tau}. For any given norm $\norm{\cdot}$, we define
\begin{equation}\label{pkh}
		\norm{ \omega \vec{ H}_{j}^\alpha h}: =
		 		   \|\omega H_{\delta_1, j} ^{\alpha_1} H_{\delta_2, j}^{\alpha_2} h \ \| \  \textrm{ for }\ \alpha=(\alpha_1,\alpha_2)\in\mathbb Z_+^2,
\end{equation}
where       $H_{\delta_i,j}$ are defined by \eqref{hdel12},  that is,
\begin{equation*}
	H_{\delta_i, j}=\frac{1}{\delta_i+1} t^{\delta_i+1}\partial_{x_j}+t^{\delta_i}\partial_{v_j},  \quad 1\leq j\leq 3,\quad i=1,2,
\end{equation*}
with $\delta_1,\delta_2$  given in \eqref{de1de2}. Finally, 
denote
\begin{equation}\label{dk}
	 \norm{\omega {\bm D}^k h }_{ \mathcal H_x^3  L_v^2}=
	 \left\{
	 \begin{aligned}
	 & \sum_{1\leq j\leq 3} \sum_{1\leq i\leq 2}\norm{\omega H_{\delta_i,j}^k h }_{ \mathcal H_x^3  L_v^2},\  \ \textrm{if }\ 	  {\gamma \over 2}+2s\geq 1,\\
	 &\sum_{1\leq j\leq 3} \sup_{|\alpha|=k} \norm{\omega \vec H_{j}^{\alpha} h }_{ \mathcal H_x^3  L_v^2}, \ \  \textrm{if }\ 	 {\gamma \over 2}+2s< 1.
	 \end{aligned}
\right.
\end{equation}

 \begin{definition}\label{assum} Using the notations in \eqref{lrhok} and \eqref{pkh},    we introduce   two Banach spaces $ X_{\rho, T}$ and   $Y_{\rho,T} $, equipped  with the   norms  $ \abs{h}_{X_{\rho,T}}$ and $ \abs{h}_{Y_{\rho,T}}$ defined as below,  respectively.   If  ${\gamma \over 2}+2s\geq 1$  we define 
  \begin{equation}\label{xyrt}
  \left\{
 \begin{aligned}
 	& \abs{h}_{X_{\rho,T}}:   = \sum_{1\leq j\leq3}\sum_{1\leq i\leq 2}\,  \sup_{k\geq 0}\Big( L_{\rho,k}  \sup_{t\leq T} \norm{\omega H_{\delta_i,j}^k h }_{ \mathcal H_x^3  L_v^2} \Big), \\
 	  &\abs{h}_{Y_{\rho,T}}  : =  \sum_{1\leq j\leq3}\sum_{1\leq i\leq 2}\, \sup_{k\geq 0}\bigg[ L_{\rho,k}  \Big(   \int_0^T \big( \normm{\omega H_{\delta_i,j}^k h}_{ \mathcal H_x^3   }^2  +\norm{ \comi{v}\omega H_{\delta_i,j}^k  h}_{ \mathcal H_x^3  L_v^2} ^2 \big)dt \Big)^{1\over2} \bigg].  
 	  	\end{aligned}
 	  	\right.
 \end{equation} 
 If  ${\gamma \over 2}+2s< 1$  we define 
  \begin{equation*}
  \left\{
  \begin{aligned}
  &	\abs{h}_{X_{\rho, T}}= \sum_{ 1\leq j\leq 3}    \sup_{k\geq 0}	\Big(L_{\rho,k}\sup_{\abs\alpha=k}\, \sup_{t\leq T} \norm{\omega \vec{ H}_{j}^\alpha h(t) }_{ \mathcal H_x^3  L_v^2} \Big),\\
  & \abs{h}_{Y_{\rho, T}}= \sum_{ 1\leq j\leq 3}    \sup_{k\geq0}	\bigg[L_{\rho,k}\sup_{\abs\alpha=k}\bigg(\int_0^T\big(\normm{ \omega \vec{ H}_{j}^\alpha h }_{ \mathcal H_x^3}^2+ \norm{\comi v\omega \vec{ H}_{j}^\alpha h }_{ \mathcal H_x^3  L_v^2}^2\big)dt\bigg)^{1\over2} \bigg].
  \end{aligned}
\right.
 \end{equation*}
 Recall $\omega$ is given in \eqref{weifun}, and  $\mathcal H_x^3$ and $\normm{\cdot}$ are  defined by \eqref{mathch2} and  \eqref{trinorm}, respectively. 
 \end{definition}

 \begin{definition}\label{defm}
Given constants $T, C_*, \rho>0,$  we define  
\begin{equation*}
     \mathcal N_{T}(\rho, C_*)=\big\{f\in L^\infty([0,T]; \mathcal H_x^3 L_v^2);  \quad   \abs{f}_{X_{\rho,T}}+\abs{f}_{Y_{\rho,T}} \leq C_*\big\},
\end{equation*}
where $\abs{f}_{X_{\rho,T}}$ and $\abs{f}_{Y_{\rho,T}}$ are given in Definition \ref{assum}. 
\end{definition}

\begin{thm}\label{thm:apri}
Let $C_0$ be the constant given in \eqref{c0}, and
let $\mathcal A_T, \mathcal M_T(C_0)$  be given in Definition \ref{def:c0mt}, and  let $\mathcal N_{T}(\rho, C_*)$ be  given  above. 
Then  there exists a time $T_*$ and  three constants $\rho, \tilde C, C_*>0,$ all depending only on $C_0$ in \eqref{c0} and the numbers $C_1,C_2$ in \eqref{lower}, such that for any given  $T\leq T_*$,  if 
\begin{equation*}
	f\in\mathcal A_T\cap \mathcal M_T(C_0)\cap \mathcal N_{T}(\rho, C_*)
\end{equation*}
 and if     
   $g$ solves the linear Boltzmann equation
\begin{equation}\label{flineq}
	\partial_tg+v\cdot\partial_xg=Q(f,g),\quad g|_{t=0}=f_{in},
\end{equation}   
 satisfying the condition 
 \begin{equation}\label{apong}
	\left\{
	\begin{aligned}
		& \abs{g}_{X_{\rho, T}}+\abs{g}_{Y_{\rho, T}}
<+\infty,\\
& 	\lim_{t\rightarrow 0} \norm{\omega {\bm D}^k g(t) }_{ \mathcal H_x^3  L_v^2} =0\  \textrm{ for any } \ k\geq 1,
	\end{aligned}
	\right. 
\end{equation}
then it holds that 
 \begin{equation}\label{3.7}
 	\abs{g}_{X_{\rho, T}}+\abs{g}_{Y_{\rho, T}}
\leq \tilde C  C_* \bigg[\int_0^T \big( \normm{\omega  g}_{ \mathcal H_x^3   }^2  +\norm{ \comi{v}\omega g}_{ \mathcal H_x^3  L_v^2} ^2 \big)dt\bigg]^{1\over2}+\frac{C_*}{2}.
 \end{equation} 
Moreover, there exists a constant $\Theta$ depending only on $C_0,C_1$ and $C_2$, such that  if 
 \begin{equation*}
 	 \bigg[\int_0^T \big( \normm{\omega  g}_{ \mathcal H_x^3   }^2  +\norm{ \comi{v}\omega g}_{ \mathcal H_x^3  L_v^2} ^2 \big)dt\bigg]^{1\over2} \leq \Theta,
 \end{equation*}
 then
 \begin{equation*}
 	g\in\mathcal A_T\cap \mathcal M_{T}(C_0)\cap \mathcal N_T(\rho, C_*).
 \end{equation*}
 Here we used the notation in  \eqref{pkh} and \eqref{dk}.
 	 \end{thm}
 
 We will proceed through Sections \ref{sec:analy} and \ref{sec:gev} to prove Theorem \ref{thm:apri}. In Section \ref{sec:analy} we consider the case when ${\ga\over 2}+2s\geq 1$, and the case of ${\ga\over 2}+2s<1$ is investigated in Section \ref{sec:gev}.  In the following discussion, we use   $C$  to denote   some generic constants,   depending only on  the Sobolev embedding constants and   the constants $C_0, C_1, C_2$ in  \eqref{c0} and \eqref{lower}.  Moreover,   we use $C_\eps$  to stand for some generic constants depending on $\eps $ additionally.

\section{Proof of the \emph{a priori} estimate: Analytic  regularity}\label{sec:analy}

This and the next sections are devoted to proving Theorem \ref{thm:apri}.   
 To clarify the argument we  consider in this section the case when ${\ga\over 2}+2s\geq 1$, recalling $\gamma$ and $s$ are the numbers given in \eqref{gamma} and \eqref{s}, respectively.  Then  we have  that 
\begin{align*}
		 \max\Big \{\frac{1}{2\tau},1\Big \}= \max\Big \{\f{2-\gamma}{4s},1\Big \}=1
\end{align*}
in  the assertion \eqref{alpha1}.  So this section is  concerned with the analytic regularization effect. The remaining Gevrey smoothing effect is postponed to Section \ref{sec:gev}. In view of \eqref{lrhok}, when  ${\gamma \over 2}+2s\geq 1$  we have 
 \begin{equation}\label{lrma}
 	L_{\rho, k}=\left\{
	\begin{aligned}
		& 1, \    \textrm{ if } \ k=0, \\
		 & \frac{(k+1)^3}{ \rho^{ k-1 }k!} ,\    \textrm{ if } \  k \geq 1. 
	\end{aligned}
	\right.
 \end{equation}
 Let $H_\delta$ be defined by \eqref{vecM}. Then in the case of ${\gamma \over 2}+2s\geq 1$,   Theorem \ref{thm:apri} will follow  from the quantitative estimate  on  the  directional derivations  with respect to $H_\delta.$ 
 
The following proposition is devoted to proving $g\in \mathcal A_T\cap  \mathcal M_{T}(C_0)$ for some small $T$ and $\Theta$.

 \begin{prop}
 \label{lem:mc0}
 Let $C_0$ be the   constant  given in \eqref{c0}. 
Then  there exists a time $T_*$ and  a constant  $ \Theta$, both depending only on $C_0$  in \eqref{c0} and the numbers $C_1$ and $C_2$ in \eqref{lower}, such that for any given  $T\leq T_*$,  if 
\begin{equation*}
	f\in\mathcal A_T\cap \mathcal M_T(C_0), 
\end{equation*}
 then for any solution $g$ to the linear Boltzmann equation
 \eqref{flineq} satisfying  the condition that
 \begin{equation}\label{+apong}
|g|_{X_{\rho, T}}+|g|_{Y_{\rho,T}}<+\infty\  \textrm{ and }\   \Big( \int_0^T  \normm{\omega   g}_{H_x^3}^2dt\Big)^{1\over2} dt \leq \Theta  
\end{equation}
for some $\rho>0$, 
 we have 	
 \begin{equation*}
 	g\in \mathcal A_T\cap  \mathcal M_{T}(C_0). 
 \end{equation*}
 \end{prop}
 
 \begin{proof} Let $T_*$ to be determined later and we first show that 
 \beno
 g\in \mathcal M_T(C_0)\  \textrm{ for }\  T\leq T_*.
 \eeno
  If $f\in \mathcal A_T\cap \mathcal M_{T}(C_0)$ with $T\leq T_*$, then it follows from Lemma \ref{lem:h3} that, for any $t\leq T,$
 \begin{equation*}
 \big (\omega Q(f, g),\ \omega  g\big )_{H^3_xL^2_v}  \leq -\frac{C_1}{2}  \normm{\omega   g}_{  H_x^3}^2 +\frac12\|\comi v \omega g\|_{ H_x^3 L^2_{v}}^2+\bar C\normm{\omega   g}_{L_x^2}^2 +\bar  C\|  \omega g\|_{  H_x^3 L^2_{v}}^2,
   \end{equation*}
   where $\bar C$ is a constant depending only on $C_0$ and $C_1, C_2$ in \eqref{lower}. 
 As a result, we take the $H_x^3 L_v^2$ inner product with $\omega g$ on both sides of the equation
 \begin{equation*}
 	\big(\partial_t+v\cdot\partial_x+\comi v^2\big)\omega g=\omega Q(f,g), \quad g|_{t=0}=f_{in},
 \end{equation*} 
 to conclude that, recalling $0<C_1\leq 1,$
  \begin{equation*}
	\frac12\frac{d}{dt}\norm{\omega g} _{H_x^3L_v^2}^2+\frac{C_1}{2}\Big(\normm{\omega g} _{H_x^3}^2+\norm{\comi v \omega g} _{H_x^3L_v^2}^2\Big) \leq  \bar C\normm{\omega   g}_{L_x^2}^2 +\bar  C\|  \omega g\|_{  H_x^3 L^2_{v}}^2.
\end{equation*} 
 Thus there exists a time $T_*>0$   depending only on $\bar C$,   such that  
  \begin{equation}\label{wg}
 \begin{aligned}
 	&\sup_{0<t\leq T_*}\norm{\omega g} _{H_x^3L_v^2}+ \Big(C_1 \int_0^{T_*}\big(\normm{\omega g} _{H_x^3}^2+\norm{\comi v \omega g} _{H_x^3L_v^2}^2\big)dt\Big)^{1\over2}\\
 	&\qquad\qquad\qquad \qquad\qquad \leq 2\norm{e^{a_0\comi v^2} f_{in}} _{H_x^3L_v^2}+2 \bar C \Big(\int_0^{T_*}\normm{\omega   g}_{L_x^2}^2dt\Big)^{1\over2}.
 \end{aligned}
 \end{equation}
 Moreover, define
 \begin{equation*}
 	\Theta:=\frac{C_0}{4\bar C}
 \end{equation*}
 with $\bar C$ the constant in \eqref{wg}. It follows from \eqref{wg} as well as \eqref{c0}  that if 
 \begin{equation*}
   \Big(\int_0^{T_*}\normm{\omega   g}_{L_x^2}^2dt\Big)^{1\over2}\leq \Theta,
  \end{equation*}
  then
  \begin{equation}\label{inistep}
   	\sup_{0\leq t\leq T_*}\norm{\omega g} _{H_x^3L_v^2} + \Big(C_1 \int_0^{T_*}\big(\normm{\omega g} _{H_x^3}^2+\norm{\comi v \omega g} _{H_x^3L_v^2}^2\big)dt\Big)^{1\over2} \leq C_0.
 \end{equation}
  This yields $g\in\mathcal M_T(C_0)$ with $T\leq T_*$. 
  
  Next we will show that
 \begin{equation*}
 	g\in \mathcal A_T\  \textrm{ for }\ T\leq T_*.
 \end{equation*}
  We begin with the non-negativity of $g$.  Write $g=g_{+}+g_{-}$ with 
  \begin{equation*}
  	g_{+}:=\max\{g,0\}\ \textrm{ and }\  g_{-}:=-\max\{-g,0\}.
  \end{equation*}
  Since $g$ is a smooth function at positive times, then for any $\alpha,\beta\in\mathbb Z_+^3,$
  \begin{equation}\label{mp}
  	(\pa_tg_+) g_-=0 \  \textrm{ and }\  (\pa^\al_x\pa^\be_v  g_+) g_-=0, \   a.e.  \textrm{ on }    \mathbb T_x^3\times \mathbb R_v^3.
  \end{equation}
  As a result, 
  \begin{equation}\label{gminus}
  		\big(\partial_t g+v\cdot\partial_x g,\   g_{-}\big)_{L^2} =	\big(\partial_t (g_{+}+g_{-})+v\cdot\partial_x (g_{+}+g_{-}),\   g_{-}\big)_{L^2} 
  		 =\frac12\norm{g_{-}}_{L^2}^2.
  \end{equation}
  On one hand,  we write 
  \begin{align*}
  		\big(Q(f,\ g),\   g_{-}\big)_{L^2}= \big(Q(f,\ g_+),\   g_{-}\big)_{L^2}+ \big(Q(f,\ g_{-}),\   g_{-}\big)_{L^2}.
  \end{align*}
  For the first term on the right-hand side, 
 \begin{align*}
 \big(Q(f,\ g_+),\   g_{-}\big)_{L^2}&=\int_{\T^3_x}\int_{\R^6\times\S^2}B(|v-v_*|,\sigma) (f_*'(g_+)'-f_*g_+)g_{-}d\si dv_*dvdx\\
& \leq \int_{\T^3_x}\int_{\R^6\times\S^2}B(|v-v_*|,\sigma)  f_*'(g_+)' g_{-}d\si dv_*dvdx\leq 0,
 \end{align*} 
the first inequality using \eqref{mp} and the last inequality following from the fact that $B,  f_*', (g_+)'  \geq 0 $ and $g_{-}\leq 0.$ On the other hand, 
using  the pre-postcollisional change of variables (cf. \eqref{prepost}  or \cite[Subsection 4.5]{MR1942465} for instance), we write
 \begin{align*}
 \big(Q(f,\ g_-),\   g_{-}\big)_{L^2}&=\int_{\T^3_x}\int_{\R^6\times\S^2}B(|v-v_*|,\sigma) (f_*'(g_-)'-f_*g_-)g_{-}d\si dv_*dvdx\\
 &=\int_{\T^3_x}\int_{\R^6\times\S^2}B(|v-v_*|,\sigma)  f_*\big((g_-)'g_-- g_- g_{-}\big)d\si dv_*dvdx.
 \end{align*} 
This,  with the fact that
\begin{align*}
	f_*\big((g_-)'g_{-}- g_- g_{-}\big)&=\frac12 f_*\big[\big((g_{-})'\big)^2- (g_{-})^2\big]-\frac12 f_*\big ((g_-)'-g_-\big)^2\\
	&\leq \frac12 f_*\big[\big((g_{-})'\big)^2- (g_{-})^2\big],
\end{align*} 
 due to the fact $f\geq 0,$ yields 
 \begin{equation*}
 \begin{aligned}
 	 \big(Q(f, g_-),   g_{-}\big)_{L^2}&\leq \frac12 \int_{\T^3_x}\int_{\R^6\times\S^2}B(|v-v_*|,\sigma)   f_* \big[\big((g_{-})'\big)^2- (g_{-})^2\big]d\si dv_*dvdx\\
 	&\leq C\int_{\T^3_x}\int_{\R^6}|v-v_*|^\ga f_* (g_-)^2dv_*dvdx \leq C \|g_-\|^2_{L^2},
 \end{aligned}
 	 \end{equation*}
where in the last line we used cancellation Lemma   (cf. \cite[Lemma 1 and Remark 6]{MR1765272} for instance) and   the last inequality holds because of the fact that 
$
 	\norm{\comi v^{\abs\gamma}f}_{L_x^2 L_v^1}\leq C 	\norm{\omega f}_{L^2}.
$  Combining the above estimates   we have
\begin{equation*}
	 \big(Q(f, g),   g_{-}\big)_{L^2}= \big(Q(f, g_+),   g_{-}\big)_{L^2}+ \big(Q(f, g_-),   g_{-}\big)_{L^2}\leq  C \|g_-\|^2_{L^2},
\end{equation*}
which with \eqref{gminus} yields
  that
 \begin{equation*}
 \begin{aligned}
 	\frac12  \frac{d}{dt}  \|g_-\|^2_{L^2} =\big(\partial_t g+v\cdot\partial_x g,\   g_{-}\big)_{L^2}=\big(Q(f, g),  \  g_{-}\big)_{L^2}\leq C \|g_-\|^2_{L^2}.
 \end{aligned}
\end{equation*}
Thus,  
\begin{equation*}
\forall\ 0\leq t\leq T,\quad	 \|g_-(t)\|^2_{L^2}\leq C \|(f_{in})_-\|^2_{L^2}\leq 0,
\end{equation*} 
the last inequality holding because of the fact that $f_{in}\geq 0.$  Thus
     $g_- (t)=0$ for $t\in[0,T]$, and hence $g\geq0$.  
  
  Secondly, we will prove the bounds of mass, energy and entropy of $g$. For the lower and upper bounds of mass, integrating \eqref{flineq} over $\mathbb R_v^3$ yields that 
\beno
\pa_t \int_{\R^3_v}g(t,x,v)dv+\pa_x\cdot\int_{\R^3_v} vg(t,x,v)dv=0.
\eeno
So we have 
\begin{align*}
&\Big|\int_{\R^3_v}g(t,x,v)dv-\int_{\R^3_v}f_{in}(x,v)dv\Big|\leq \int_0^t \Big|\pa_x\cdot\int_{\R^3_v} vg(r,x,v)dv\Big| dr  \\
&\leq C\int_0^t  \|\om g(r)\|_{ H^3_xL^2_v}dr\leq CC_0t ,\end{align*}
the last inequality using \eqref{inistep}. 
Thus for $T_*\leq (CC_0)^{-1} \min\big \{\frac{m_0}{2},M_0\big \},$ we have that, for any  $ (t,x)\in [0,T]\times \T^3_x\subset [0,T_*]\times \T^3_x$,
\begin{equation}\label{lowerbfepsilon}
\f{m_0}2\leq \int_{\R^3_v}f_{in}(x,v)dv-\f{m_0}2\leq \int_{\R^3_v}g(t,x,v)dv\leq \int_{\R^3_v}f_{in}(x,v)dv+M_0\leq 2M_0.
\end{equation}
For the upper bound of energy, we multiply \eqref{flineq} by $\abs v^2$ and then integrate over $\mathbb R_v^3$ to get  that 
\beno
\pa_t \int_{\R^3_v} g(t,x,v)|v|^2dv+\pa_x\cdot\int_{\R^3_v} vg(t,x,v)|v|^2dv=\int_{\R^3_v}Q(f,g)|v|^2dv.
\eeno
 From    \cite[Theorem 1.1]{MR3942041}  we have that, for  any  $ (t,x)\in [0,T]\times \T^3_x$,
\beno
\int_{\R^3_v}Q\big (f(t,x,v),g(t,x,v)\big )|v|^2dv\leq C\|\om f(t,x,\cdot)\|_{ L^2_v}\|\om g(t,x,\cdot)\|_{L^2_v}\leq C C_0^2,
\eeno
the last inequality using \eqref{inistep} and  the fact that $f\in\mathcal M_T(C_0)$.  Thus we can derive that
\begin{equation}\label{2E0}
\begin{aligned}
  &\int_{\R^3_v}g(t,x,v)|v|^2dv \leq \int_{\R^3_v}f_{in}(x,v)|v|^2dv+\int_0^t \Big|\pa_x\cdot\int_{\R^3_v} vg(r,x,v)|v|^2dv\Big| dr+CC_0^2t\\
&\leq E_0+C \int_0^t  \|\om g(r)\|_{ H^3_xL^2_v}dr+C  C_0^2 t \leq E_0+C(C_0+C_0^2)t\leq 2E_0,
\end{aligned}
\end{equation}
if we choose $T_*\leq \frac{E_0}{ C(C_0+C^2_0)}.$

Finally, for the upper bound of entropy, we have 
\begin{align*}
&\pa_t \int_{\R^3_v} g\log\big (1+g\big )dv=\int_{\R^3_v}(\pa_t g)\log(1+g)dv+\int_{\R^3_v}\f g{1+g} \pa_t gdv\\
&=\int_{\R^3_v}Q(f,g)\Big(\log(1+g)+\f g{1+g}\Big)dv-\int_{\R^3_v}\Big[(v\cdot \pa_x g)\log(1+g)+ \frac{ (v\cdot\partial_x g)g}{1+g} \Big]dv.
\end{align*}
Direct verification shows that
\begin{eqnarray*}
	\Big| \int_{\R^3_v}\Big[(v\cdot \pa_x g)\log(1+g)+ \frac{ (v\cdot\partial_x g)g}{1+g} \Big]dv\Big|\leq C\|\om g\|^2_{H^3_xL^2_v}.
\end{eqnarray*}
Moreover, by  \cite[Theorem 1.1]{MR3942041} we have  
\begin{align*}
	&\Big|\int_{\R^3_v}Q(f,g)\Big(\log(1+g)+\f g{1+g}\Big)dv\Big|\\
	& \leq  \|\om f\|_{H^3_xL^2_v} \normm{\om g}_{H^3_x}
 \Big (\|\log(1+g)\|_{H^3_xH^s_v}+\big \|  g/( 1+g)\big \|_{H^3_xH^s_v}\Big )\\
& \leq C  \|\om f\|_{H^3_xL^2_v} \normm{\om g}_{H^3_x}
 \|g\|_{H^3_xH^s_v}  \leq CC_0\normm{\om g}^2_{H^3_x},
\end{align*}
where in the last line we used Lemma \ref{lem:A.2} in Appendix \ref{sec:appinter} as well as the fact that $f\in\mathcal M_T(C_0).$  Then we combine the above estimates to conclude that 
\begin{eqnarray*}
	\pa_t \int_{\R^3_v} g\log\big (1+g\big )dv \leq  C\|\om g\|^2_{H^3_xL^2_v}+CC_0\normm{\om g}^2_{H^3_x}
\end{eqnarray*} 
and thus, for any $(t,x)\in [0,T]\times \mathbb T_x^3\subset  [0,T_*]\times \mathbb T_x^3,$
\begin{align*}
& \int_{\R^3_v} g(t,x,v)\log\big (1+g(t,x,v)\big )dv\\
&\leq \int_{\R^3_v} f_{in}(x,v)\log\big (1+f_{in}(x,v)\big )dv +\int_{0}^t \Big(C\|\om g(r)\|^2_{H^3_xL^2_v} +CC_0 \normm{\om g(r)}^2_{H^3_x}\Big)dr\\
& \leq H_0+CC^2_0t+CC_0\Theta^2\leq 2H_0,
\end{align*}
the second inequality using \eqref{+apong} and \eqref{inistep} and the last inequality holding provided we choose $T_*$ and $\Theta$ small such that  $T_*\leq H_0/(4CC^2_0)$ and $\Theta^2<H_0/(4CC_0).$ This with \eqref{lowerbfepsilon} and \eqref{2E0} as well as the fact that $g\geq 0$ yields 
  $g\in \mathcal A_T.$ The proof of Proposition \ref{lem:mc0} is completed. \end{proof}

\begin{prop}
\label{thm: dide} 
Let ${\gamma \over 2}+2s\geq 1$ and let   $C_0$ be   given in \eqref{c0}.
Then  there exists a time $T_*$ 
depending only on $C_0$ in \eqref{c0} and the numbers $C_1$ and $C_2$ in \eqref{lower}, such that for any given  $T\leq T_*$,  if 
\begin{equation*}
	f\in\mathcal A_T\cap \mathcal M_T(C_0)\ \textrm{ and }\  |f|_{X_{\rho, T}} <+\infty\  \textrm{ for some }\ \rho>0, \end{equation*}
 and  if    
   $g$ is a solution to the linear Boltzmann equation \eqref{flineq}
 satisfying   condition \eqref{apong},  
 then  
 there exists a constant $\bar C$ depending only on $C_0, C_1, C_2$ in \eqref{c0} and \eqref{lower} such that 
   the following estimate
 	\begin{equation}
 		\label{mainest}
 		\begin{aligned}
 			  &L_{\rho,k}\sup_{0 < t \leq T}\norm{ \omega  H_{\delta}^{k}g(t)}_{\mathcal{H}^3_xL^2_v}\\
    &\qquad   +L_{\rho,k} \bigg[ \int_{0}^{T}\inner{\normm{\omega H_{\delta}^{k}g (t)}^{2}_{\mathcal{H}^3_x}+\|\<v\>\om H^k_\de g(t)\|^2_{\mathcal{H}^3_xL^2_v}}dt \bigg ]^{1\over 2}\\
 & \leq  \bar C \bigg[ \int_{0}^{T}\inner{\normm{\omega  g  }^{2}_{\mathcal{H}^3_x}+\|\<v\>\om  g \|^2_{\mathcal{H}^3_xL^2_v}}dt \bigg ]^{1\over 2}|f|_{X_{\rho, T}}  + \frac{\bar C}{\rho} |g|_{Y_{\rho, T}} +\frac{\bar C }{\rho} \abs{f}_{X_{\rho, T}} \abs{g}_{Y_{\rho, T}}
 		\end{aligned}
 	\end{equation} 
 	holds true for any integer $k\geq2.$   	Note that $L_{\rho,k}$ in \eqref{mainest} is defined by \eqref{lrma}.
\end{prop}

As a preliminary step to prove Proposition \ref{thm: dide}, we present the following lemma. 
\begin{lem}
\label{lem:prelin} Let $C_0$ be given in \eqref{c0}  and  let $	f\in\mathcal A_T\cap \mathcal M_T(C_0)$ for some $0<T<1$. Suppose  $g$ 
is a solution to the linear Boltzmann equation \eqref{flineq},  satisfying  that
\begin{equation}\label{xyt}
 \abs{g}_{X_{\rho, T}}+\abs{g}_{Y_{\rho, T}}
<+\infty\  \textrm{ for some }   \rho>0 \ \textrm{ and }\ 
 \lim_{t\rightarrow 0} \|\omega  H_\delta^{k} g(t)\|_{\mathcal{H}^3_xL^2_v}=0 \ \textrm{ for }  k\geq1. 
\end{equation}
Then there exists a constant $\tilde C$, depending only on $C_0,C_1, C_2$ in \eqref{c0} and \eqref{lower},  such that for any $k\geq 1$, it holds that 
	 \begin{equation}\label{++energy}
  \begin{aligned}
	&  \sup_{t\leq T}  \|\omega  H_\delta^{k} g(t)\|_{\mathcal{H}^3_xL^2_v}+ \bigg[\int_0^{T} \Big(\normm{\omega   H_\delta^k g}_{\mathcal H_x^3}^2 + \|\<v\>\omega H_\delta^{k}g\|^2_{\mathcal{H}^3_xL^2_v}\Big)dt\bigg]^{1\over2}\\
	&  \leq \tilde C\sup_{t\leq T}  \|\omega  H_\delta^{k} f(t)\|_{\mathcal{H}^3_xL^2_v} \bigg[\int_0^T \Big(  \normm{\omega g}_{\mathcal H_x^{3}}^2+\norm{\comi v\omega g}_{\mathcal H_x^{3} L_v^2}^2\Big) dt\bigg]^{1\over2} \\
	&\quad+ \tilde  C k\bigg[  \int_0^{T}   \Big(\normm{  \omega   H_\delta^{k-1}g}^2_{\mathcal{H}^3_x}+\| \comi v  \omega H_\delta^{k-1}g\|_{\mathcal{H}^3_xL^2_v}^2\Big) dt\bigg]^{1\over2}\\
	&\quad +\tilde C\bigg[\sum_{j=1}^{k-1} {k\choose j}\int_0^{T}  \big| \big(\omega Q (H_\delta^j f,\ H_\delta^{k-j}g), \ \omega H_\delta^{k} g\big)_{\mathcal{H}^3_xL^2_v}\big|dt\bigg]^{1\over2}.	 \end{aligned}
	 \end{equation}
\end{lem}
 
 \begin{proof}
 Observe that  condition \eqref{xyt}    
 ensures rigorous  rather than  formal computations in the following argument.  Thus  for each $k\geq 1$, we  apply $H_\delta^{k}$ to equation \eqref{flineq} and use the Leibniz formula, to get
\begin{multline}\label{hk}
(\pa_t+v\cdot\partial_x)H_\delta^{k}g = -[H_\delta^{k},\pa_t+v\cdot\partial_x]g+H_\delta^{k}Q(f,g)\\
 = k \delta t^{\delta-1}\partial_{v_1}  H_\delta^{k-1}g+\sum_{0\leq j\leq k} {k\choose j}Q(H_\delta^j f,H_\delta^{k-j}g),
\end{multline}
the last line using \eqref{kehigher}.   We multiply  $\omega$  on both sides of \eqref{hk} and observe \eqref{vweight}; this gives that 
\begin{equation}\label{hkf+}
\big (\pa_t+v\cdot\partial_x+\<v\>^2\big )\omega H_\delta^{k} g = k \delta t^{\delta-1}\omega  \partial_{v_1} H_\delta^{k-1} g +\sum_{0\leq j\leq k} {k\choose j}\omega  Q(H_\delta^j f,H_\delta^{k-j}g).
\end{equation}
Furthermore, taking $\mathcal{H}^3_xL^2_v$ inner product with  $\omega H_\delta^{k} g$, we obtain  
\begin{equation}\label{energy}
\begin{aligned}
	&\frac12 \frac{d}{dt} \|\omega  H_\delta^{k} g(t)\|^2_{\mathcal{H}^3_xL^2_v}+ \|\<v\>\omega H_\delta^{k}g\|^2_{\mathcal{H}^3_xL^2_v} \\
	& \leq     Ck t^{\delta-1}  \big| \big ( \omega  \partial_{v_1} H_\delta^{k-1}   g,\ \omega H_\delta^{k} g\big )_{\mathcal{H}^3_xL^2_v}\big|  +\big(\omega Q (   f,\ H_\delta^{k}g), \ \omega H_\delta^{k} g\big)_{\mathcal{H}^3_xL^2_v}\\
	&\quad+ \big(\omega Q (H_\delta^k f,   g),  \omega H_\delta^{k} g\big)_{\mathcal{H}^3_xL^2_v}+ \sum_{j=1}^{k-1} \binom{k}{j}   \big(\omega Q (H_\delta^j f,\ H_\delta^{k-j}g), \ \omega H_\delta^{k} g\big)_{\mathcal{H}^3_xL^2_v}.
	\end{aligned}
	\end{equation}
For   the second and the third terms on the right-hand side of \eqref{energy},  we use Lemmas \ref{lem:j1}  and  \ref{lem:j2}   to obtain that
\begin{equation*}
 \big (\omega Q(f, H_\delta^k g),  \omega  H_\delta^k g\big )_{\mathcal{H}^3_xL^2_v}  \leq -\frac{C_1}{2}  \normm{\omega   H_\delta^k  g}_{\mathcal H_x^3}^2+\frac{1}{2}\|\comi v \omega H_\delta^k g\|_{\mathcal H_x^3 L^2_{v}}^2  +C     \|  \omega H_\delta^k g\|_{\mathcal H_x^3 L^2_{v}}^2, 
    \end{equation*}
and that, for any $\eps>0,$  
   \begin{align*}
	  \big (\omega Q(H_\delta^k f, g),\ \omega  H_\delta^k g\big )_{\mathcal{H}^3_xL^2_v}  \leq & \eps  \big(\normm{\omega H_\delta^k g}_{\mathcal H_x^3}^2+\norm{\comi v \omega H_\delta^k g}_{\mathcal H_x^3L_v^2}^2\big) \\ 
	  &+C\eps^{-1}  \norm{\omega H_\delta^k f}_{\mathcal H_x^3L_v^2}^2  \big(\normm{\omega g}_{\mathcal H_x^{3}}^2+\norm{\comi v\omega g}_{\mathcal H_x^{3} L_v^2}^2\big).    
\end{align*}
Substituting the two inequalities above into \eqref{energy} yields that
\begin{equation}\label{ert}
	\begin{aligned}
	&	\frac12 \frac{d}{dt} \|\omega  H_\delta^{k} g(t)\|^2_{\mathcal{H}^3_xL^2_v}+\frac{C_1}{4} \big(\normm{\omega H_\delta^k g}_{\mathcal H_x^3}^2+\|\<v\>\omega H_\delta^{k}g\|^2_{\mathcal{H}^3_xL^2_v}\big)\\
	&\leq Ck t^{\delta-1}  \big| \big ( \omega  \partial_{v_1} H_\delta^{k-1}   g,  \omega H_\delta^{k} g\big )_{\mathcal{H}^3_xL^2_v}\big|  + \sum_{j=1}^{k-1} \binom{k}{j}   \big(\omega Q (H_\delta^j f,\ H_\delta^{k-j}g), \ \omega H_\delta^{k}g\big)_{\mathcal{H}^3_xL^2_v} \\
	&\quad+C \|  \omega H_\delta^k g\|_{\mathcal H_x^3 L^2_{v}}^2+C  \norm{\omega H_\delta^k f}_{\mathcal H_x^3L_v^2}^2  \big(\normm{\omega g}_{\mathcal H_x^{3}}^2+\norm{\comi v\omega g}_{\mathcal H_x^{3} L_v^2}^2\big).
	\end{aligned}
\end{equation}
To deal with the first term on the right-hand side of \eqref{ert},	we use the fact that
  $\delta\geq 1$ to compute that,  for any $\eps>0$ and for any  $t\in[0,T]$ with $T\leq 1$,    
\begin{multline}\label{dp1}
  k t^{\delta-1} \big| \big ( \omega  \partial_{v_1} H_\delta^{k-1}   g,\ \omega H_\delta^{k} g\big )_{\mathcal{H}^3_xL^2_v}\big|  \\
  \leq 
\eps     \norm{\comi {D_v}^{\frac12}  (\omega H_\delta^{k}g)}^2_{\mathcal{H}^3_xL^2_v}  + C \eps ^{-1}  k^2   \| \comi {D_v}^{-\frac12}(\omega   \pa_{v_1} H_\delta^{k-1}g )\|^2_{\mathcal{H}^3_xL^2_v}.
\end{multline}
Moreover,  in view of \eqref{vweight},
\begin{multline*}
	  \| \comi {D_v}^{-\frac12}(\omega \pa_{v_1} H_\delta^{k-1}g)\|_{\mathcal{H}^3_xL^2_v} \\
	  \leq     \| \comi {D_v}^{-\frac12}\pa_{v_1}  (\omega H_\delta^{k-1}g)\|_{\mathcal{H}^3_xL^2_v}+  \| \inner{ \partial_{v_1} \omega }H_\delta^{k-1}g\|_{\mathcal{H}^3_xL^2_v}\\
	 \leq    \| \comi {D_v}^{\frac12}  ( \omega  H_\delta^{k-1}g)\|_{\mathcal{H}^3_xL^2_v}+C \| \comi v  \omega H_\delta^{k-1}g\|_{\mathcal{H}^3_xL^2_v}.
\end{multline*}
Recall $\frac{\ga}{2}+2s\geq 1$ and thus 
  $\tau=\frac{2s}{2-\gamma}\geq \frac12$. This, with Corollary \ref{corollary:coer}, yields that
\begin{multline*}
\| \comi {D_v}^{\frac{1}{2}}  ( \omega  H_\delta^{k-1}g)\|_{\mathcal{H}^3_xL^2_v}\leq 	\| \comi {D_v}^{\tau}   (\omega  H_\delta^{k-1}g)\|_{\mathcal{H}^3_xL^2_v}\\
\leq C\normm{   \omega  H_\delta^{k-1}g}_{\mathcal{H}^3_x}+C  \| \comi v     \omega  H_\delta^{k-1}g\|_{\mathcal{H}^3_xL^2_v}.
	\end{multline*}
	Similarly for the first term on the right-hand side of \eqref{dp1}. Thus
	combining the above estimates yields that, for any $\eps>0,$
\begin{equation}\label{ktd1}
\begin{aligned}
	&k t^{\delta-1}  \big| \big ( \omega  \partial_{v_1} H_\delta^{k-1}   g,\ \omega H_\delta^{k} g\big )_{\mathcal{H}^3_xL^2_v}\big|  \\
&  \leq   \eps     \big(\normm{  \omega   H_\delta^{k}g}^2_{\mathcal{H}^3_x}+\| \comi v  \omega H_\delta^{k}g\|_{\mathcal{H}^3_xL^2_v}^2\big)  + C_\eps   k^2   \big(\normm{  \omega   H_\delta^{k-1}g}^2_{\mathcal{H}^3_x}+\| \comi v  \omega H_\delta^{k-1}g\|_{\mathcal{H}^3_xL^2_v}^2\big).
 \end{aligned}
\end{equation}
 Now we combine the   above estimate  with \eqref{ert}, to conclude that, for any $k\geq 1,$ 
\begin{equation*}
	\begin{aligned}
	&\frac12 \frac{d}{dt} \|\omega  H_\delta^{k} g(t)\|^2_{\mathcal{H}^3_xL^2_v}+\frac{C_1}{8}\big(\normm{\omega   H_\delta^k  g}_{\mathcal H_x^3}^2 + \|\<v\>\omega H_\delta^{k}g\|^2_{\mathcal{H}^3_xL^2_v}\big) \\
	& \leq   C \|  \omega H_\delta^k g\|_{\mathcal H_x^3 L^2_{v}}^2+  C\|  \omega H_\delta^k f\|_{\mathcal H_x^3 L^2_{v}}^2 \big(  \normm{\omega g}_{H_x^{3}}^2+\norm{\comi v\omega g}_{H_x^{3} L_v^2}^2\big)  \\
	& \quad+C    k^2   \big(\normm{  \omega   H_\delta^{k-1}g}^2_{\mathcal{H}^3_x}+\| \comi v  \omega H_\delta^{k-1}g\|_{\mathcal{H}^3_xL^2_v}^2\big)\\
  &\quad + \sum_{j=1}^{k-1} {k\choose j} \big(\omega Q (H_\delta^j f,\ H_\delta^{k-j}g), \ \omega H_\delta^{k} g\big)_{\mathcal{H}^3_xL^2_v},
	\end{aligned}
\end{equation*}
 where   $C_1$   is the constant in \eqref{lower}. This, with Gronwall's inequality and assumption \eqref{xyt},  yields the assertion in   Lemma \ref{lem:prelin}. The proof  is completed. 
\end{proof}

\begin{proof}
	[Proof of Proposition \ref{thm: dide}]   
By  Lemma \ref{lem:prelin} it follows that, for $k\geq 2,$
\begin{equation}\label{riindu}
	 \begin{aligned}
	&  \sup_{t\leq T}  \|\omega  H_\delta^{k}g(t)\|_{\mathcal{H}^3_xL^2_v}+ \bigg[\int_0^{T} \Big(\normm{\omega   H_\delta^k  g}_{\mathcal H_x^3}^2 + \|\<v\>\omega H_\delta^{k}g\|^2_{\mathcal{H}^3_xL^2_v}\Big)dt\bigg]^{1\over2}\\
	&  \leq \tilde C\sup_{t\leq T}  \|\omega  H_\delta^{k} f \|_{\mathcal{H}^3_xL^2_v} \bigg[\int_0^T \Big(  \normm{\omega g}_{\mathcal H_x^{3}}^2+\norm{\comi v\omega g}_{\mathcal H_x^{3} L_v^2}^2\Big) dt\bigg]^{1\over2} \\
	&\quad+  \tilde C k\bigg[  \int_0^{T}   \Big(\normm{  \omega   H_\delta^{k-1}g}^2_{\mathcal{H}^3_x}+\| \comi v  \omega H_\delta^{k-1}g\|_{\mathcal{H}^3_xL^2_v}^2\Big) dt\bigg]^{1\over2}\\
	&\quad +\tilde C\bigg[\sum_{j=1}^{k-1} {k\choose j}\int_0^{T}  \big| \big(\omega Q (H_\delta^j f,\ H_\delta^{k-j}g), \ \omega H_\delta^{k} g\big)_{\mathcal{H}^3_xL^2_v}\big|dt\bigg]^{1\over2},	
	 \end{aligned}
\end{equation}
where the constant $\tilde C$ depends only on $C_0,C_1$ and $C_2$ in \eqref{c0} and \eqref{lower}.  
To deal with the terms on the right-hand side, we use  Definition \ref{assum} for ${\ga\over 2}+2s\geq 1$,  to conclude   the following estimates
 that, for any $j\geq 1,$
\begin{equation}\label{max}
\sup_{t\leq T} \norm{ \omega H_\delta^j f}_{ \mathcal H_x^3  L_v^2} \leq \abs{f}_{X_{\rho,T}}\frac{ \rho^{ j-1} j! }{(j+1)^3}  
\end{equation}
and
\begin{equation}
	\label{eng}
	\bigg[\int_0^{T}\Big(  \normm{ \omega  H_\delta^jg}_{\mathcal H_x^3 }^2+ \norm{ \comi v\omega  H_\delta^jg}_{\mathcal H_x^3L_v^2}^2\Big)dt\bigg]^{1\over2}
	 \leq \abs{g}_{Y_{\rho,T}}\frac{ \rho^{ j-1} j! }{(j+1)^3}.
\end{equation} 
 Then we have  
\begin{equation}\label{tch}
\begin{aligned}
	&\tilde C\sup_{t\leq T}  \|\omega  H_\delta^{k} f \|_{\mathcal{H}^3_xL^2_v} \bigg[\int_0^T \Big(  \normm{\omega g}_{\mathcal H_x^{3}}^2+\norm{\comi v\omega g}_{\mathcal H_x^{3} L_v^2}^2\Big) dt\bigg]^{1\over2} \\
	&\leq  \tilde C  |f|_{X_{\rho, T}} \bigg[\int_0^T \Big(  \normm{\omega g}_{\mathcal H_x^{3}}^2+\norm{\comi v\omega g}_{\mathcal H_x^{3} L_v^2}^2\Big) dt\bigg]^{1\over2}  \frac{ \rho^{ k-1} k! } {(k+1)^3}. 
	\end{aligned}
\end{equation}
Moreover, using \eqref{eng}  yields that, for $k\geq 2,$
\begin{equation}\label{tck}
	 \tilde C k\bigg[  \int_0^{T}   \Big(\normm{  \omega   H_\delta^{k-1}g}^2_{\mathcal{H}^3_x}+\| \comi v  \omega H_\delta^{k-1}g\|_{\mathcal{H}^3_xL^2_v}^2\Big) dt\bigg]^{1\over2} \leq  \frac{\tilde C}{\rho} |g|_{Y_{\rho, T}}\frac{ \rho^{ k-1} k! }{(k+1)^3}. 
\end{equation}
For the last term on the right-hand side of \eqref{riindu}, we use  Corollary  \ref{cor:upper}, to obtain that
\begin{multline}
    \label{wq}
    \big (\omega Q(H_\de^jf,  H_\de^{k-j}g),\ \omega  H_\de^{k} g\big)_{\mathcal{H}^3_xL^2_v}\\
 \leq  C  \|\om H^j_\de f\|_{\mathcal{H}^3_xL^2_v}\Big(\normm{\om H^{k-j}_\de g}_{\mathcal{H}^3_x}+\|\comi v\om H^{k-j}_\de g\|_{\mathcal{H}^3_xL^2_v}\Big) \\
 \times \Big(\normm{\om H^{k}_\de g}_{\mathcal{H}^3_x}+\|\comi v\om H^{k}_\de g\|_{\mathcal{H}^3_xL^2_v}\Big).
\end{multline}
So that
\begin{equation}\label{368}
\begin{aligned}
&\sum_{j=1}^{k-1} {k\choose j}\int_0^{T}  \big|\big(\omega Q (H_\delta^j f,\ H_\delta^{k-j}g), \ \omega H_\delta^{k} g\big)_{\mathcal{H}^3_xL^2_v}\big|dt\\
&  \leq     C \bigg[\int_0^{T}\Big(\normm{\om H^{k}_\de g}_{\mathcal{H}^3_x}^2+\|\comi v\om H^{k}_\de g\|_{\mathcal{H}^3_xL^2_v}^2\Big) dt\bigg]^{1\over2}\\
&\quad\quad\times \sum_{j=1}^{k-1}\binom{k}{j}  \sup_{t\leq T} \|\om H^j_\de f\|_{\mathcal{H}^3_xL^2_v}    \bigg[\int_0^{T}\Big(\normm{\om H^{k-j}_\de g}_{\mathcal{H}^3_x}^2+\|\comi v\om H^{k-j}_\de g\|_{\mathcal{H}^3_xL^2_v}^2\Big) dt\bigg]^{1\over2}.
	 \end{aligned}
\end{equation}
For the last term on the right-hand side, we use \eqref{max} and \eqref{eng}, to compute
\begin{equation}
    \label{comp}
    \begin{aligned}
	&\sum_{j=1}^{k-1}\binom{k}{j}  \sup_{t\leq T} \|\om H^j_\de f\|_{\mathcal{H}^3_xL^2_v}    \bigg[\int_0^{T}\Big(\normm{\om H^{k-j}_\de g}_{\mathcal{H}^3_x}^2+\|\comi v\om H^{k-j}_\de g\|_{\mathcal{H}^3_xL^2_v}^2\Big) dt\bigg]^{1\over2}\\
	 &\leq  \abs{f}_{X_{\rho, T}} \abs{g}_{Y_{\rho, T}}\sum_{j=1}^{k-1}\frac{k!}{j!(k-j)!}\frac{ \rho^{j-1}j!}{(j+1)^3}\frac{  \rho^{k-j-1}(k-j)!}{(k-j+1)^3}\\
	 &\leq  \frac{C }{\rho} \abs{f}_{X_{\rho, T}} \abs{g}_{Y_{\rho, T}} \frac{  \rho^{k-1}k!}{(k+1)^3}, 
\end{aligned}
\end{equation} 
where in the last line, we use inequality
\begin{align*}
\sum_{ 1\leq j \leq k-1}\frac{ (k+1)^3}{(j+1)^3(k-j+1)^3}\leq \sum_{j\in \Z_+}\frac{C}{(j+1)^3}\leq C. 
\end{align*} 
Substituting the above inequality into \eqref{368} yields that, for any $\eps>0,$
\begin{multline*}
\sum_{j=1}^{k-1} {k\choose j}\int_0^{T}  \big|\big(\omega Q (H_\delta^j f,\ H_\delta^{k-j}f), \ \omega H_\delta^{k} f\big)_{\mathcal{H}^3_xL^2_v}\big|dt\\
	\leq \eps \int_0^{T}\Big(\normm{\om H^{k}_\de f}_{\mathcal{H}^3_x}^2+\|\comi v\om H^{k}_\de f\|_{\mathcal{H}^3_xL^2_v}^2\Big) dt+ \eps^{-1}  \bigg( \frac{C }{\rho} \abs{f}_{X_{\rho, T}} \abs{g}_{Y_{\rho, T}} \frac{  \rho^{k-1}k!}{(k+1)^3}\bigg)^2.
\end{multline*}
We combine the above estimate and \eqref{tch} and \eqref{tck} with \eqref{riindu}, to conclude that, for any $k\geq 2$, 
\begin{multline*}
	 \frac{(k+1)^3}{ \rho^{ k-1} k! }  \sup_{t\leq T}  \|\omega  H_\delta^{k} f(t)\|_{\mathcal{H}^3_xL^2_v}+ \frac{(k+1)^3}{ \rho^{ k-1} k! }  \bigg[\int_0^{T} \Big(\normm{\omega   H_\delta^k  f}_{\mathcal H_x^3}^2 + \|\<v\>\omega H_\delta^{k}f\|^2_{\mathcal{H}^3_xL^2_v}\Big)dt\bigg]^{1\over2} \\
	 \leq \tilde C  |f|_{X_{\rho, T}} \bigg[\int_0^T \Big(  \normm{\omega g}_{\mathcal H_x^{3}}^2+\norm{\comi v\omega g}_{\mathcal H_x^{3} L_v^2}^2\Big) dt\bigg]^{1\over2} + \frac{C}{\rho} |g|_{Y_{\rho, T}} +\frac{C }{\rho} \abs{f}_{X_{\rho, T}} \abs{g}_{Y_{\rho, T}}. 
\end{multline*}
Symmetrically,    the above estimate  still holds  true  if we replace $H_\delta$ by  
 	\begin{equation*}
 		\frac{1}{\delta+1}t^{\delta+1} \partial_{x_j}+ t^{\delta} \partial_{v_j} \textrm{ with } j=2 \textrm{ or } 3.
 	\end{equation*}
 This completes the proof of Proposition  \ref{thm: dide}. 
 \end{proof}

 \begin{proof}[Proof of Theorem \ref{thm:apri}: 
the case of $\frac{\gamma}{2}+2s\geq 1$]  Recall $C_0$ is given in \eqref{c0} and  $\mathcal H_x^3=\mathcal H_{x,N}^3$ with $N$ depending only on $C_0$ and $C_1$ (see  Lemma \ref{lem:j1}). It   follows from \eqref{inistep} that 
	\begin{equation}\label{wg0}
	\begin{aligned}
		& \sup_{t\leq T}  \|\omega    g \|_{\mathcal{H}^3_xL^2_v}+ \bigg[\int_0^{T} \Big(\normm{\omega      g}_{\mathcal H_x^3}^2 + \|\<v\>\omega   g\|^2_{\mathcal{H}^3_xL^2_v}\Big)dt\bigg]^{1\over2}\\
		&\leq 
		N\sup_{t\leq T}  \|\omega    g \|_{H^3_xL^2_v}+N \bigg[\int_0^{T} \Big(\normm{\omega      g}_{H_x^3}^2 + \|\<v\>\omega   g\|^2_{H^3_xL^2_v}\Big)dt\bigg]^{1\over2}\leq   \frac{NC_0}{C_1}.
		\end{aligned}
	\end{equation}
	Moreover, 
	we  use \eqref{++energy} for $k=1$ to obtain that 
\begin{equation}\label{ke1}
	 \begin{aligned}
	&  \sup_{t\leq T}  \|\omega  H_\delta g(t)\|_{\mathcal{H}^3_xL^2_v}+ \bigg[\int_0^{T} \Big(\normm{\omega   H_\delta   g}_{\mathcal H_x^3}^2 + \|\<v\>\omega H_\delta g\|^2_{\mathcal{H}^3_xL^2_v}\Big)dt\bigg]^{1\over2}\\
	&  \leq \tilde C\sup_{t\leq T}  \|\omega  H_\delta f \|_{\mathcal{H}^3_xL^2_v} \bigg[\int_0^T \Big(  \normm{\omega g}_{\mathcal H_x^{3}}^2+\norm{\comi v\omega g}_{\mathcal H_x^{3} L_v^2}^2\Big) dt\bigg]^{1\over2} \\
	&\quad+  \tilde C  \bigg[  \int_0^{T}   \Big(\normm{  \omega    g}^2_{\mathcal{H}^3_x}+\| \comi v  \omega  g\|_{\mathcal{H}^3_xL^2_v}^2\Big) dt\bigg]^{1\over2}\\
	&\leq \tilde C\sup_{t\leq T}  \|\omega  H_\delta f \|_{\mathcal{H}^3_xL^2_v} \bigg[\int_0^T \Big(  \normm{\omega g}_{\mathcal H_x^{3}}^2+\norm{\comi v\omega g}_{\mathcal H_x^{3} L_v^2}^2\Big) dt\bigg]^{1\over2} +\frac{N\tilde C C_0}{C_1},	
	 \end{aligned}
\end{equation}	
the last inequality using \eqref{wg0}. 
Now we choose $C_*$ large enough such that 
	\begin{equation*}
		C_*\geq  192 \frac{N\tilde C C_0}{C_1} +1.
	\end{equation*}
 As a result,   we combine \eqref{max} and \eqref{ke1} to obtain  that if
 \begin{equation*}
	\abs{f}_{X_{\rho, T}}\leq C_*,
\end{equation*}
then for $0\leq k\leq 1$ it holds that 
\begin{equation*}
	\begin{aligned}
&     L_{\rho,k} \sup_{t\leq T}  \|\omega  H_\delta^{k} f(t)\|_{\mathcal{H}^3_xL^2_v}+L_{\rho,k} \bigg[\int_0^{T} \Big(\normm{\omega   H_\delta^k  f}_{\mathcal H_x^3}^2 + \|\<v\>\omega H_\delta^{k}f\|^2_{\mathcal{H}^3_xL^2_v}\Big)dt\bigg]^{1\over2}\\
&\leq \tilde C C_*  \bigg[\int_0^T \Big(  \normm{\omega g}_{\mathcal H_x^{3}}^2+\norm{\comi v\omega g}_{\mathcal H_x^{3} L_v^2}^2\Big) dt\bigg]^{1\over2} + \frac{C_*}{24}.
\end{aligned}
\end{equation*}
Furthermore, for  $k\geq 2,$  by Proposition \ref{thm: dide}  it follows that if  
\begin{equation*}
\abs{f}_{X_{\rho, T}}\leq C_*
\end{equation*}
then
\begin{equation}
 		\label{mainest+}
 		\begin{aligned}
 			  &L_{\rho,k}\sup_{t \leq T}\norm{ \omega  H_{\delta}^{k}g(t)}_{\mathcal{H}^3_xL^2_v}\\
    &\qquad   +L_{\rho,k} \bigg[ \int_{0}^{T}\inner{\normm{\omega H_{\delta}^{k}g (t)}^{2}_{\mathcal{H}^3_x}+\|\<v\>\om H^k_\de g(t)\|^2_{\mathcal{H}^3_xL^2_v}}dt \bigg ]^{1\over 2}\\
 & \leq   \bar C  \bigg[ \int_{0}^{T}\inner{\normm{\omega  g  }^{2}_{\mathcal{H}^3_x}+\|\<v\>\om  g \|^2_{\mathcal{H}^3_xL^2_v}}dt \bigg ]^{1\over 2}|f|_{X_{\rho, T}}    + \frac{\bar C}{\rho} |g|_{Y_{\rho, T}} +\frac{\bar C }{\rho} \abs{f}_{X_{\rho, T}} \abs{g}_{Y_{\rho, T}}\\
 & \leq   \bar C C_* \bigg[ \int_{0}^{T}\inner{\normm{\omega  g  }^{2}_{\mathcal{H}^3_x}+\|\<v\>\om  g \|^2_{\mathcal{H}^3_xL^2_v}}dt \bigg ]^{1\over 2}    + \frac{\bar C}{\rho} |g|_{Y_{\rho, T}} +\frac{\bar C C_*}{\rho}  \abs{g}_{Y_{\rho, T}}.
 		\end{aligned}
 	\end{equation} 
 Symmetrically,  the  two   estimates above   still hold   true  if we replace $H_\delta$ by  
 	\begin{equation*}
 		\frac{1}{\delta+1}t^{\delta+1} \partial_{x_j}+ t^{\delta} \partial_{v_j} \textrm{ with } j=2 \textrm{ or } 3.
 	\end{equation*}
Consequently,  in view of   \eqref{xyrt}
, it follows that  
  \begin{equation*}
  \begin{aligned}
  	 |g|_{X_{\rho, T}}+|g|_{Y_{\rho, T}}\leq &    C C_* \bigg[ \int_{0}^{T}\inner{\normm{\omega  g  }^{2}_{\mathcal{H}^3_x}+\|\<v\>\om  g \|^2_{\mathcal{H}^3_xL^2_v}}dt \bigg ]^{1\over 2} \\
  	&\qquad  \qquad + \frac{1}{4} C_*  + \frac{  C}{\rho} |g|_{Y_{\rho, T}} +\frac{  C C_* }{ \rho}  \abs{g}_{Y_{\rho, T}} 
  \end{aligned}
 \end{equation*}	
 for some constant $C$ depending only only $C_0, C_1$ and $C_2$  in \eqref{c0} and \eqref{lower}. 
We choose $\rho$ large enough such that 
\begin{equation*}
	\rho\geq  2 C  +2 C C_*.
\end{equation*}
Then it follows that 
\begin{equation}\label{12C}
	 |g|_{X_{\rho, T}}+|g|_{Y_{\rho, T}}\leq    2C C_* \bigg[ \int_{0}^{T}\inner{\normm{\omega  g  }^{2}_{\mathcal{H}^3_x}+\|\<v\>\om  g \|^2_{\mathcal{H}^3_xL^2_v}}dt \bigg ]^{1\over 2}+\frac{C_*}{2}.
\end{equation}
Thus \eqref{3.7} holds true for $\tilde{C}=2 C$. Moreover, 
by Proposition \ref{lem:mc0}, we can find two constants $T_*, \Theta>0$, such that for any $T\leq T_*,$ if 
\begin{equation*}
	f\in\mathcal A_T\cap \mathcal M_T(C_0) \ \textrm{ and }\  \bigg[ \int_0^T \big( \normm{\omega  g}_{ \mathcal H_x^3   }^2  +\norm{ \comi{v}\omega g}_{ \mathcal H_x^3  L_v^2} ^2 \big)dt \bigg]^{1\over2} \leq \Theta, \end{equation*} 
	then 
	\begin{equation*}
		g\in\mathcal A_T\cap\mathcal M_T(C_0). 
	\end{equation*}
Shrinking $\Theta$ if necessary, we deduce from \eqref{12C} that  if  
\beno
\bigg[\int_{0}^{T}\inner{\normm{\omega  g  }^{2}_{\mathcal{H}^3_x}+\|\<v\>\om  g \|^2_{\mathcal{H}^3_xL^2_v}}dt\bigg]^{1\over2}<\Theta,
\eeno
then  
\beno
|g|_{X_{\rho, T}}+|g|_{Y_{\rho, T}}\leq 2 C  C_*\Theta+\f{C_*} 2<C_*.
\eeno
This yields that $g\in \mathcal N_T(\rho, C_*).$
Thus we complete the proof of Theorem \ref{thm:apri} in the case of  $\frac{\gamma}{2}+2s\geq 1.$
\end{proof}

\section{Proof of the \emph{a priori} estimate: smoothing effect in  Gevrey space}\label{sec:gev}

 In this part we will prove Theorem \ref{thm:apri}
  for  the   remaining case that $\frac{\ga}{2}+2s<1$, and the argument is similar to that in the previous case of   $\frac{\ga}{2}+2s\geq 1$, with the main difference arising from the estimate on the commutator between the transport operator and $H_{\delta_i}, i=1,2.$    Accordingly we will work with the  
  mixed  derivatives of $H_{\delta_1,j}$ and $H_{\delta_2,j}$ (see \eqref{pkh} or \eqref{mixedder} below for more detail).    
  
   For ${\gamma \over 2}+2s< 1$  we have,  in view of Definition \ref{assum},
 \begin{equation*}
 	\abs{h}_{X_{\rho, T}}= \sum_{ 1\leq j\leq 3}    \sup_{k\geq0}	\Big(L_{\rho,k}\sup_{\abs\alpha=k} \sup_{t\leq T} \norm{\omega \vec{ H}_{j}^\alpha h(t) }_{ \mathcal H_x^3  L_v^2} \Big),
 		 \end{equation*}
 		 and
 		 \begin{equation*}
 		 \abs{h}_{Y_{\rho, T}}= \sum_{ 1\leq j\leq 3}    \sup_{k\geq0}	\bigg[L_{\rho,k}\sup_{\abs\alpha=k}\bigg(\int_0^T\Big(\normm{ \omega \vec{ H}_{j}^\alpha h(t) }_{ \mathcal H_x^3}^2+ \norm{\comi v\omega \vec{ H}_{j}^\alpha h(t) }_{ \mathcal H_x^3  L_v^2}^2\Big)dt\bigg)^{1\over2} \bigg],
 		 \end{equation*}
 where  we use the notation that
 	\begin{equation*}
 		\vec H_j^\alpha:=H_{\delta_1,j}^{\alpha_1}H_{\delta_2,j}^{\alpha_2} \ \textrm{ for } \ \alpha=(\alpha_1,\alpha_2)\in\mathbb Z_+^2,
 	\end{equation*} 
and 
 \begin{equation}\label{lrt}
 	L_{\rho, k}=\left\{
	\begin{aligned}
		& 1, \    \textrm{ if } \ k=0, \\
		 & \frac{(k+1)^3}{ \rho^{ k-1 }(k!)^{(2\tau)^{-1}}} ,\    \textrm{ if } \  k \geq 1. 
	\end{aligned}
	\right.
 \end{equation}
 Note  that Proposition \ref{lem:mc0} still holds true in the case of   ${\gamma \over 2}+2s<1.$ Thus as in the previous Section \ref{sec:analy}, in order  to prove  Theorem \ref{thm:apri}  for  ${\gamma \over 2}+2s<1$,  it suffices to proving the following  quantitative estimate  on  the mixed directional derivations  with respect to $H_{\delta_1}$  and $H_{\delta_2}$, which is the counterpart of Proposition \ref{thm: dide} in the case of   ${\gamma \over 2}+2s<1.$

\begin{prop}\label{thm:sharpgev}
There exists a time $T_*$  depending only on $C_0$ in \eqref{c0} and the numbers $C_1$ and $C_2$ in \eqref{lower}, such that for any given  $T\leq T_*$,  if 
\begin{equation*}
	f\in\mathcal A_T\cap \mathcal M_T(C_0), \quad |f|_{X_{\rho, T}} <+\infty\  \textrm{ for some }\ \rho>0, \end{equation*}
 and    if 
   $g$ is a solution to the linear Boltzmann equation \eqref{flineq}
 satisfying   condition \eqref{apong},  
 then    there exists a constant $\bar C$ depending only on $C_0, C_1, C_2$ in \eqref{c0} and \eqref{lower}, such that 
   the following estimate
   \begin{equation} \label{mainquan+}
   \begin{aligned}
   & L_{\rho,k}\sup_{\abs\alpha=k}\   \sup_{t \leq T}  \norm{ \omega  \vec H^\alpha f(t)}_{\mathcal{H}^3_xL^2_v}  \\ 	
&\qquad+ L_{\rho,k} \sup_{\abs\alpha=k}   \Big(  \int_{0}^{T}\big(\normm{\omega   \vec H^\alpha f }^{2}_{\mathcal{H}^3_x}+\|\<v\>\om \vec H^\alpha f\|^2_{\mathcal{H}^3_xL^2_v}\big)dt \Big )^{1\over 2} \\
& \leq  \bar C \bigg[ \int_{0}^{T}\inner{\normm{\omega  g  }^{2}_{\mathcal{H}^3_x}+\|\<v\>\om  g \|^2_{\mathcal{H}^3_xL^2_v}}dt \bigg ]^{1\over 2}|f|_{X_{\rho, T}}    + \frac{\bar C}{\rho} |g|_{Y_{\rho, T}} +\frac{\bar C }{\rho} \abs{f}_{X_{\rho, T}} \abs{g}_{Y_{\rho, T}} 
 \end{aligned}
 	 	\end{equation}
 	holds true for any integer $k\geq 2$,   where  
 	\begin{equation*}
 		\vec H^\alpha:=H_{\delta_1}^{\alpha_1}H_{\delta_2}^{\alpha_2},\quad \alpha=(\alpha_1,\alpha_2)\in\mathbb Z_+^2, 
 	\end{equation*} 
 	with $\delta_1,\delta_2$ given in \eqref{de1de2} and $H_{\delta_j}$ defined by \eqref{vecM}.   
 	 	 	Note $L_{\rho,k}$ in \eqref{mainquan+}  is defined by \eqref{lrt}.
\end{prop}
 
\begin{proof} 
The proof is quite similar to that of Proposition \ref{thm: dide},  and the main difference arises from  the treatment of the commutator between $\partial_t+v\cdot\partial_x$  and the directional derivatives.  	 
	
Similar to \eqref{hk},  for any multi-index $\alpha=(\alpha_1,\alpha_2)\in\mathbb Z_+^2$ ,  we apply $\omega \vec{ H}^\alpha  $ to \eqref{Bolt} to obtain that 
\begin{align*}
	 (\pa_t+v\cdot\partial_x+\<v\>^2)\omega  \vec{ H}^\alpha g
 &=\omega      (\alpha_1 t^{\de_1-1}\partial_{v_1} H_{\delta_2}^{\alpha_2 } H_{\delta_1}^{\alpha_1-1}+\alpha_2 t^{\de_2-1}\partial_{v_1} H_{\delta_1}^{\alpha_1 } H_{\delta_2}^{\alpha_2-1} )g
 \\+& \sum_{\beta\leq \alpha}{\alpha\choose \beta} \omega Q\big ( \vec H ^\beta f, \ \vec H^{\alpha-\beta} g \big ),
\end{align*}
where we used the fact that 
\begin{align*}
	[\pa_t+v\cdot\partial_x, \ \vec H ^{\alpha}]&=
	[\pa_t+v\cdot\partial_x, \ H_{\delta_1}^{\alpha_1}]H_{\delta_2}^{\alpha_2}+H_{\delta_1}^{\alpha_1}[\pa_t+v\cdot\partial_x, \ H_{\delta_2}^{\alpha_2}]\\
	&= \alpha_1 t^{\de_1-1}\partial_{v_1} H_{\delta_1}^{\alpha_1-1} H_{\delta_2}^{\alpha_2}+ \alpha_2 t^{\de_2-1}\partial_{v_1} H_{\delta_1}^{\alpha_1 }H_{\delta_2}^{\alpha_2-1}
\end{align*}
due to \eqref{kehigher}.   Repeating the argument in 
  the  proof of Lemma \ref{lem:prelin},   we   conclude that, for any    $\alpha\in\mathbb Z_+^2$  with $\abs\alpha=k\geq 2,$
  \begin{equation}\label{supal}
	 \begin{aligned}
	&  \sup_{t\leq T}  \|\omega   \vec{ H}^\alpha g(t)\|_{\mathcal{H}^3_xL^2_v}+\bigg[ \int_0^{T} \Big(\normm{\omega    \vec{ H}^\alpha g}_{\mathcal H_x^3}^2 + \|\<v\>\omega  \vec{ H}^\alpha g\|^2_{\mathcal{H}^3_xL^2_v}\Big)dt\bigg]^{1\over2}\\
	&  \leq  
	 C\sup_{t\leq T}  \|\omega   \vec H^{\alpha} f \|_{\mathcal{H}^3_xL^2_v} \bigg[\int_0^T \Big(  \normm{\omega g}_{\mathcal H_x^{3}}^2+\norm{\comi v\omega g}_{\mathcal H_x^{3} L_v^2}^2\Big) dt\bigg]^{1\over2} \\
	&\quad+   C\bigg[  \sum_{\stackrel{1\leq i,j\leq 2}{ i\neq j}}\int_0^{T}  \big| \big (\omega     \alpha_j t^{\de_j-1}\partial_{v_1} H_{\delta_i}^{\alpha_i} H_{\delta_j}^{\alpha_j-1} g,\  \omega \vec{ H}^\alpha g\big )_{\mathcal{H}^3_xL^2_v}\big|dt\bigg]^{1\over2}  \\
	&\quad +  C\bigg[ \sum_{\stackrel{ \beta\leq \alpha}{ 1\leq \abs\beta \leq k-1}} {\alpha\choose \beta}\int_0^{T}   \big| \big(\omega Q ( \vec{ H}^\beta f,\  \vec{ H}^{\alpha-\beta}g), \ \omega \vec H^{\alpha} g\big)_{\mathcal{H}^3_xL^2_v}\big|dt\bigg]^{1\over2}. 	 \end{aligned}
\end{equation}
Next we deal with the second term on the right-hand side of \eqref{supal}, and the treatment is different from that of the counterpart in \eqref{energy}. 
For any  $\eps>0$, Cauchy inequality yields that 
\begin{equation}\label{tranest}
	\begin{aligned}
		&\int_0^{T}   \big|\big ( \omega     \alpha_j t^{\de_j-1}\partial_{v_1} H_{\delta_i}^{\alpha_i} H_{\delta_j}^{\alpha_j-1} g,\ \omega \vec H^{\alpha}g\big )_{\mathcal{H}^3_xL^2_v}\big|dt\\
	&	\leq \eps  \int_0^{T}\norm{\comi{D_v}^\tau (\omega \vec H^{\alpha} g)}_{\mathcal{H}^3_xL^2_v}^2 dt\\
	&\quad+    \int_0^{T}\eps^{-1}  \alpha_j^2 t^{2(\delta_j-1)}  \norm{\comi{D_v}^{-\tau}  (\omega  \partial_{v_1}H_{\delta_i}^{\alpha_i} H_{\delta_j}^{\alpha_j-1}  g)}_{\mathcal{H}^3_xL^2_v}^2dt,
	\end{aligned}
\end{equation}
recalling  $\tau=\frac{2s}{2-\ga}$ is defined in \eqref{tau}. 
The assumption that   $\frac{\ga}{2}+2s<1$  implies $\tau<\frac12$.  This enables us to  
 use  the interpolation inequality that 
$$
 \forall\ \tilde\eps>0,\quad \|\comi {D_v}^{-\tau}h\|^2_{L_v^{2}}\leq \tilde\eps  \|\comi {D_v}^{\tau} h\|^{2}_{L_v^2}+\tilde\eps^{-\frac{1-2\tau}{2\tau}}\|\comi {D_v}^{\tau-1}h\|^{2}_{L_v^2},
$$
with $\tilde \eps =\eps^2 t^{2\delta_1}\alpha_j^{-2}  t^{-2(\delta_j-1)}$ and $h= \omega  \partial_{v_1}H_{\delta_i}^{\alpha_i} H_{\delta_j}^{\alpha_j-1}  g;$ this gives
\begin{equation}
\begin{aligned}\label{eqpre}
	&\eps^{-1}  \alpha_j^2 t^{2(\delta_j-1)}  \norm{\comi{D_v}^{-\tau}  (\omega  \partial_{v_1}H_{\delta_i}^{\alpha_i} H_{\delta_j}^{\alpha_j-1}  g)}_{\mathcal{H}^3_xL^2_v}^2\\
	&\leq  \eps  t^{2\delta_1}\norm{\comi{D_v}^{\tau} (\omega   \partial_{v_1}H_{\delta_i}^{\alpha_i} H_{\delta_j}^{\alpha_j-1}  g)}_{\mathcal{H}^3_xL^2_v}^2\\
	&\quad+\eps^{\frac{\tau-1}{\tau}}\alpha_j^{\frac{1}{\tau}}  t^{\frac{1}{\tau}(\delta_j-1)} t^{-\frac{1 }{\tau}(1-2\tau)\delta_1}\norm{\comi{D_v}^{\tau-1}  (\omega   \partial_{v_1}H_{\delta_i}^{\alpha_i} H_{\delta_j}^{\alpha_j-1}  g)}_{\mathcal{H}^3_xL^2_v}^2.
\end{aligned}
\end{equation}

As for the last term on the right-hand side of \eqref{eqpre}, we use    definition  \eqref{de1de2}  of $\delta_j$   and  the fact that $\delta_1>\delta_2$, to compute, for $j=1, 2,$
 \begin{align*}
	\delta_j-1-(1-2\tau)\delta_1 \geq \delta_2-1-(1-2\tau)\delta_1\geq   (1-2\tau)\lambda-(1-2\tau)\lambda \geq 0.
	\end{align*}
This yields that, for any $t\in [0, T]$ with $T\leq 1$,
\begin{equation}\label{intersec}
\begin{aligned}
	&\eps^{\frac{\tau-1}{\tau}}\alpha_j^{\frac{1}{\tau}}  t^{\frac{1}{\tau}(\delta_j-1)} t^{-\frac{1 }{\tau}(1-2\tau)\delta_1}\norm{\comi{D_v}^{\tau-1}  (\omega   \partial_{v_1}H_{\delta_i}^{\alpha_i} H_{\delta_j}^{\alpha_j-1} g)}_{\mathcal{H}^3_xL^2_v}^2\\
	&\leq C_\eps \alpha_j^{\frac{1}{\tau}}   \norm{\comi{D_v}^{\tau-1}  (\omega  \partial_{v_1}H_{\delta_i}^{\alpha_i} H_{\delta_j}^{\alpha_j-1}  g)}_{\mathcal{H}^3_xL^2_v}^2\\
	&\leq C_\eps \alpha_j^{\frac{1}{\tau}}   \norm{\comi{D_v}^{\tau }  (\omega   H_{\delta_i}^{\alpha_i} H_{\delta_j}^{\alpha_j-1} g)}_{\mathcal{H}^3_xL^2_v}^2+C_\eps \alpha_j^{\frac{1}{\tau}}   \norm{\comi{v}   \omega   H_{\delta_i}^{\alpha_i} H_{\delta_j}^{\alpha_j-1}  g}_{\mathcal{H}^3_xL^2_v}^2\\
	&\leq C_\eps \alpha_j^{\frac{1}{\tau}} \Big(  \normm{  \omega   H_{\delta_i}^{\alpha_i} H_{\delta_j}^{\alpha_j-1}  g}_{\mathcal{H}^3_x}^2+  \norm{\comi{v}   \omega   H_{\delta_i}^{\alpha_i} H_{\delta_j}^{\alpha_j-1}  g}_{\mathcal{H}^3_xL^2_v}^2\Big),
\end{aligned}
\end{equation}
where in the second equality we used the fact that $\tau<\frac12$ and $|\partial_{v_1}\omega|\leq C\comi v\omega$ in view of \eqref{vweight}, and the last line follows from Corollary \ref{corollary:coer}.  

As for the first term on the right-hand side of \eqref{eqpre}, we use the fact that
\begin{eqnarray*}
	 t^{\delta_1}\partial_{v_1}=-\frac{\delta_1+ 1}{\delta_2-\delta_1} H_{\delta_1}+\frac{\delta_2+ 1}{\delta_2-\delta_1}t^{\delta_1-\delta_2}H_{\delta_2}
\end{eqnarray*}
in view of 
\eqref{generate},  to conclude that 
\begin{equation}\label{mixedder}
\begin{aligned}
	& \eps  t^{2\delta_1}\norm{\comi{D_v}^{\tau} (\omega   \partial_{v_1}H_{\delta_i}^{\alpha_i} H_{\delta_j}^{\alpha_j-1}  g)}_{\mathcal{H}^3_xL^2_v}^2= \eps  \norm{\comi{D_v}^{\tau} (\omega t^{\delta_1} \partial_{v_1}H_{\delta_i}^{\alpha_i} H_{\delta_j}^{\alpha_j-1}  g)}_{\mathcal{H}^3_xL^2_v}^2\\
	& \leq \eps C\Big(\norm{\comi{D_v}^{\tau} (\omega H_{\delta_1}H_{\delta_i}^{\alpha_i} H_{\delta_j}^{\alpha_j-1}  g)}_{\mathcal{H}^3_xL^2_v}^2+\norm{\comi{D_v}^{\tau} (\omega H_{\delta_2}H_{\delta_i}^{\alpha_i} H_{\delta_j}^{\alpha_j-1}  g)}_{\mathcal{H}^3_xL^2_v}^2\Big). 
	 \end{aligned} 
\end{equation}
Substituting the above estimate and \eqref{intersec} into \eqref{eqpre} yields that, for any    $\alpha\in\mathbb Z_+^2$  with $\abs\alpha=k,$
\begin{align*}
	&\eps^{-1}  \alpha_j^2 t^{2(\delta_j-1)}  \norm{\comi{D_v}^{-\tau}  (\omega  \partial_{v_1}H_{\delta_i}^{\alpha_i} H_{\delta_j}^{\alpha_j-1}  g)}_{\mathcal{H}^3_xL^2_v}^2\\
	&\leq \eps C\Big(\norm{\comi{D_v}^{\tau} (\omega H_{\delta_1}H_{\delta_i}^{\alpha_i} H_{\delta_j}^{\alpha_j-1}  g)}_{\mathcal{H}^3_xL^2_v}^2+\norm{\comi{D_v}^{\tau} (\omega H_{\delta_2}H_{\delta_i}^{\alpha_i} H_{\delta_j}^{\alpha_j-1}  g)}_{\mathcal{H}^3_xL^2_v}^2\Big)\\
	&\qquad+ C_\eps \alpha_j^{\frac{1}{\tau}} \Big(  \normm{  \omega   H_{\delta_i}^{\alpha_i} H_{\delta_j}^{\alpha_j-1}  g}_{\mathcal{H}^3_x}^2+  \norm{\comi{v}   \omega   H_{\delta_i}^{\alpha_i} H_{\delta_j}^{\alpha_j-1}  g}_{\mathcal{H}^3_xL^2_v}^2\Big).
\end{align*}
We integrate the above estimate over $[0,T]$ and then use 
  Corollary \ref{corollary:coer};  this implies that,  for any $\eps>0$ and any $\alpha\in\mathbb Z_+^2$ with $\abs\alpha=k$, 
\begin{multline*}
	 \bigg[ \int_0^{T}\eps^{-1}  \alpha_j^2 t^{2(\delta_j-1)}  \norm{\comi{D_v}^{-\tau} ( \omega  \partial_{v_1}H_{\delta_i}^{\alpha_i} H_{\delta_j}^{\alpha_j-1}  g)}_{\mathcal{H}^3_xL^2_v}^2dt\bigg]^{1\over2}\\
	  \leq \eps \sup_{\abs\alpha=k}\bigg[ \int_0^{T}\Big(  \normm{  \omega  \vec H^\alpha g}_{\mathcal{H}^3_x}^2+  \norm{\comi{v}   \omega    \vec H^\alpha  g}_{\mathcal{H}^3_xL^2_v}^2\Big)dt\bigg]^{1\over2}\\
	  +C_\eps k^{\frac{1}{2\tau}}  \sup_{\abs\beta=k-1} \bigg[\int_0^{T}\Big(  \normm{  \omega  \vec H^\beta g}_{\mathcal{H}^3_x}^2+  \norm{\comi{v}   \omega    \vec H^\beta  g}_{\mathcal{H}^3_xL^2_v}^2\Big) dt\bigg]^{1\over2}.
\end{multline*}
This with  \eqref{tranest} and Corollary \ref{corollary:coer}
 yields that, for any $\eps>0$ and  any $\alpha\in\mathbb Z_+^2$ with $\abs\alpha=k,$
 \begin{multline*}
 	\bigg[\int_0^{T}   \big|\big ( \omega     \alpha_j t^{\de_j-1}\partial_{v_1} H_{\delta_i}^{\alpha_i} H_{\delta_j}^{\alpha_j-1} g,\ \omega \vec H^{\alpha} g\big )_{\mathcal{H}^3_xL^2_v}\big|dt\bigg]^{1\over2}\\
 \leq 	\eps \sup_{\abs\alpha=k} \bigg[\int_0^{T}\Big(  \normm{  \omega  \vec H^\alpha g}_{\mathcal{H}^3_x}^2+  \norm{\comi{v}   \omega    \vec H^\alpha  g}_{\mathcal{H}^3_xL^2_v}^2\Big)dt\bigg]^{1\over2}\\
	  +C_\eps k^{\frac{1}{2\tau}} \sup_{\abs\beta=k-1} \bigg[\int_0^{T}\Big(  \normm{  \omega  \vec H^\beta g}_{\mathcal{H}^3_x}^2+  \norm{\comi{v}   \omega    \vec H^\beta g}_{\mathcal{H}^3_xL^2_v}^2\Big) dt\bigg]^{1\over2}.
 \end{multline*}
  Then we substitute  the above estimate into \eqref{supal} and then choose $\eps >0$ small sufficiently, this gives, for $k\geq 2,$ 
\beno
	&& \sup_{\abs\alpha=k}\  \sup_{t\leq T}  \|\omega   \vec{ H}^\alpha g(t)\|_{\mathcal{H}^3_xL^2_v}+ \sup_{\abs\alpha=k} \bigg[\int_0^{T} \Big(\normm{\omega    \vec{ H}^\alpha g}_{\mathcal H_x^3}^2 + \|\<v\>\omega  \vec{ H}^\alpha g\|^2_{\mathcal{H}^3_xL^2_v}\Big)dt\bigg]^{1\over2}\\
	&  \leq&  
	 C \sup_{\abs\alpha=k}\   \sup_{t\leq T}  \|\omega   \vec H^{\alpha} f \|_{\mathcal{H}^3_xL^2_v} \bigg[\int_0^T \Big(  \normm{\omega g}_{\mathcal H_x^{3}}^2+\norm{\comi v\omega g}_{\mathcal H_x^{3} L_v^2}^2\Big) dt\bigg]^{1\over2} \\
  \eeno\ben\label{s1s}
	&& +    C  k^{\frac{1}{2\tau}} \sup_{\abs\beta=k-1} \bigg[ \int_0^{T}\Big(  \normm{  \omega  \vec H^\beta g}_{\mathcal{H}^3_x}^2+  \norm{\comi{v}   \omega    \vec H^\beta  g}_{\mathcal{H}^3_xL^2_v}^2\Big) dt \bigg]^{1\over2}\\
	&&\notag +C \sup_{\abs\alpha=k} \bigg[\sum_{\stackrel{ \beta\leq \alpha}{ 1\leq \abs\beta \leq k-1}} {\alpha\choose \beta}\int_0^{T} \big| \big(\omega Q ( \vec{ H}^\beta f,\  \vec{ H}^{\alpha-\beta} g), \ \omega H_\delta^{k} g\big)_{\mathcal{H}^3_xL^2_v}\big|dt\bigg]^{1\over2},
 \een
which corresponds to estimate \eqref{++energy} in the case of $\frac{\gamma}{2}+2s\geq 1$. Note the main  difference between \eqref{++energy} and \eqref{s1s}  arises from  the second terms on the right-hand sides, and   here we have the factor  $k^{\frac{1}{2\tau}}$ in \eqref{s1s}.  So that in this case of $\frac{\gamma}{2}+2s< 1$,  only Gevrey regularity with index $\frac{1}{2\tau}$ may be expected.  For the last term in \eqref{s1s},  it   
  can be treated in a similar way to  that  after  \eqref{wq} by modifying the argument in \eqref{comp}.  Precisely,  corresponding to  \eqref{max} and \eqref{eng} we have
 that, for any $\abs \beta  \geq 1,$
\begin{equation*}
  \sup_{t\leq T} \norm{ \omega \vec H^\beta f}_{ \mathcal H_x^3  L_v^2} \leq \abs{f}_{X_{\rho,T}}\frac{ \rho^{\abs\beta-1} (\abs\beta !)^{(2\tau)^{-1}} }{(\abs\beta+1)^3}  
\end{equation*}
and
\begin{equation*}
 \bigg[\int_0^{T}\Big(  \normm{ \omega \vec H^\beta g}_{\mathcal H_x^3 }^2+ \norm{ \comi v\omega  \vec H^\beta g}_{\mathcal H_x^3L_v^2}^2\Big)dt\bigg]^{1\over2}
	 \leq \abs{g}_{Y_{\rho,T}}\frac{ \rho^{ \abs\beta-1} (\abs\beta!)^{(2\tau)^{-1}} }{(\abs\beta+1)^3}.
\end{equation*} 
This with the fact that
\begin{equation*}
	{\alpha\choose\beta}\leq {{\abs\alpha}\choose {\abs\beta}} 
	\end{equation*}
enables  us to follow the computation in   \eqref{comp}, to obtain that, for any  fixed $\alpha\in\mathbb Z_+^2$ with $\abs\alpha=k\geq 2,$ 
\begin{align*}
	& \sum_{\stackrel{ \beta\leq \alpha}{ 1\leq \abs\beta \leq k-1}} {\alpha\choose \beta}  \sup_{t\leq T} \|\om \vec H^\beta  f\|_{\mathcal{H}^3_xL^2_v}\bigg[\int_0^{T}\Big(\normm{\om \vec H^{\alpha-\beta}  g}_{\mathcal{H}^3_x}^2+\|\comi v\om \vec H^{\alpha-\beta} g\|_{\mathcal{H}^3_xL^2_v}^2\Big) dt\bigg]^{1\over2}\\
	 &\leq  C \abs{f}_{X_{\rho,T}} \abs{g}_{Y_{\rho,T}}\sum_{\stackrel{ \beta\leq \alpha}{ 1\leq \abs\beta \leq k-1}} \frac{|\alpha|!}{\abs\beta!(\abs\alpha-\abs\beta)!}\frac{ \rho^{\abs\beta-1}|\beta|!^{\frac{1}{2\tau}}}{(\abs\beta+1)^3}\frac{ \rho^{|\alpha|-\abs\beta-1}(|\alpha|-\abs\beta)!^{\frac{1}{2\tau}}}{(|\alpha|-\abs\beta+1)^3}\\
	 &\leq  \frac{C }{\rho} \abs{f}_{X_{\rho,T}} \abs{g}_{Y_{\rho,T}} \frac{  \rho^{|\alpha|-1}|\alpha|!^{\frac{1}{2\tau}}}{(|\alpha|+1)^3}\sum_{\stackrel{ \beta\leq \alpha}{ 1\leq \abs\beta \leq k-1}}\frac{ (|\alpha|+1)^3}{(\abs\beta+1)^3(|\alpha|-\abs\beta+1)^3}\\
	 &\leq   \frac{C}{\rho}\abs{f}_{X_{\rho,T}} \abs{g}_{Y_{\rho,T}} \frac{ \rho^{|\alpha|-1}|\alpha|!^{\frac{1}{2\tau}}}{(|\alpha|+1)^3},
\end{align*} 	
where in the second inequality   we use the facts that $m!^{\frac{1}{2\tau}-1} n!^{\frac{1}{2\tau}-1}\leq (m+n)!^{\frac{1}{2\tau}-1}$ and the last line follows from the fact that 
\begin{align*}
\sum_{\stackrel{ \beta\leq \alpha}{ 1\leq \abs\beta \leq k-1}}\frac{ (|\alpha|+1)^3}{(\abs\beta+1)^3(|\alpha|-\abs\beta+1)^3}\leq \sum_{\beta\in\mathbb Z_+^2}\frac{C}{(\abs\beta+1)^3}\leq \sum_{j=1}^{+\infty}\sum_{\stackrel{\beta\in\mathbb Z_+^2}{ \abs\beta=j}}\frac{C}{ \abs\beta ^3}\leq  \sum_{j=1}^{+\infty}\frac{C}{j^2}\leq C. 
\end{align*} 
The rest part for proving  \eqref{mainquan+} is quite similar to that in the previous Section \ref{sec:analy}. So we omit the detail for brevity.   Thus the proof of Proposition \ref{thm:sharpgev}  is completed.
    \end{proof}

\section{Linear Boltzmann equation}\label{sec:hinfty} 

This part is devoted to   proving     the existence, uniqueness and  smoothing effect existence of solutions to the following linear Boltzmann equation 
\begin{equation}\label{lincau+}
	\partial_tg+v\cdot\partial_xg=Q(f,g), \quad g|_{t=0}=f_{in},
\end{equation}
where $f_{in}$ satisfies the assumption in Theorem \ref{thm:Gevrey}. Recall the sets of $\mathcal A_T, \mathcal M_T(C_0)$ and  $\mathcal N_{T}(\rho,C_*)$ are given in Definitions \ref{def:c0mt} and \ref{defm}, that is,
\begin{equation*}
	\left\{
	\begin{aligned}
		& \mathcal A_T=\mathcal A_T(m_0,M_0,E_0,H_0)=\big\{f;\ \ f\geq 0\  \textrm{ and }\  f \ \textrm{satisfies} \   \eqref{aat}\big\},\\
		&\mathcal M_T(C_0)=\Big\{ f;\ \  \sup_{ t\leq T}\norm{\omega f}_{H_x^3 L_v^2 } \leq C_0 \Big\},\\
		&\mathcal N_{T}(\rho,C_*)=\big\{f; \  \    \abs{f}_{X_{\rho,T}}+\abs{h}_{Y_{\rho,T}} \leq C_*\big\}.
	\end{aligned}
	\right. 
\end{equation*}
With the spaces defined above, the main result of this section can be stated as follows.
 
 \begin{thm}\label{thm:lin}
 Suppose the initial datum $f_{in}\geq 0$ in 
\eqref{lincau+} satisfies the same assumption as that in Theorem \ref{thm:Gevrey}.  
Recall $C_0, C_1$ and $C_2$ are the constants  given in \eqref{c0} and \eqref{lower}.  Then there exists  a time $T_*>0$ and   three constants $\rho, C_*$ and $\Theta$,   all depending only on $C_0, C_1$ and $C_2,$     such that  for any given $f\in\mathcal  A_{T}\cap \mathcal M_{T}(C_0)\cap \mathcal N_{T}(\rho, C_*)$ with $T\leq T_*$, 
  the linear Boltzmann equation \eqref{lincau+} admits a local solution 
$
  	 g\in L^\infty([0,T]; \mathcal H_x^3 L_v^3),
$
 satisfying that
 \begin{equation}\label{thm611}
	\left\{
	\begin{aligned}
		& \abs{g}_{X_{\rho, T}}+\abs{g}_{Y_{\rho, T}}
<+\infty,\\
& 	\lim_{t\rightarrow 0} \norm{\omega {\bm D}^k g(t) }_{ \mathcal H_x^3  L_v^2} =0\  \textrm{ for any } \ k\geq 1,
	\end{aligned}
	\right. 
\end{equation}
with $\norm{\omega {\bm D}^k g}_{ \mathcal H_x^3  L_v^2}$ defined by \eqref{dk}. 
 Moreover,   if
 \begin{equation*}
 	 \bigg[\int_0^T \big( \normm{\omega  g}_{ \mathcal H_x^3   }^2  +\norm{ \comi{v}\omega g}_{ \mathcal H_x^3  L_v^2} ^2 \big)dt\bigg]^{1\over2} \leq \Theta,
 \end{equation*}
  then  $g\in \mathcal A_T\cap \mathcal M_T(C_0)\cap  \mathcal N_{T}(\rho, C_*).$
 \end{thm}

{\bf Methodology.}   We will use standard iteration scheme to prove Theorem \ref{thm:lin}. To clarify the proof we list the outline as follows.

\underline{\it Step 1.} 	To avoid     formal computations,  we first consider a linear parabolic Cauchy problem  which can be regarded as a regularized version of the Boltzmann equation. Precisely,    let $0<\epsilon\ll 1$ be given small parameter  and let $h$ be a given function. Consider the following linear parabolic equation:
\begin{equation}
	\label{linparabolic}
	\begin{aligned}
		 \partial_t g^\epsilon  + v\cdot\pa_x g^\epsilon  +\epsilon\big(\comi v^{\frac{2 }{1-s}}- \Delta_{x,v}\big) g^\epsilon  =Q(f, h),\quad 
		g^\epsilon|_{t=0}=f_{in},
	\end{aligned}
		\end{equation}
where    $g^\epsilon$ is the unknown function. For the Cauchy problem \eqref{linparabolic},  we may take advantage of the classical existence and regularity theory for parabolic equations, to obtain the local  existence in 
  $L^\infty([0,T_\epsilon];  H^3_{x}L_v^2)\cap L^2([0,T_\epsilon];  H^4_{x} H_v^1)$ for some $\epsilon$-dependent lifespan $T_\epsilon$. Moreover,  similar to the smoothing effect for the heat equation,  the solution $g^\epsilon$ to \eqref{linparabolic} will become $H^{+\infty}$ smoothness for $0<t\leq T_\epsilon$.

 \underline{\it Step 2.} In the previous step, the lifespan $T_\epsilon$ may dependent on $\epsilon$ but is independent of  the given function $h$ in \eqref{linparabolic}.  This ensures the common lifespan $T_\epsilon$ for the iteration:
  \begin{equation}
	\label{itera1}
	\begin{aligned}
		 \partial_t g^{n}  + v\cdot\pa_x g^n  +\epsilon\big(\comi v^{\frac{2 }{1-s}}- \Delta_{x,v}\big) g^n  =Q(f,g^{n-1}),\quad 
		  g^n|_{t=0}=f_{in}.
	\end{aligned}
		\end{equation}
 As shown in the previous step, $\{g^n\}_{n\geq 0}$ is  a sequence of regular functions lying in $H^{+\infty}$ space at positive times. Moreover, we can show that $\{g^n\}_{n\geq 0}$ is Cauchy  sequence in the energy space which consists of $H^{+\infty}$-smooth functions at positive times.  So that the limit, denoted by $ g^\epsilon$,   belongs to the same space and thus is smooth at positive times. By letting $n\rightarrow +\infty$ in \eqref{itera1}, we see the limit $g^\epsilon$ will solve the equation
 \begin{equation}\label{app}
 \begin{aligned}
 	 \partial_t g^\epsilon +v\cdot\partial_xg^\epsilon+\epsilon\big(\comi v^{\frac{2 }{1-s}}- \Delta_{x,v}\big) g^\epsilon =Q(f, g^\epsilon),\quad g^\epsilon|_{t=0}=f_{in}.
 \end{aligned}
\end{equation}
 The key part is to remove the $\epsilon$-dependence of the lifespan $T_\epsilon$ and get the uniform estimate for $g^\epsilon$ with respect to $\epsilon.$

\underline{\it Step 3.} By compactness argument,   the family of solutions $g^{\epsilon}, 0<\epsilon\ll1,$ to \eqref{app}  contains a convergent subsequence and thus  its limit $g$, admitting the same regularity to that of $g^\epsilon$, will solve the linear Boltzmann equation 
   \begin{equation}\label{licau}
   \begin{aligned}
  		\partial_tg+v\cdot\partial_x g-Q(f,g)=0,\quad  g|_{t=0}=f_{in},
   \end{aligned}
  \end{equation}
  The crucial part here is to  show that $g$ is non-negativity and has bounded mass, energy and entropy.  
 
\subsection{Existence and smoothing effect for  the regularized  equation} 
In this part we study the regularized equation \eqref{linparabolic}, that is,
 \begin{equation}\label{regularied}
   \partial_t g^\epsilon  + v\cdot\pa_x g^\epsilon  +\epsilon\big(\comi v^{\frac{2 }{1-s}}- \Delta_{x,v}\big) g^\epsilon  =Q(f, h),\quad  g^\epsilon|_{t=0}=f_{in},
  \end{equation}  
  where $f,h$ are given functions.   
  Corresponding to the space $Y_{\rho, T}$ in Definition \ref{assum}, we introduce the following energy function space  for the above regularized equation.   

  \begin{definition}\label{assum+}  Given  small parameter $0<\epsilon\ll 1$, we   introduce   a Banach space $ Z_{\rho, T, \epsilon}$, equipped  with the   norm   $ \abs{\cdot}_{Z_{\rho,T,\epsilon}}$  defined  as below.   If  ${\gamma \over 2}+2s\geq 1$  we define 
  \begin{equation*} 
  \begin{aligned}
 &\abs{h}_{Z_{\rho,T,\epsilon}}  =  \sum_{1\leq j\leq3}\sum_{1\leq i\leq 2}\, \sup_{k\geq 0}\bigg[ L_{\rho,k}  \Big(   \int_0^T  \norm{ \comi{v}\omega H_{\delta_i,j}^k  h}_{ \mathcal H_x^3  L_v^2} ^2  dt \Big)^{1\over2} \bigg]\\
 &\ + \sum_{1\leq j\leq 3}\sum_{1\leq i\leq 2}\sup_{k\geq0}\bigg[ L_{\rho,k}   \bigg(\int_0^T \epsilon\Big(\norm{ \partial_{x,v}\omega H_{\delta_i,j}^k h}_{ \mathcal H_x^3  L_v^2}^2  +\norm{ \comi{v}^{\frac{1}{1-s}} \omega  H_{\delta_i,j}^k  h}_{ \mathcal H_x^3  L_v^2} ^2 \Big)dt\bigg) ^{1\over2}  \bigg] .  
 \end{aligned}
 \end{equation*} 
 If  ${\gamma \over 2}+2s< 1$  we define 
  \begin{equation*}
  \begin{aligned}
&\abs{h}_{Z_{\rho,T,\epsilon}} =  \sum_{ 1\leq j\leq 3}    \sup_{k\geq0}	\bigg[L_{\rho,k}\sup_{\abs\alpha=k}\bigg(\int_0^T  \norm{\comi v\omega \vec{ H}_{j}^\alpha h }_{ \mathcal H_x^3  L_v^2}^2 dt\bigg)^{1\over2} \bigg]\\
&\quad+ \sum_{ 1\leq j\leq 3}    \sup_{k\geq0}	\bigg[L_{\rho,k}\sup_{\abs\alpha=k}\bigg(\int_0^T\epsilon\big(\norm{ \partial_{x,v}\omega \vec{ H}_{j}^\alpha h }_{ \mathcal H_x^3L_v^2}^2+ \norm{\comi v^{\frac{1}{1-s}}\omega \vec{ H}_{j}^\alpha h }_{ \mathcal H_x^3  L_v^2}^2\big)dt\bigg)^{1\over2} \bigg].
\end{aligned}
 \end{equation*}
Here we used  the notations in \eqref{lrhok} and  \eqref{pkh}. 
 \end{definition} 
   
 \begin{rmk} The
 three norms in Definitions \ref{assum} and \ref{assum+} are linked by the following estimate:  for any $\tilde\eps >0$, 
 \begin{equation}\label{conclu}
	\abs{h}_{Y_{\rho,T}}   \leq \tilde\eps \abs{h}_{Z_{\rho,T,\epsilon}} +C_{\epsilon, \tilde\eps} T^{1\over2} \abs{h}_{X_{\rho,T}}, 
\end{equation}
where $C_{\epsilon, \tilde\eps} $  is a constant depending on $\tilde\eps $ and the small parameter $\epsilon$. To prove \eqref{conclu}, we use the interpolation inequality \eqref{tau -j}  in Appendix \ref{sec:appinter},    to conclude  that,  for any $\tilde\eps >0$,
\begin{equation}\label{noreps}
\begin{aligned}
	\normm{h}^2&\leq C\norm{\comi v^s \comi{D_v}^s h}_{L^2_v}^2 \leq \tilde \eps  \norm{ \comi{D_v}  h}_{L^2_v}^2 +C_{\tilde \eps}\norm{ \comi{ v}^{\frac{s}{1-s}} h}_{L^2_v}^2\\
&\leq  \tilde \eps \big(  \norm{ \comi{D_v}  h}_{L^2_v}^2+\norm{ \comi{ v}^{\frac{1}{1-s}} h}_{L^2_v}^2 \big)+C_{\tilde \eps}\norm{ h}_{L^2_v}^2.
\end{aligned}
\end{equation}
Then  assertion \eqref{conclu} follows in view of Definitions \ref{assum} and \ref{assum+}.    

 \end{rmk}

By direct verification, we can find   a    constant $R_s>0$ depending only on $s,$ such that
\begin{equation}\label{rs}
\forall\  j\geq 1, \ 	\forall \ v\in\mathbb R^3,\quad \sum_{1\leq i\leq 3 } \big|\partial_{v_i}^j\comi v^{\frac{2 }{1-s}}\big|\leq \comi v^{\frac{2 }{1-s}-1} R_s^{j}j!.
\end{equation}

\begin{prop}\label{prp:linpra}  
Let $X_{\rho,T}, Y_{\rho,T}$ and    $Z_{\rho,T,\epsilon}$ be  given in Definitions  \ref{assum} and \ref{assum+}. Suppose that $f$ satisfies   \begin{equation}\label{endif}
  	\abs{  f}_{X_{\rho,T}}\leq   C_*  
  \end{equation} 
for  some   $T,  C_*>0$ and some $\rho\geq 2R_s$ with $R_s$ given in \eqref{rs}.  
 Then there exists a small $T_\epsilon$, depending only on $\epsilon $  and $ C_*$ in \eqref{endif} as well as the constants $C_1, C_2$ in \eqref{lower},  such that
 if the given function $h$ in \eqref{regularied} satisfies  
 \begin{equation}\label{ah}
  \abs { h}_{X_{\rho,T_\epsilon}}+   \abs { h}_{Z_{\rho,T_\epsilon,\epsilon}}<+\infty,
 \end{equation}
 then the linear parabolic equation \eqref{regularied} 
 admits a   solution $g^\epsilon\in L^\infty([0,T_\epsilon];  \mathcal H_x^3 L_v^2)$ satisfying that
\begin{equation*}
	 \abs{  g^\epsilon}^2_{X_{\rho,T_\epsilon}}+ \abs{  g^\epsilon}^2_{Z_{\rho,T_\epsilon,\epsilon}} \leq 12 \norm{e^{a_0\comi v^2}f_{in}}_{\mathcal H_x^3 L_v^2}^2+ \frac{1}{2} \big( \abs{ h}_{X_{\rho,T_\epsilon}}^2+ \abs{h}_{Z_{\rho,T_\epsilon,\epsilon}}^2\big),
\end{equation*}	
and that
\begin{equation*}
\forall\  k\geq 1,\quad 	\lim_{t\rightarrow 0} \norm{\om{\bm D}^k g^\epsilon}_{\mathcal H_x^3 L_v^2}=0.
\end{equation*}
Recall $\norm{\om{\bm D}^k g^\epsilon}_{\mathcal H_x^3 L_v^2}$ is defined in \eqref{dk}. 
\end{prop}

The proof of Proposition \ref{prp:linpra} relies on the classical existence theory for parabolic equations, since  the source term $Q(f,h)$ in \eqref{regularied} can be controlled in terms  of the diffusion part on the left-hand side. In fact,  by Corollary \ref{cor:upper}, 
\begin{multline*}
	\big|\big( \omega Q(f,   h),\  \omega g \big)_{\mathcal H_x^3 L_v^2}\big|\\
	 \leq \epsilon \big( \normm{\omega g  }_{\mathcal H_x^3 }^2+\norm{\comi v\omega g }_{\mathcal H_x^3 L_v^2}^2\big) + C \epsilon^{-1} \norm{\omega f}_{\mathcal H_x^3 L_v^2}^2\big( \normm{\omega h}_{\mathcal H_x^3 }^2+\norm{\comi v \omega h}_{\mathcal H_x^3 L_v^2}^2\big).
\end{multline*}
 This with \eqref{noreps} and the fact   $\abs{f}_{X_{\rho, T}}\leq   C_*$ yields  that,  for any $\tilde\eps>0,$
\begin{equation}\label{aprie}
\begin{aligned}
	& \big|\big( \omega Q(f,   h),\  \omega g \big)_{\mathcal H_x^3 L_v^2}\big| \leq \frac{\epsilon}{4}  \Big( \norm{  \partial_{x,v}  \omega    g  }_{\mathcal H_x^3 L_v^2}^2+   \norm{\comi v^{\frac{1}{1-s}} \omega   g }_{\mathcal H_x^3 L_v^2}^2\Big)+C\epsilon \norm{ \omega   g }_{\mathcal H_x^3 L_v^2}^2\\
	&\qquad+ C C_*^2 \epsilon^{-1} \Big(  \tilde \eps  \norm{  \partial_{x,v}  \omega   h }_{\mathcal H_x^3 L_v^2}^2+  \tilde \eps   \norm{\comi v^{\frac{1}{1-s}} \omega   h}_{\mathcal H_x^3 L_v^2}^2+C_{\tilde\eps}  \norm{ \omega   h}_{\mathcal H_x^3 L_v^2}^2\Big),
	\end{aligned}
\end{equation}
 where the constant $ C_{\tilde\eps }$ depends only on $\tilde\eps.$ 
 
 \begin{lem}[Existence of weak solutions]\label{lem:ws}
 Under the hypothesis of Proposition \ref{prp:linpra}, we can find
  a small  $T_\epsilon>0$, depending only on $\epsilon$ and $  C_*$ in \eqref{endif},   such that the linear parabolic equation  \eqref{regularied} admits a local solution $g^\epsilon\in L^\infty([0,T_\epsilon];  \mathcal H_x^3 L_v^2)$ satisfying  that
  \begin{equation}
  	\label{fenergy+}
  	\begin{aligned}
&\sup_{t\leq T_\epsilon} \norm{\omega  g^{\eps}}_{\mathcal H_x^3 L_v^2} \\
&\qquad+ \bigg[   \int_0^{T_\epsilon}\Big(\norm{\comi v \omega   g^\eps }_{\mathcal H_x^3 L_v^2}^2+  \epsilon \norm{  \partial_{x,v} (\omega    g )}_{\mathcal H_x^3 L_v^2}^2+ \epsilon    \norm{\comi v^{\frac{1}{1-s}} \omega   g }_{\mathcal H_x^3 L_v^2}^2\Big)dt\bigg]^{1\over2}\\
&\leq  \norm{e^{a_0\comi v^2}f_{in}}_{\mathcal H_x^3 L_v^2}+\frac1{24} \big(\abs{h}_{X_{\rho,T_\epsilon}}+\abs{h}_{Z_{\rho,T_\epsilon,\epsilon}}\big). 	
  	\end{aligned}
  \end{equation}
 \end{lem}
 
 \begin{proof} Taking $\mathcal H_x^3 L_v^2$ inner product with $\omega  g^\eps$ on both sides of the equation
 \begin{equation*}
     (\partial_t+v\cdot\partial_x +\comi v^2) \omega g^\epsilon=\omega Q(f,h)
 \end{equation*}
 and then using \eqref{aprie}, we have that,     observing $\epsilon\ll 1$,    
\begin{equation*}
\begin{aligned}
	&\frac12 \frac{d}{dt}\norm{\omega  g^\eps }_{\mathcal H_x^3 L_v^2}^2+\frac12 \norm{\comi v \omega   g^\eps }_{\mathcal H_x^3 L_v^2}^2+ \frac12\epsilon \Big( \norm{  \partial_{x,v} (\omega    g )}_{\mathcal H_x^3 L_v^2}^2+   \norm{\comi v^{\frac{1}{1-s}} \omega   g }_{\mathcal H_x^3 L_v^2}^2\Big) \\
		&\leq \frac{1}{96}\epsilon  \Big(   \norm{  \partial_{x,v} (\omega   h)}_{\mathcal H_x^3 L_v^2}^2+    \norm{\comi v^{\frac{1}{1-s}} \omega   h}_{\mathcal H_x^3 L_v^2}^2\Big)+ C_{\epsilon}    \norm{ \omega   h}_{\mathcal H_x^3 L_v^2}^2,
		\end{aligned}
\end{equation*}
where the constant $C_\epsilon$ depends only on $\epsilon$ and $ C_*$ in \eqref{endif}. 
The {\it a priori} estimate above enables us to apply the classic existence   theory for parabolic equations, to conclude that  there exists a small constant $T_\epsilon$ depending only on $\epsilon,$ such that 
if $h$ satisfies  \eqref{ah}, then  the
 Cauchy problem \eqref{regularied} admits a   solution $g^\epsilon\in L^\infty([0,T_\epsilon]; \mathcal H_x^3 L_v^2)$ satisfying that
\begin{equation*}
\begin{aligned}
	&\sup_{t\leq T_\epsilon} \norm{\omega  g^{\eps}}_{\mathcal H_x^3 L_v^2} +\bigg(\int_0^{T_\epsilon} \norm{\comi v \omega   g^\eps }_{\mathcal H_x^3 L_v^2}^2 dt \bigg)^{1\over 2} \\
	&\qquad\qquad +  \bigg(\epsilon\int_0^{T_\epsilon}\big( \norm{  \partial_{x,v} (\omega    g )}_{\mathcal H_x^3 L_v^2}^2+   \norm{\comi v^{\frac{1}{1-s}} \omega   g }_{\mathcal H_x^3 L_v^2}^2\big)dt\bigg)^{1\over 2}\\ 
&	 \leq  \norm{e^{a_0\comi v^2}f_{in}}_{\mathcal H_x^3 L_v^2}+ C_{\epsilon}   T_\epsilon^{1\over 2}   \sup_{t\leq T_\eps } \norm{ \omega   h}_{\mathcal H_x^3 L_v^2}\\
&\qquad+\frac1{24}\bigg(\epsilon \int_0^{T_\epsilon}\big(   \norm{  \partial_{x,v} (\omega   h)}_{\mathcal H_x^3 L_v^2}^2+    \norm{\comi v^{\frac{1}{1-s}} \omega   h}_{\mathcal H_x^3 L_v^2}^2\big)dt\bigg)^{1\over2}.
	\end{aligned}
\end{equation*}
This yields  estimate \eqref{fenergy+} if we choose $T_\epsilon$  small enough such that 
$
	C_\epsilon   T_\epsilon^{1\over 2} \leq \f 1{24}.
$
The proof of Lemma \ref{lem:ws} is completed.  
 \end{proof}

\begin{lem}[Smoothing effect of weak solutions and short-time behaviour]\label{lem:smeshort} Let $\frac{\gamma}{2}+2s\geq 1,$ and 
	 let $g^\epsilon\in L^\infty([0,T_\epsilon];  \mathcal H_x^3 L_v^2)$ be the weak solution to \eqref{regularied} constructed in the previous part, satisfying estimate \eqref{fenergy+}.  Then
	  for each $k\geq 1, $
 \begin{equation}\label{kgeq1}
  \begin{aligned}
	&L_{\rho,k}  \sup_{t\leq T_\epsilon}\norm{\omega  H_\delta^kg^\epsilon}_{\mathcal H_x^3 L_v^2}+  L_{\rho,k} \bigg( \int_0^{T_\epsilon}   \norm{\comi v    \omega H_\delta^k g^\epsilon  }_{\mathcal H_x^3 L_v^2} ^2 dt\bigg)^{1\over2}\\
	&\qquad  +  L_{\rho,k} \bigg[ \int_0^{T_\epsilon}  \Big (   \epsilon\norm{  \partial_{x,v} (  \omega H_\delta^k g^\epsilon  )}_{\mathcal H_x^3 L_v^2}^2+   \epsilon \norm{\comi v^{\frac{1}{1-s}}   \omega  H_\delta^k g^\epsilon  }_{\mathcal H_x^3 L_v^2}^2\Big)dt\bigg]^{1\over2}\\
	&\leq \norm{e^{a_0\comi v^2}f_{in}}_{\mathcal H_x^3 L_v^2}+ \frac1{24} \big(\abs{h}_{X_{\rho,T_\epsilon}}+\abs{h}_{Z_{\rho,T_\epsilon,\epsilon}}\big), 
	\end{aligned}
	\end{equation}
where we used the notations in \eqref{lrhok}  and $\delta=\delta_1$ or $\delta=\delta_2$ with $ \delta_1, \delta_2$ defined in \eqref{de1de2}. Moreover,
\begin{equation*}
	\forall\ k\geq 1,\quad \lim_{t\rightarrow 0}\norm{\omega H_\delta^kg^\epsilon}_{\mathcal H_x^3 L_v^2}=0.
\end{equation*}
\end{lem}

 \begin{proof} The proof is divided into two steps, one devoted to the smoothing effect and another to the short-time behavior.

{\it Step 1 (Smoothing effect).} We will use induction on $k$ to prove \eqref{kgeq1}. The validity of \eqref{kgeq1} for $k=0$ just follows from \eqref{fenergy+}.  
 Now let $k\geq 1$ and  assume that, for any $ m\leq k-1$,
	\begin{equation}\label{min}
	\begin{aligned}
		&L_{\rho,m}  \sup_{t\leq T_\epsilon}\norm{\omega  H_\delta^m g^\epsilon}_{\mathcal H_x^3 L_v^2}+  L_{\rho,m} \bigg( \int_0^{T_\epsilon}   \norm{\comi v    \omega  H_\delta^m g^\epsilon  }_{\mathcal H_x^3 L_v^2} ^2 dt\bigg)^{1\over2}\\
	&\qquad  +  L_{\rho,m} \bigg[ \int_0^{T_\epsilon}  \epsilon \Big (  \norm{  \partial_{x,v} (  \omega  H_\delta^m g^\epsilon  )}_{\mathcal H_x^3 L_v^2}^2+   \norm{\comi v^{\frac{1}{1-s}}   \omega  H_\delta^m g^\epsilon  }_{\mathcal H_x^3 L_v^2}^2\Big)dt\bigg]^{1\over2}\\
	 &\leq \norm{e^{a_0\comi v^2}f_{in}}_{\mathcal H_x^3 L_v^2}+\frac1{24} \big(\abs{h}_{X_{\rho,T_\epsilon}}+\abs{h}_{Z_{\rho,T_\epsilon,\epsilon}}\big).
	 \end{aligned}
	\end{equation}
	We will prove the above assertion still holds true for $m=k$.  
	Note we can not perform estimates directly for $\omega H_\delta^kg^\epsilon $
due to the  low regularity of $g^\epsilon$.  To avoid the formal computation,   we will work with  the regularization
   $\Lambda_\vartheta^{-1} \omega H_\delta^k g^{\epsilon} $ instead, where the regularization operator $\Lambda_\vartheta^{-1}$ is defined 
   by
   \begin{equation}\label{reguoperator}
	\Lambda_\vartheta:= 1-\vartheta \Delta_{x,v},     \quad   0<\vartheta\ll 1.
	\end{equation}
 Note $ \Lambda_\vartheta^{-1}$ is  a Fourier multiplier with symbol  $(1+\vartheta|\xi | ^2+\vartheta|\eta | ^2)^{-1}$. Here and below, $\xi$ and $\eta$ are the Fourier dual variable of $ x$ and $v,$ respectively.    Some basic properties for the regularization operator are listed in  Appendix \ref{secapp:reguopera} (see Lemma \ref{lem:regular}  therein).  In the proof, the generic constant, denoted by $C$, is independent of the parameter $\vartheta$ in \eqref{reguoperator}.
	
	To simplify the notation we denote
	\begin{equation}\label{gkeps+}
		g_k^\epsilon= \Lambda_\vartheta^{-1}  \omega   H_\delta^{k} g^\epsilon.
	\end{equation}
Then 	applying  $\Lambda_\vartheta^{-1}\omega  H_\delta^{k}  $ to the    parabolic equation \eqref{regularied},  we have that, similar to \eqref{hkf+},  
  \begin{align*}
	&(\partial_t   + v\cdot\pa_x+\comi v^2  )        g_k^\epsilon     +\epsilon\big(\comi v^{\frac{2 }{1-s}}- \Delta_{x,v}\big)   g_k^\epsilon  \\
	 & = \Lambda_\vartheta^{-1} \big( \omega   H_\delta^{k}  Q( f,   h)+k \delta t^{\delta-1}   \omega  \partial_{v_1} H_\delta^{k-1}   g^\epsilon\big)+[v\cdot \partial_x+\comi v^2+\epsilon\comi v^{\frac{2 }{1-s}},   \Lambda_\vartheta^{-1}]\omega  H_\delta^{k} g^\epsilon\\ 
	 &\qquad +\epsilon \Lambda_\vartheta^{-1}\omega[\comi v^{\frac{2 }{1-s}},\  H_\delta^{k}]g^\epsilon-\epsilon \Lambda_\vartheta^{-1}  [\Delta_{v},\  \omega]    H_\delta^{k} g^\epsilon,
\end{align*}
recalling $[\cdot, \cdot]$ stands for the commutator between two operators. 
We take $\mathcal H_x^3 L_v^2$ inner product with $g_k^\epsilon$, to obtain that
\begin{equation}\label{gkeps}
	\begin{aligned}
		&\frac{1}{2}\frac{d}{dt}\norm{g_k^\epsilon}_{\mathcal H_x^3 L_v^2}^2+\norm{\comi v g_k^\epsilon}_{\mathcal H_x^3 L_v^2}^2+\epsilon \norm{\partial_{x,v} g_k^\epsilon}_{\mathcal H_x^3 L_v^2}^2+\epsilon\norm{\comi v^{\frac{1 }{1-s}}g_k^\epsilon}_{\mathcal H_x^3 L_v^2}^2\\
		&\leq \big|\big(  \omega   H_\delta^k Q(f, h), \     \Lambda_\vartheta^{-1} g_k^\epsilon\big)_{\mathcal H_x^3 L_v^2}\big|+Ck\big|\big( \omega  \partial_{v_1} H_\delta^{k-1}   g^\epsilon, \     \Lambda_\vartheta^{-1} g_k^\epsilon\big)_{\mathcal H_x^3 L_v^2}\big|\\
		&\quad+\big|\big([v\cdot\partial_x+\comi v^2+\epsilon\comi v^{\frac{2 }{1-s}} , \ \Lambda_\vartheta^{-1}]\omega H_\delta^k  g^\epsilon,\   g_k^\epsilon\big)_{\mathcal H_x^3 L_v^2}\big| 	\\
 &\quad +  \big|\big(  \epsilon \Lambda_\vartheta^{-1}\omega[\comi v^{\frac{2 }{1-s}},\  H_\delta^{k}]g^\epsilon-\epsilon \Lambda_\vartheta^{-1}  [\Delta_v,\  \omega]    H_\delta^{k} g^\epsilon,\  g_k^\epsilon\big)_{\mathcal H_x^3 L_v^2}\big|.
	\end{aligned}
\end{equation}
We claim that
\begin{equation}\label{claim}
	\begin{aligned}
		&\big|\big(  \omega   H_\delta^k Q(f, h), \     \Lambda_\vartheta^{-1} g_k^\epsilon\big)_{\mathcal H_x^3 L_v^2}\big|+Ck\big|\big( \omega  \partial_{v_1} H_\delta^{k-1}   g^\epsilon, \     \Lambda_\vartheta^{-1} g_k^\epsilon\big)_{\mathcal H_x^3 L_v^2}\big|\\
		&\qquad+\big|\big([v\cdot\partial_x+\comi v^2+\epsilon\comi v^{\frac{2 }{1-s}} , \ \Lambda_\vartheta^{-1}]\omega H_\delta^k  g^\epsilon,\   g_k^\epsilon\big)_{\mathcal H_x^3 L_v^2}\big| 	\\
 &\qquad +  \big|\big(  \epsilon \Lambda_\vartheta^{-1}\omega[\comi v^{\frac{2 }{1-s}},\  H_\delta^{k}]g^\epsilon-\epsilon \Lambda_\vartheta^{-1}  [\Delta_v,\  \omega]    H_\delta^{k} g^\epsilon,\  g_k^\epsilon\big)_{\mathcal H_x^3 L_v^2}\big|\\
 &\leq \frac12\Big( \norm{\comi v g_k^\epsilon}_{\mathcal H_x^3 L_v^2}^2+\epsilon \norm{\partial_{x,v} g_k^\epsilon}_{\mathcal H_x^3 L_v^2}^2+\epsilon\norm{\comi v^{\frac{1 }{1-s}}g_k^\epsilon}_{\mathcal H_x^3 L_v^2}^2\Big) +C \norm{g_k^\epsilon}_{\mathcal H_x^3 L_v^2}^2\\
 &\quad+C_\epsilon   k^2   \Big(\normm{  \omega   H_\delta^{k-1}g^\epsilon}^2_{\mathcal{H}^3_x}+\| \comi v  \omega H_\delta^{k-1}g^\epsilon\|_{\mathcal{H}^3_xL^2_v}^2\Big)\\
		& \quad + C_\epsilon 
\bigg[\sum_{j=0}^k\binom{k}{j} \norm{\omega H_\delta^j f}_{\mathcal H_x^3 L_v^2}   \big( \normm{\omega H_\delta^{k-j}  h}_{\mathcal H_x^3 } +\norm{\comi v \omega  H_\delta^{k-j} h}_{\mathcal H_x^3 L_v^2}  \big) \bigg ]^{2}\\
		&\quad + C\bigg[\sum_{j=1}^{k}\frac{k!}{(k-j)!} R_s^j \epsilon^{1\over2}\norm{ \comi v^{ \frac{1 }{1-s}-1}  \omega H_\delta^{k-j} g^\epsilon}_{\mathcal H_x^3 L_v^2}\bigg]^2,
	\end{aligned}
\end{equation}
where in the last line $R_s$ is the constant  in \eqref{rs} depending only on $s$. 
The proof of \eqref{claim} is postponed to the end of this subsection.   
Combining \eqref{gkeps} and \eqref{claim} and then using Gronwall's inequality,  we have that, for any $t\in [0,T_\epsilon]$,
 \begin{equation}\label{egk}
 	\begin{aligned}
 		& \norm{g_k^\epsilon (t)}_{\mathcal H_x^3 L_v^2} +\bigg[\int_0^{T_\epsilon}\Big(\norm{\comi v g_k^\epsilon}_{\mathcal H_x^3 L_v^2}^2+\epsilon \norm{\partial_{x,v} g_k^\epsilon}_{\mathcal H_x^3 L_v^2}^2+\epsilon\norm{\comi v^{\frac{1 }{1-s}}g_k^\epsilon}_{\mathcal H_x^3 L_v^2}^2\Big)dt\bigg]^{1\over2}\\
 		&\leq \lim_{t\rightarrow 0} \norm{g_k^\epsilon (t)}_{\mathcal H_x^3 L_v^2} + C_\epsilon   k    \bigg[\int_0^{T_\epsilon}\Big(\normm{  \omega   H_\delta^{k-1}g^\epsilon}^2_{\mathcal{H}^3_x}+\| \comi v  \omega H_\delta^{k-1}g^\epsilon\|_{\mathcal{H}^3_xL^2_v}^2\Big)dt\bigg]^{1\over2}\\
		& \quad + C_\epsilon 
\sum_{j=0}^k\binom{k}{j} \bigg[\int_0^{T_\epsilon}\norm{\omega H_\delta^j f}_{\mathcal H_x^3 L_v^2}^2   \big( \normm{\omega H_\delta^{k-j}  h}_{\mathcal H_x^3 } +\norm{\comi v \omega  H_\delta^{k-j} h}_{\mathcal H_x^3 L_v^2}  \big)^2 dt \bigg ]^{1\over 2}\\
		&\quad + C\sum_{j=1}^{k}\frac{k!}{(k-j)!} R_s^j\bigg[ \int_0^{T_\epsilon} \epsilon  \norm{ \comi v^{ \frac{1 }{1-s}-1}  \omega H_\delta^{k-j} g^\epsilon}_{\mathcal H_x^3 L_v^2}^2dt\bigg]^{1\over2}.
 	\end{aligned}
 \end{equation}
For the first term on the right-hand side of \eqref{egk},  we claim that
\begin{equation}\label{shtb}
	\lim_{t\rightarrow 0}\norm{  g_k^\epsilon}_{\mathcal H_x^3 L_v^2}=0.
\end{equation} 
In fact, in view of \eqref{gkeps+} and \eqref{vweight},  we use Lemma \ref{lem:regular} in Appendix \ref{secapp:reguopera} to   compute  
\begin{equation*}
	\norm{  g_k^\epsilon}_{\mathcal H_x^3 L_v^2}=\norm{ \Lambda_\vartheta^{-1} \omega H_\delta H_\delta^{k-1} g^\epsilon}_{\mathcal H_x^3 L_v^2}\leq  C_\vartheta t^\delta \norm{  \comi v \omega   H_\delta^{k-1} g^\epsilon}_{\mathcal H_x^3 L_v^2} 
\end{equation*}
for some constant $C_\vartheta$ depending only on $\vartheta.$
This   with the inductive assumption \eqref{min} yields that
\begin{equation*}
	  \int_0^{T_\epsilon} t^{-2\delta}\norm{  g_k^\epsilon}_{\mathcal H_x^3 L_v^2}^2dt\leq C_\vartheta \int_0^{T_\epsilon}  \norm{  \comi v \omega   H_\delta^{k-1} g^\epsilon}_{\mathcal H_x^3 L_v^2}^2dt<+\infty. 
\end{equation*}
Observe $2\delta\geq 2$, and thus the above estimate, together with the continuity of the mapping $t\mapsto \norm{g_k^\epsilon(t)}_{\mathcal H_x^3 L_v^2}^2$, yields  assertion \eqref{shtb}.

For the second term on the right-hand side of \eqref{egk}, we use   \eqref{noreps} as well as the interpolation inequality, to obtain that
 \begin{equation}\label{eps1}
 	\begin{aligned}
 		&C_\epsilon   k    \bigg[\int_0^{T_\epsilon}\Big(\normm{  \omega   H_\delta^{k-1}g^\epsilon}^2_{\mathcal{H}^3_x}+\| \comi v  \omega H_\delta^{k-1}g^\epsilon\|_{\mathcal{H}^3_xL^2_v}^2\Big)dt\bigg]^{1\over2}\\
 		&\leq \frac{1}{64}   k    \bigg[\int_0^{T_\epsilon}\epsilon\Big(\norm{\partial_v  \omega   H_\delta^{k-1}g^\epsilon}^2_{\mathcal{H}^3_x L_v^2}+\| \comi v^{\frac{1}{1-s}}  \omega H_\delta^{k-1}g^\epsilon\|_{\mathcal{H}^3_xL^2_v}^2\Big)dt\bigg]^{1\over2}\\
 		&\quad+ C_\epsilon  k T_\epsilon^{1\over2} \sup_{t\leq T_\epsilon} \norm{  \omega   H_\delta^{k-1}g^\epsilon}_{\mathcal{H}^3_x L_v^2}  \\
 		&\leq \Big(\frac1{8}+C_\epsilon T_\epsilon^{1\over2}
 	\Big)\Big(\norm{e^{a_0\comi v^2}f_{in}}_{\mathcal H_x^3 L_v^2}+\frac1{24} \big(\abs{h}_{X_{\rho,T_\epsilon}}+\abs{h}_{Z_{\rho,T_\epsilon,\epsilon}}\big)\Big)  \frac{\rho^{k-1}k! }{(k+1)^3} \\
 		&\leq\frac14\Big(\norm{e^{a_0\comi v^2}f_{in}}_{\mathcal H_x^3 L_v^2}+\frac1{24} \big(\abs{h}_{X_{\rho,T_\epsilon}}+\abs{h}_{Z_{\rho,T_\epsilon,\epsilon}}\big)\Big)  \frac{\rho^{k-1}k! }{(k+1)^3},
 	\end{aligned}
 \end{equation}
the second inequality using the inductive assumption \eqref{min} and the last line holding provided we choose $T_\epsilon$ small sufficiently such that $C_\epsilon T_\epsilon^{1\over2}\leq 2^{-3}$.

Consider the third term on the right-hand side of \eqref{egk}. We use  Definition \ref{assum} and \eqref{max} and \eqref{eng} to compute   
\begin{equation*}
	\begin{aligned} 	
	&\sum_{j=1}^{k-1}\binom{k}{j} \sup_{t\leq T_\epsilon}\norm{\omega H_\delta^j f}_{\mathcal H_x^3 L_v^2} \bigg[\int_0^{T_\epsilon} \big( \normm{\omega H_\delta^{k-j}  h}_{\mathcal H_x^3 }^2+\norm{\comi v \omega  H_\delta^{k-j} h}_{\mathcal H_x^3 L_v^2}^2 \big)dt\bigg ]^{1\over2}\\
	&\leq \abs{f}_{X_{\rho,T_\epsilon}} \abs{h}_{Y_{\rho,T_\epsilon}}\sum_{j=1}^{k-1}\frac{k!}{j!(k-j)!} \frac{\rho^{j-1}j!}{(j+1)^3} \frac{\rho^{k-j-1} (k-j)! }{(k-j+1)^3}\\
	&\leq \frac{C}{\rho} 	 \abs{f}_{X_{\rho,T_\epsilon}} \abs{h}_{Y_{\rho,T_\epsilon}} \frac{\rho^{k-1} k! }{(k+1)^3},
 	\end{aligned}
\end{equation*}
where in the  last inequality we used  a similar computation to that in \eqref{comp}.  For $j=0$ or $j=k$, it follows  from \eqref{max} and \eqref{eng}  directly that  \begin{multline*}
	\sup_{t\leq T_\epsilon}\norm{\omega H_\delta^j f}_{\mathcal H_x^3 L_v^2} \bigg[\int_0^{T_\epsilon} \big( \normm{\omega H_\delta^{k-j}  h}_{\mathcal H_x^3 }^2+\norm{\comi v \omega  H_\delta^{k-j} h}_{\mathcal H_x^3 L_v^2}^2 \big)dt\bigg ]^{1\over2}\\
	\leq   \abs{f}_{X_{\rho,T_\epsilon}} \abs{h}_{Y_{\rho,T_\epsilon}}\frac{\rho^{k-1}k!}{(k+1)^3}.
\end{multline*}
Then combining the above inequalities  yields that, for any $\tilde\eps>0, $ 
\begin{equation*}
	\begin{aligned}
		&C_\epsilon 
\sum_{j=0}^k\binom{k}{j} \bigg[\int_0^{T_\epsilon}\norm{\omega H_\delta^j f}_{\mathcal H_x^3 L_v^2}^2   \big( \normm{\omega H_\delta^{k-j}  h}_{\mathcal H_x^3 } +\norm{\comi v \omega  H_\delta^{k-j} h}_{\mathcal H_x^3 L_v^2}  \big)^2 dt \bigg ]^{1\over 2}\\
&\leq C_\epsilon \sum_{j=0}^{k}\binom{k}{j} \sup_{t\leq T_\epsilon}\norm{\omega H_\delta^j f}_{\mathcal H_x^3 L_v^2} \bigg[\int_0^{T_\epsilon} \big( \normm{\omega H_\delta^{k-j}  h}_{\mathcal H_x^3 }^2+\norm{\comi v \omega  H_\delta^{k-j} h}_{\mathcal H_x^3 L_v^2}^2 \big)dt\bigg ]^{1\over2}\\
&\leq C_\epsilon  \abs{f}_{X_{\rho,T_\epsilon}} \abs{h}_{Y_{\rho,T_\epsilon}}\frac{\rho^{k-1}k! }{(k+1)^3}\\
&  \leq C_\epsilon   C_*   \Big( \tilde\eps \abs{h}_{Z_{\rho,T_\epsilon,\epsilon}}+ C_{\epsilon,\tilde\eps} T_\epsilon^{1\over 2}  \abs{h}_{X_{\rho,T_\epsilon}}\Big)\frac{\rho^{k-1}k! }{(k+1)^3}, 
	\end{aligned}
\end{equation*}
 where the last inequality follows from \eqref{endif} and \eqref{conclu}.  As a result, there exists a constant $\tilde C_\epsilon$, depending only on $\epsilon$ and $  C_*$, such that
 \begin{equation}\label{ceps}
	\begin{aligned}
	&C_\epsilon 
\sum_{j=0}^k\binom{k}{j} \bigg[\int_0^{T_\epsilon}\norm{\omega H_\delta^j f}_{\mathcal H_x^3 L_v^2}^2   \big( \normm{\omega H_\delta^{k-j}  h}_{\mathcal H_x^3 } +\norm{\comi v \omega  H_\delta^{k-j} h}_{\mathcal H_x^3 L_v^2}  \big)^2 dt \bigg ]^{1\over 2}\\
&\leq \Big(\frac{1}{96} \abs{h}_{Z_{\rho,T_\epsilon,\epsilon}}+\tilde C_\epsilon T_\epsilon^{1\over2} \abs{h}_{X_{\rho,T_\epsilon}}\Big)\frac{\rho^{k-1}k! }{(k+1)^3} \leq 	\frac {1}{96} \Big(  \abs{h}_{Z_{\rho,T_\epsilon,\epsilon}}+    \abs{h}_{X_{\rho,T_\epsilon}}\Big)\frac{\rho^{k-1}k! }{(k+1)^3},
	\end{aligned}
\end{equation}
  provided  $\tilde C_\epsilon T_\epsilon^{1\over 2}\leq \f1{96}.$

 For the last term on the right-hand side of \eqref{egk},  we use the fact that 
 \begin{equation*}
        \norm{ \comi v^{ \frac{1 }{1-s}-1}  \omega H_\delta^{k-j} g^\epsilon}_{\mathcal H_x^3 L_v^2}\leq \tilde\eps  \norm{\comi v^{ \frac{1 }{1-s}}  \omega H_\delta^{k-j} g^\epsilon}_{\mathcal H_x^3 L_v^2}+C_{\tilde\eps}\norm{  \omega H_\delta^{k-j} g^\epsilon}_{\mathcal H_x^3 L_v^2}
 \end{equation*}
 due to the 
 interpolation inequality, to conclude that, for any $\tilde\eps$,  
 \begin{equation*}
 	\begin{aligned}
 		&C\sum_{j=1}^{k}\frac{k!}{(k-j)!} R_s^j\bigg[ \int_0^{T_\epsilon} \epsilon  \norm{ \comi v^{ \frac{1 }{1-s}-1}  \omega H_\delta^{k-j} g^\epsilon}_{\mathcal H_x^3 L_v^2}^2dt\bigg]^{1\over2}\\
 		&\leq  \tilde \eps \sum_{j=1}^{k}\frac{k!}{(k-j)!} R_s^j\bigg[ \int_0^{T_\epsilon} \epsilon  \norm{ \comi v^{ \frac{1 }{1-s}}  \omega H_\delta^{k-j} g^\epsilon}_{\mathcal H_x^3 L_v^2}^2dt\bigg]^{1\over2}\\
 		&\qquad+C_{\tilde\eps}\epsilon^{1\over2} T_\epsilon ^{1\over2}\sum_{j=1}^{k}\frac{k!}{(k-j)!} R_s^j   \sup_{t\leq T_\epsilon}\norm{   \omega H_\delta^{k-j} g^\epsilon}_{\mathcal H_x^3 L_v^2}.
 	\end{aligned}
 \end{equation*} 
 This,  with the inductive assumption \eqref{min} and the fact that 
 \begin{equation*}
 	\sum_{j=1}^{k}\frac{k!}{(k-j)!} R_s^j \frac{\rho^{k-j}(k-j)!}{(k-j+1)^3}\leq CR_s \frac{\rho^{k-1}k! }{(k+1)^3}
 \end{equation*}
for $\rho\geq 2R_s$,
 yields  
 \begin{equation*}
 	\begin{aligned}
 		&C\sum_{j=1}^{k}\frac{k!}{(k-j)!} R_s^j\bigg[ \int_0^{T_\epsilon} \epsilon  \norm{ \comi v^{ \frac{1 }{1-s}-1}  \omega H_\delta^{k-j} g^\epsilon}_{\mathcal H_x^3 L_v^2}^2dt\bigg]^{1\over2}\\
 		& \leq   \Big(\frac{1}{8} + C  \epsilon^{1\over2} T_\epsilon^{1\over2} \Big)\Big(\norm{e^{a\comi v^2}f_{in}}_{\mathcal H_x^3 L_v^2}+\frac1{24} \big(\abs{h}_{X_{\rho,T_\epsilon}}+\abs{h}_{Z_{\rho,T_\epsilon,\epsilon}}\big)\Big)\frac{\rho^{k-1}k! }{(k+1)^3}\\
 		&  \leq   \frac14\Big(\norm{e^{a_0\comi v^2}f_{in}}_{\mathcal H_x^3 L_v^2}+\frac1{24} \big(\abs{h}_{X_{\rho,T_\epsilon}}+\abs{h}_{Z_{\rho,T_\epsilon,\epsilon}}\big)\Big)\frac{\rho^{k-1}k! }{(k+1)^3},
 	\end{aligned}
 \end{equation*}  
 the last inequality holding because of   $\epsilon\ll 1.$ 
Substituting  the above inequality and estimates \eqref{shtb}, \eqref{eps1} and \eqref{ceps} into \eqref{egk} yields that
\begin{equation*}
\begin{aligned}
&	\norm{g_k^\epsilon (t)}_{\mathcal H_x^3 L_v^2} +\bigg[\int_0^{T_\epsilon}\Big(\norm{\comi v g_k^\epsilon}_{\mathcal H_x^3 L_v^2}^2+\epsilon \norm{\partial_{x,v} g_k^\epsilon}_{\mathcal H_x^3 L_v^2}^2+\epsilon\norm{\comi v^{\frac{1 }{1-s}}g_k^\epsilon}_{\mathcal H_x^3 L_v^2}^2\Big)dt\bigg]^{1\over2}\\
&\leq \Big(\norm{e^{a_0\comi v^2}f_{in}}_{\mathcal H_x^3 L_v^2}+\frac1{24} \big(\abs{h}_{X_{\rho,T_\epsilon}}+\abs{h}_{Z_{\rho,T_\epsilon,\epsilon}}\big)\Big)\frac{\rho^{k-1}k! }{(k+1)^3}.
	\end{aligned}
\end{equation*}
Recall $g_k^\epsilon=\Lambda_{\vartheta}^{-1}\omega H_\delta^k g^\epsilon.$ 
Note the constants on the right-hand side are independent of $\vartheta$. Thus letting $\vartheta\rightarrow 0$ in the above estimate,  we use monotone convergence theorem to conclude that
\begin{equation*}
\begin{aligned}
&	\norm{\omega H_\delta^k g^\epsilon (t)}_{\mathcal H_x^3 L_v^2}\\
&\quad +\bigg[\int_0^{T_\epsilon}\Big(\norm{\comi v\omega H_\delta^k g^\epsilon  }_{\mathcal H_x^3 L_v^2}^2+\epsilon \norm{\partial_{x,v}\omega H_\delta^k g^\epsilon}_{\mathcal H_x^3 L_v^2}^2+\epsilon\norm{\comi v^{\frac{1 }{1-s}}\omega H_\delta^k g^\epsilon}_{\mathcal H_x^3 L_v^2}^2\Big)dt\bigg]^{1\over2}\\
&\leq \Big(\norm{e^{a_0\comi v^2}f_{in}}_{\mathcal H_x^3 L_v^2}+\frac1{24} \big(\abs{h}_{X_{\rho,T_\epsilon}}+\abs{h}_{Z_{\rho,T_\epsilon,\epsilon}}\big)\Big)\frac{\rho^{k-1}k! }{(k+1)^3}.
	\end{aligned}
\end{equation*}
This gives the validity of \eqref{min} for $m=k.$  Thus \eqref{kgeq1} follows for any $k\geq 1$.

{\it Step 2 (Short-time behavior).} Observe that there exists a constant $C_\delta$ depending only on $\delta\geq 1$ such that 
\begin{equation*}
t^{-2}\norm{\comi{D_{x,v}}^{-1}H_\delta h}_{\mathcal H_x^3 L_v^2}^2\leq C_\delta t^{2(\delta-1)}\norm{   h}_{\mathcal H_x^3 L_v^2}.
\end{equation*}
As a result, it follows from the interpolation as well as \eqref{vweight}  that, for any  $k\geq 1,$
\begin{equation*}
	\begin{aligned}
&\int_0^{T_\epsilon}	t^{-1}\norm{\omega H_\delta^k  g^\epsilon}_{\mathcal H_x^3 L_v^2}^2dt \\
&\leq C\int_0^{T_\epsilon}
 \norm{\partial_{x,v}\omega H_\delta^k g^\epsilon}_{\mathcal H_x^3 L_v^2}^2dt +  C\int_0^{T_\epsilon} t^{-2} \norm{\comi{D_{x,v}}^{-1}\omega H_\delta^k g^\epsilon}_{\mathcal H_x^3 L_v^2}^2 dt\\
&\leq C\int_0^{T_\epsilon}\big (\norm{\partial_{x,v}\omega H_\delta^{k} g^\epsilon}_{\mathcal H_x^3 L_v^2}^2+ \norm{\comi v \omega H_\delta^{k-1} g^\epsilon}_{\mathcal H_x^3 L_v^2}^2\big)dt<+\infty,
	\end{aligned}
\end{equation*}
the last inequality using \eqref{kgeq1}. This, with the continuity of the mapping $$t\mapsto \norm{\omega H_\delta^k  g^\epsilon}_{\mathcal H_x^3 L_v^2},$$  yields that  $\lim_{t\rightarrow 0}\norm{\omega H_\delta^k  g^\epsilon}_{\mathcal H_x^3 L_v^2}=0.$  The proof of Lemma \ref{lem:smeshort} is completed.
 \end{proof}

The rest part of this subsection is devoted to proving estimate \eqref{claim}.

	\begin{proof}
		[Proof of assertion \eqref{claim}]
By   Leibniz's formula and Corollary \ref{cor:upper} as well as \eqref{ux} in Appendix \ref{secapp:reguopera},  it follows that  
\begin{equation*} 
\begin{aligned}
	& \big|\big(   \omega   H_\delta^k Q(f, h), \  \Lambda_\vartheta^{-1}  g_k^\epsilon\big)_{\mathcal H_x^3 L_v^2}\big|  \leq \epsilon  \big( \normm{  g_k^\epsilon}_{\mathcal H_x^3 }^2 +\norm{\comi v  g_k^\epsilon}_{\mathcal H_x^3 L_v^2}^2 \big)  \\
	& \qquad \qquad+ C\epsilon^{-1}
\bigg[\sum_{j=0}^k\binom{k}{j} \norm{\omega H_\delta^j f}_{\mathcal H_x^3 L_v^2}   \big( \normm{\omega H_\delta^{k-j}  h}_{\mathcal H_x^3 } +\norm{\comi v \omega  H_\delta^{k-j} h}_{\mathcal H_x^3 L_v^2}  \big) \bigg ]^{2}.	\end{aligned}
	\end{equation*}
 We use \eqref{ktd1} as well as \eqref{ux} to conclude 
\begin{equation*}
	\begin{aligned}
		& k\big|\big( \omega  \partial_{v_1} H_\delta^{k-1}   g^\epsilon, \     \Lambda_\vartheta^{-1} g_k^\epsilon\big)_{\mathcal H_x^3 L_v^2}\big|\\
		& \leq  \epsilon     \big(\normm{g_k^\epsilon}^2_{\mathcal{H}^3_x}+\| \comi v g_k^\epsilon\|_{\mathcal{H}^3_xL^2_v}^2\big)  + C_\epsilon   k^2   \big(\normm{  \omega   H_\delta^{k-1}g^\epsilon}^2_{\mathcal{H}^3_x}+\| \comi v  \omega H_\delta^{k-1}g^\epsilon\|_{\mathcal{H}^3_xL^2_v}^2\big).
	\end{aligned}
\end{equation*}
Thus we combine the two estimates above and the fact that
\begin{equation*}
\epsilon  \big( \normm{  g_k^\epsilon}_{\mathcal H_x^3 }^2 +\norm{\comi v  g_k^\epsilon}_{\mathcal H_x^3 L_v^2}^2 \big)\leq  \frac{\epsilon}{16}\Big(  \norm{\partial_{v} g_k^\epsilon}_{\mathcal H_x^3 L_v^2}^2+ \norm{\comi v^{\frac{1 }{1-s}}g_k^\epsilon}_{\mathcal H_x^3 L_v^2}^2\Big) 
		  +C\norm{g_k^\epsilon}_{\mathcal H_x^3 L_v^2}^2
\end{equation*}
due to \eqref{noreps},  to get that
\begin{equation}\label{ohdk}
	\begin{aligned}
		&\big|\big(   \omega   H_\delta^k Q(f, h), \  \Lambda_\vartheta^{-1}  g_k^\epsilon\big)_{\mathcal H_x^3 L_v^2}\big|+ k\big|\big( \omega  \partial_{v_1} H_\delta^{k-1}   g^\epsilon, \     \Lambda_\vartheta^{-1} g_k^\epsilon\big)_{\mathcal H_x^3 L_v^2}\big|\\
		&\leq  \frac{\epsilon}{8}\Big(  \norm{\partial_{v} g_k^\epsilon}_{\mathcal H_x^3 L_v^2}^2+ \norm{\comi v^{\frac{1 }{1-s}}g_k^\epsilon}_{\mathcal H_x^3 L_v^2}^2\Big) 
		  +C\norm{g_k^\epsilon}_{\mathcal H_x^3 L_v^2}^2\\
		  &\quad+ C_\epsilon   k^2   \Big(\normm{  \omega   H_\delta^{k-1}g^\epsilon}^2_{\mathcal{H}^3_x}+\| \comi v  \omega H_\delta^{k-1}g^\epsilon\|_{\mathcal{H}^3_xL^2_v}^2\Big)\\
		& \quad + C_\epsilon 
\bigg[\sum_{j=0}^k\binom{k}{j} \norm{\omega H_\delta^j f}_{\mathcal H_x^3 L_v^2}   \big( \normm{\omega H_\delta^{k-j}  h}_{\mathcal H_x^3 } +\norm{\comi v \omega  H_\delta^{k-j} h}_{\mathcal H_x^3 L_v^2}  \big) \bigg ]^{2}.
	\end{aligned}
\end{equation}
It remains to deal with the commutators. We write 
 \begin{equation*}
 [v\cdot\partial_x, \ \Lambda_\vartheta^{-1}]  =[v\cdot\partial_x, \ \Lambda_\vartheta^{-1}]\Lambda_\vartheta \Lambda_\vartheta^{-1}=-\Lambda_\vartheta^{-1}[v\cdot\partial_x, \ \Lambda_\vartheta]  \Lambda_\vartheta^{-1} = -2 \vartheta \Lambda_\vartheta^{-1} \partial_v\cdot\partial_x    \Lambda_\vartheta^{-1}.
 \end{equation*}
Note that  $\vartheta \Lambda_\vartheta^{-1} \partial_v\cdot\partial_x $ is uniformly (w.r.t. $\vartheta$) bounded in $\mathcal H_x^3 L_v^2$, and thus
 \begin{equation}\label{comvx}
 	\big|\big( [v\cdot\partial_x, \ \Lambda_\vartheta^{-1}]\omega H_\delta^k g^\epsilon,\  g_k^\epsilon\big)_{\mathcal H_x^3 L_v^2}\big|\leq \norm{  g_k^\epsilon}_{\mathcal H_x^3 L_v^2}^2.
 \end{equation}
 Similarly, using the relationship 
 \begin{align*}
 &	 [\comi v^2+\epsilon\comi v^{\frac{2 }{1-s}} , \ \Lambda_\vartheta^{-1}] h  =-\Lambda_\vartheta^{-1}[\comi v^2+\epsilon\comi v^{\frac{2 }{1-s}} , \ \Lambda_\vartheta] \Lambda_\vartheta^{-1}h\\
  &= -2\Lambda_\vartheta^{-1}\vartheta\partial_v \Big (\big[\partial_v (\comi v^2 +\epsilon\comi v^{\frac{2 }{1-s}})\big] \Lambda_\vartheta^{-1} h\Big)+\vartheta \Lambda_\vartheta^{-1}    \Big ( \partial_v^2 (\comi v^2 +\epsilon\comi v^{\frac{2 }{1-s}})  \Big) \Lambda_\vartheta^{-1} h,   
 	 \end{align*}
 	 we have, with help of Lemma \ref{lem:regular} in Appendix \ref{secapp:reguopera},  
 	 \begin{equation}\label{comvepsil}
 	 	\begin{aligned}
 	 		&\big|\big([\comi v^2+\epsilon\comi v^{\frac{2 }{1-s}},  \Lambda_\vartheta^{-1}]\omega H_\delta^k g^\epsilon, g_k^\epsilon\big)_{\mathcal H_x^3 L_v^2}\big| \\
 	 		& \leq C\big( \norm{\comi v^{\frac12}     g_k^\epsilon }_{\mathcal H_x^3 L_v^2}^2+\epsilon \norm{\comi v^{\frac{1+s}{2(1-s)}}    g_k^\epsilon }_{\mathcal H_x^3 L_v^2}^2 \big) \\
 	 	& 	\leq  \frac18  \big( \norm{\comi v   g_k^\epsilon }_{\mathcal H_x^3 L_v^2}^2+\epsilon \norm{\comi v^{\frac{1}{1-s}}    g_k^\epsilon }_{\mathcal H_x^3 L_v^2}^2 \big)+C\norm{  g_k^\epsilon}_{\mathcal H_x^3 L_v^2}^2.
 	 \end{aligned}
 	 \end{equation}
 	 
 Observe  
 \begin{equation*}
	\Lambda_\vartheta^{-1} [\Delta_v,\   \omega]   = \Lambda_\vartheta^{-1} [\Delta_v,\   \omega] \omega^{-1}   \Lambda_\vartheta \Lambda_\vartheta^{-1} \omega =-\Lambda_\vartheta^{-1} \omega[\Delta_v,\  \omega^{-1}   \Lambda_\vartheta]  \Lambda_\vartheta^{-1} \omega,
\end{equation*}
which implies that
\begin{align*}
	 \big|\big(\epsilon \Lambda_\vartheta^{-1}  [\Delta_v,   \omega]   H_\delta^k g^\epsilon,  g_k^\epsilon\big)_{\mathcal H_x^3 L_v^2}\big|  &\leq   \frac18   \epsilon \norm{ \partial_v   g_k^\epsilon }_{\mathcal H_x^3 L_v^2}^2+ C\epsilon \norm{\comi v g_k^\epsilon}_{\mathcal H_x^3 L_v^2}^2\\
	 &\leq    \frac18   \epsilon\Big( \norm{ \partial_v   g_k^\epsilon }_{\mathcal H_x^3 L_v^2}^2+ \norm{ \comi v^{\frac{1 }{1-s}}g_k^\epsilon}_{\mathcal H_x^3 L_v^2} ^2\Big)+C \norm{g_k^\epsilon}_{\mathcal H_x^3 L_v^2}^2.
\end{align*}
Finally, 
   by Leibniz's formula  we compute 
\begin{equation*}
	\begin{aligned}
		&\big|\big(  \epsilon \Lambda_\vartheta^{-1}\omega[\comi v^{\frac{2 }{1-s}},   H_\delta^{k}]g^\epsilon,   g_k^\epsilon\big)_{\mathcal H_x^3 L_v^2}\big|
		\leq \sum_{j=1}^{k}\binom{k}{j}\big|\big(  \epsilon \Lambda_\vartheta^{-1}\omega\big(H_\delta^j\comi v^{\frac{2 }{1-s}}\big) H_\delta^{k-j} g^\epsilon,  g_k^\epsilon\big)_{\mathcal H_x^3 L_v^2}\big|\\
		&\leq \epsilon\sum_{j=1}^{k}\binom{k}{j}\norm{ \comi v^{-\frac{1 }{1-s}} \Lambda_\vartheta^{-1}\omega\big(H_\delta^j\comi v^{\frac{2 }{1-s}}\big) H_\delta^{k-j} g^\epsilon}_{\mathcal H_x^3 L_v^2} \norm{ \comi v^{\frac{1 }{1-s}}g_k^\epsilon}_{\mathcal H_x^3 L_v^2}\\
		&\leq \frac18 \epsilon \norm{ \comi v^{\frac{1 }{1-s}}g_k^\epsilon}_{\mathcal H_x^3 L_v^2} ^2+ C\epsilon\bigg[\sum_{j=1}^{k}\frac{k!}{(k-j)!}R^j_s \norm{ \comi v^{ \frac{1 }{1-s}-1}  \omega H_\delta^{k-j} g^\epsilon}_{\mathcal H_x^3 L_v^2}\bigg]^2, 
	\end{aligned}
\end{equation*}
the last line using  \eqref{rs} as well as  \eqref{b4} in Appendix \ref{secapp:reguopera}.  Combining the above inequalities we conclude that
\begin{multline*}
		  \big|\big(  \epsilon \Lambda_\vartheta^{-1}\omega[\comi v^{\frac{2 }{1-s}},\  H_\delta^{k}]g^\epsilon-\epsilon \Lambda_\vartheta^{-1}  [\Delta_v,\  \omega]    H_\delta^{k} g^\epsilon,\  g_k^\epsilon\big)_{\mathcal H_x^3 L_v^2}\big|\\
		\leq \frac14 \epsilon \Big( \norm{ \partial_v   g_k^\epsilon }_{\mathcal H_x^3 L_v^2}^2+ \norm{ \comi v^{\frac{1 }{1-s}}g_k^\epsilon}_{\mathcal H_x^3 L_v^2} ^2\Big)+C \norm{g_k^\epsilon}_{\mathcal H_x^3 L_v^2}^2\\
		 + C\epsilon\bigg[\sum_{j=1}^{k}\frac{k!}{(k-j)!}R_s^j \norm{ \comi v^{ \frac{1 }{1-s}-1}  \omega H_\delta^{k-j} g^\epsilon}_{\mathcal H_x^3 L_v^2}\bigg]^2.
	\end{multline*}
This with \eqref{ohdk}, \eqref{comvx} and \eqref{comvepsil} yields the desired  \eqref{claim}. The proof of \eqref{claim}  is thus completed.  
\end{proof}
  
 \begin{proof}[Completing the proof of Proposition \ref{prp:linpra}]
 When  $\frac{\gamma}{2}+2s\geq1 $ 
   the desired results in Proposition \ref{prp:linpra} just follow derived from Lemmas \ref{lem:ws} and  \ref{lem:smeshort}.  The remaining case of  $\frac{\gamma}{2}+2s<1 $ can be treated in a similar way, so we omit it for brevity. 
\end{proof}

\subsection{Existence and smoothing effect for the approximation equation} 
In this part we  
  	consider the following   approximation equation   of \eqref{licau}:     
\begin{equation}\label{app+}
	 \partial_t g^\epsilon +v\cdot\partial_xg^\epsilon+\epsilon\big(\comi v^{\frac{2 }{1-s}}- \Delta_{x,v}\big) g^\epsilon =Q(f, g^\epsilon),\quad g^\epsilon|_{t=0}=f_{in},
\end{equation}
with given parameter $0<\epsilon\ll 1$. 

\begin{prop}
	\label{lem:exist}
	Let $C_0$ be the constant given in \eqref{c0}, and
let $\mathcal A_T, \mathcal M_T(C_0)$  be given in Definition \ref{def:c0mt}  and  let $\mathcal N_{T}(\rho, C_*)$ be  given  in Definition \ref{defm}.  Then  there exists a time $T_*>0$ and  three constants $\rho, \tilde C, C_*>0,$ all depending only on $C_0$ in \eqref{c0} and the numbers $C_1,C_2$ in \eqref{lower} but independent of $\epsilon$,  such that for any given  $T\leq T_*$,  if 
\begin{equation}\label{famn}
	f\in\mathcal A_T\cap \mathcal M_T(C_0)\cap \mathcal N_{T}(\rho, C_*),
\end{equation}
then the approximation equation \eqref{app+} admits a local    solution  $g^\eps\in L^\infty\big ([0,T]; \ \mathcal H_x^3 L_v^2\big )$ satisfying that
\begin{equation*}
\forall\ k\geq 1, \quad 	\lim_{t\rightarrow 0}\norm{\omega  {\bm D}^kg^\epsilon}_{\mathcal H_x^3 L_v^2}=0,
\end{equation*}
  and  that 
\begin{equation}\label{gepxy}
	  \abs{ g^\epsilon }_{X_{\rho,T}}+\abs{ g^\epsilon }_{Y_{\rho,T}}\leq \tilde CC_*\bigg(\int_0^T\big( \normm{\om g^\epsilon}_{\mathcal H_x^3  }^2+\norm{\comi v \om g^\epsilon}_{\mathcal H_x^3 L_v^2}^2\big)dt\bigg)^{1\over 2}+\frac{C_*}{2}. 
\end{equation}
Here we used the notation in \eqref{dk}.
\end{prop}

\begin{proof}
 We will proceed to prove	Proposition \ref{lem:exist}  through the following two steps, one devoted  to  proving the existence of solutions to \eqref{app+} in $X_{\rho, T_\epsilon}\cap Z_{\rho, T_\epsilon,\epsilon}$ for some $\epsilon$-dependent lifespan $T_\epsilon$.  In the second one we will remove the $\epsilon$-dependence of the lifespan   by  deriving  uniform (in $\epsilon$) estimates  for $g^\epsilon,$ and show that the solution indeed belong to $X_{\rho, T}\cap Y_{\rho, T}$ for some time $T$ independent of $\epsilon.$  
 
	 {\it Step 1).} Let $g^0=0$ and let $g^n,n\geq1,$  solve  the following linear parabolic equations:
\begin{equation}\label{Picard+++}
		 \partial_t g^n   + v\cdot\pa_xg^n +\epsilon\big(\comi v^{\frac{2 }{1-s}}- \Delta_{x,v}\big) g^n =Q(f,g^{n-1}),\quad		 g^n|_{t=0}=f_{in}. 
\end{equation}
Then by Proposition  \ref{prp:linpra}, for any $n\geq 1$, it holds that $g^n\in  L^\infty\big ([0,T_\epsilon]; \ \mathcal H_x^3 L_v^2\big )$ with the lifespan $T_\epsilon$  independent of $n,$ and moreover we have  that 
\begin{equation}\label{limgn}
		  \abs{g^n}_{X_{\rho,T_\epsilon}}^2+  \abs{g^n}_{Z_{\rho,T_\epsilon,\epsilon}}^2  \leq 12 \norm{e^{a_0\comi v^2} f_{in}}^2_{\mathcal H^3_xL^2_v}+ \frac12 \big(\abs{g^{n-1}}_{X_{\rho,T_\epsilon}}^2+  \abs{g^{n-1}}_{Z_{\rho,T_\epsilon,\epsilon}}^2\big) 		\end{equation}
		and that 
		\begin{equation}\label{stba}
			\forall\  k\geq 1,\quad 	\lim_{t\rightarrow 0} \norm{\om{\bm D}^k g^n}_{\mathcal H_x^3 L_v^2}=0.
		\end{equation}
Denoting 
\begin{eqnarray*}
 \phi^n=g^{n+1}-g^{n},\quad n\geq 1,
\end{eqnarray*}
we have, for $n\geq 1,$ 
\begin{equation*}
	\partial_t  \phi^n   + v\cdot\pa_x \phi^n +\epsilon\big(\comi v^{\frac{2 }{1-s}}- \Delta_{x,v}\big)  \phi^n=Q(f,  \phi^{n-1}), \quad \phi^n|_{t=0}=0. 
\end{equation*}
Using Proposition \ref{prp:linpra} again yields that  
\begin{equation*}
\begin{aligned}
	\forall\ n\geq 1,\quad   \abs{\phi^n}_{X_{\rho,T_\epsilon}}^2+  \abs{\phi^n}_{Z_{\rho,T_\epsilon,\epsilon}}^2  \leq \frac12 \big(\abs{\phi^{n-1}}_{X_{\rho,T_\epsilon}}^2+  \abs{\phi^{n-1}}_{Z_{\rho,T_\epsilon,\epsilon}}^2\big),
	\end{aligned}
\end{equation*}
which implies that
\begin{equation*}
	\forall\ n\geq 1,\quad  \abs{\phi^n}_{X_{\rho,T_\epsilon}}^2+  \abs{\phi^n}_{Z_{\rho,T_\epsilon,\epsilon}}^2  	 \leq  C2^{-n}. 
\end{equation*}
Thus $\{g^n\}_{n\geq 1}$ is a Cauchy sequence in  $X_{\rho, T_\epsilon }$ and $Z_{\rho, T_\epsilon,\epsilon}$. Then 
  there exists a $g^\eps$ such that 
  \begin{equation*}
\lim_{n\rightarrow+\infty}	 \abs{g^n-g^\epsilon}_{X_{\rho,T_\epsilon}}=\lim_{n\rightarrow+\infty}	 \abs{g^n-g^\epsilon}_{Z_{\rho,T_\epsilon,\epsilon}} =0. 
\end{equation*}
Letting $n\rightarrow+\infty$ in \eqref{Picard+++},     the limit $g^\epsilon $  solves equation \eqref{app+}  satisfying that
\begin{equation}\label{smreg}
		\abs{g^\epsilon}^2_{X_{\rho,T_\epsilon}}+\abs{g^\epsilon}^2_{Z_{\rho,T_\epsilon,\epsilon}}  \leq 24 \norm{e^{a_0\comi v^2} f_{in}}^2_{\mathcal H^3_xL^2_v}
\end{equation}
in view of \eqref{limgn}. Moreover, it follows from \eqref{stba} that
\begin{equation*}
\forall\  k\geq 1,\quad 	\lim_{t\rightarrow 0} \norm{\om{\bm D}^k g^\epsilon}_{\mathcal H_x^3 L_v^2}=0.
\end{equation*}

{\it Step 2).}  Let  $g^\epsilon \in  L^\infty\big ([0,T_\epsilon]; \ \mathcal H_x^3 L_v^2\big )$ be the solution constructed in the previous step satisfying condition \eqref{smreg}. 
 This smooth regularity  enables us to  avoid the formal computation when performing estimates on  $g^\epsilon$. So that we may repeat the proof of  Theorem \ref{thm:apri} with slight modification, to conclude that  if $f$ satisfies \eqref{famn} then
 	\begin{equation*}
 		\abs{g^\epsilon}_{X_{\rho,T_\epsilon}}+\abs{g^\epsilon}_{Y_{\rho,T_\epsilon}}
\leq \tilde C C_* \bigg(\int_0^{T_\epsilon}\big( \normm{\om g^\epsilon}_{\mathcal H_x^3  }^2+\norm{\comi v \om g^\epsilon}_{\mathcal H_x^3 L_v^2}^2\big)dt\bigg)^{1\over 2} + \frac{ C_*}{2},
 	\end{equation*}
 where $\tilde C$ is a constant independent of $\epsilon.$
Thus we  apply the standard continuous induction argument to extend the lifespan to a time interval $[0,T]$ with $T$ independent of  $\epsilon,$ such that $g^\epsilon \in  L^\infty\big ([0,T]; \ \mathcal H_x^3 L_v^2\big )$ satisfying \eqref{gepxy}.    Then we 
 complete the proof of Proposition \ref{lem:exist}. 
\end{proof}

 \subsection{Completing the proof of Theorem \ref{thm:lin}}
 
  The existence part  follows from Proposition \ref{lem:exist}. In fact,  let $g^\epsilon \in  L^\infty\big ([0,T]; \ \mathcal H_x^3 L_v^2\big )$ be the solution to 
  \begin{equation*}
  	\partial_t g^\epsilon +v\cdot\partial_xg^\epsilon+\epsilon\big(\comi v^{\frac{2 }{1-s}}- \Delta_{x,v}\big) g^\epsilon =Q(f, g^\epsilon),\quad g^\epsilon|_{t=0}=f_{in},
  \end{equation*}
  constructed in Proposition \ref{lem:exist}. The direct energy method implies that
  \beno
  \int_0^T\big( \normm{\om g^\epsilon}_{\mathcal H_x^3  }^2+\norm{\comi v \om g^\epsilon}_{\mathcal H_x^3 L_v^2}^2\big)dt\leq C, 
  \eeno
  where the constant  $C$ does not depend on $\epsilon$. This together with \eqref{gepxy} yields that
\begin{equation*}
	  \abs{ g^\epsilon }_{X_{\rho,T}}+\abs{ g^\epsilon }_{Y_{\rho,T}}\leq 
   \tilde CC_*\bigg(\int_0^T\big( \normm{\om g^\epsilon}_{\mathcal H_x^3  }^2+\norm{\comi v \om g^\epsilon}_{\mathcal H_x^3 L_v^2}^2\big)dt\bigg)^{1\over 2}+\frac{C_*}{2}
   \leq C.
\end{equation*}
 Then by the standard compact argument, we can find a  subsequence $ g^{\epsilon_j} $ of $ g^{\epsilon}$ which  converges    in $X_{\rho,T}\cap Y_{\rho,T}$ as $\epsilon_j\rightarrow 0$.  Then the limit, denoted by $g$,  solves 
  the linear Boltzmann equation \eqref{lincau+}, satisfying  condition \eqref{thm611} in Theorem \ref{thm:lin}. The rest part is similar to that in the proof of Theorem \ref{thm:apri}. We omit it for brevity.

\section{Existence, uniqueness and smoothing effect for the nonlinear problem}\label{sec:sobsmo} 
This part is devoted to proving Theorem \ref{thm:Gevrey}.    To do so, let $f^0=e^{-2a_0\comi v^2}$ and we consider the following iteration equations:
\begin{equation}\label{lin}
 	\partial_{t}f^{n}+v\cdot\partial_{x}f^{n} 
 	=Q(f^{n-1},f^{n}), \quad f^n|_{t=0}=f_{in}, \quad n\geq 1,
\end{equation}
where $f_{in}$ is the initial datum to the Boltzmann equation \eqref{Bolt}  satisfying the conditions in Theorem \ref{thm:Gevrey}.   We will apply Theorem   \ref{thm:lin} to prove Theorem \ref{thm:Gevrey} and the crucial part is to derive the common lifespan for all $n\geq1.$  
   Precisely, we will show that 
\begin{equation*}
	f^{n}\in \mathcal A_T\cap \mathcal M_T(C_0)\cap \mathcal N_T(\rho, C_*)\Longrightarrow f^{n+1}\in \mathcal A_T\cap \mathcal M_T(C_0)\cap \mathcal N_T(\rho, C_*),
\end{equation*}
where the constants $T, C_0, \rho$ and $C_*$ are independent of $n.$ Recall   $\mathcal A_T, \mathcal M_T(C_0)$ and $\mathcal N_{T}(\rho, C_*)$  are given in Definitions \ref{def:c0mt}  and  \ref{defm}.

\begin{proof}
	[Proof of Theorem \ref{thm:Gevrey}]  We proceed to prove the main result through the following   steps.    Let $C_0$ be given in \eqref{c0}.
Accordingly, let $\Theta, T_*\leq 1 $ and   $ \rho, C_*\geq 1 $ be the constants determined in Theorem \ref{thm:apri},  which depend only on $C_0$ in \eqref{c0} and the constants $C_1$ and $C_2$ in \eqref{lower}. Let $m_0, M_0, E_0$ and  $H_0$ be the numbers in  \eqref{aat}.   
In the following discussion    if necessary we may  shrink $m_0, \Theta$  and increase $C_*, \rho, M_0, H_0$ and $E_0$.

\underline{\it Step (i).}	
Direct verification shows that 
  $f^0=e^{-2a_0\comi v^2}$ satisfies that      
\begin{equation*}
	f^0\in \mathcal A_{T_*}   \cap \mathcal M_{T_*}(C_0)
\end{equation*}
and moreover
\begin{equation}\label{gaes}
	f^0\in   \mathcal N_{T_*}(\rho, C_*), 
\end{equation}
seeing  \cite[estimate (4.17)]{MR4612704} for instance for the proof of \eqref{gaes}.      
This enables us to apply  Theorem  \ref{thm:lin}, 
to  conclude that the linear equation 
\begin{equation*}
 	\partial_{t}f^{1}+v\cdot\partial_{x}f^{1} 
 	=Q(f^{0},f^{1}), \quad f^1|_{t=0}=f_{in},
\end{equation*}
admits a local solution $f^1\in L^\infty([0, T_1]; \mathcal H_x^3 L_v^2)$ for some $T_1\leq T_*$ with $T_*$ the lifespan given in Theorem \ref{thm:lin}.  Moreover,  similar to \eqref{inistep}, we have 
\begin{equation*}
	\int_0^{T_1}\big(\normm{ \omega f^1}_{\mathcal H_x^3}^2+\norm{ \comi v \omega f^1}_{\mathcal H_x^3 L_v^2}^2)dt  \leq C\norm{e^{a_0\comi v^2}f_{in}}^2_{\mathcal H^3_xL^2_v},
\end{equation*} 
which implies that
\begin{equation}\label{ubu}
	\lim_{T\rightarrow 0}\int_0^{T}\big(\normm{ \omega f^1}_{\mathcal H_x^3}^2+\norm{ \comi v \omega f^1}_{\mathcal H_x^3 L_v^2}^2)dt =0.
\end{equation}
Thus, in view of \eqref{ubu},  by shrinking $T_1$ if necessary,  we have 
\begin{equation*}
\bigg(	\int_0^{T_1}\big(\normm{ \omega f^1}_{\mathcal H_x^3}^2+\norm{ \comi v \omega f^1}_{\mathcal H_x^3 L_v^2}^2dt\bigg)^{1\over 2} \leq \frac{\Theta}{4(1+C_*)},
\end{equation*}
which with   Theorem  \ref{thm:lin}, yields 
   \begin{equation*}
 	f^1\in  \mathcal A_{T_1}\cap \mathcal M_{T_1}(C_0) \cap \mathcal N_{T_1}\big (\rho, C_*\big).
 \end{equation*}
In summary, we have   that 
\begin{equation*}
	f^1\in \mathcal A_{T_1}\cap \mathcal M_{T_1}(C_0) \cap \mathcal N_{T_1}\big (\rho, C_*\big)
\end{equation*}
and
\begin{equation*}
	\bigg(	\int_0^{T_1}\big(\normm{ \omega f^1}_{\mathcal H_x^3}^2+\norm{ \comi v \omega f^1}_{\mathcal H_x^3 L_v^2}^2dt\bigg)^{1\over 2} \leq \frac{\Theta}{4(1+C_*)}.
\end{equation*}
By standard iteration, we can find a sequence of positive
numbers  
\begin{equation*}
	T_* \geq T_1\geq T_2\cdots\geq T_n\geq \cdots,
\end{equation*} 
 such that,   for any $ n\geq 1, $
\ben\label{ann}
\left\{
\begin{aligned}
&	f^n\in \mathcal A_{T_n}\cap \mathcal M_{T_n}(C_0) \cap \mathcal N_{T_n}\big (\rho, C_*\big),\\
&\bigg(	\int_0^{T_n}\big(\normm{ \omega f^n}_{\mathcal H_x^3}^2+\norm{ \comi v \omega f^n}_{\mathcal H_x^3 L_v^2}^2dt\bigg)^{1\over 2} \leq \frac{\Theta}{4(1+C_*)}.
\end{aligned}
\right. 
\een

\underline{\it Step (ii).} 
Choose a large integer $n_0\geq 1$   such that 
 \begin{equation}\label{n0}
 	2^{-n_0}  \Big(1+\sup_{t\leq T_1 } \norm{ \omega (f^1-f^0)}_{\mathcal H_x^3 L_v^2} \Big)\leq \frac{\Theta}{4(1+C_*)}.
 \end{equation}
Moreover, define $T$ by setting
 \begin{equation}\label{T}
 	T=\min\big\{T_1, T_2,\cdots, T_{n_0}, T_*\big\}= T_{n_0}.
 \end{equation}
 Then  it follows from \eqref{ann} that 
\begin{equation}\label{nn0}
\forall\ n \leq n_0, \quad 	f^n\in \mathcal A_{T}\cap \mathcal M_{T}(C_0)  \cap \mathcal N_{T}\big (\rho, C_*\big)
\end{equation}
and that
\begin{equation}\label{nn11}
\forall\ n \leq n_0,  \quad  	\bigg[\int_0^{T} \big( \normm{\omega   f^n }_{\mathcal H_x^3}^2 + \norm{\comi v \omega f^n }_{\mathcal H_x^3 L_v^2}^2\big)dt\bigg]^{1\over2}
 \leq \frac{\Theta}{4(1+C_*)}.
\end{equation}

\underline{\it Step (iii).} This step is devoted to show that
 \begin{equation}\label{thetafn}
\forall \ n\geq 1,\quad  	\bigg[\int_0^{T} \big( \normm{\omega   f^n }_{\mathcal H_x^3}^2 + \norm{\comi v \omega f^n }_{\mathcal H_x^3 L_v^2}^2\big)dt\bigg]^{1\over2} \leq \Theta
 \end{equation} 
with $T$ defined by \eqref{T}, which is independent of $n.$

 To do so, define the difference $\zeta^n$ by   setting 
 	\begin{equation}\label{zn}
 		\zeta^n=f^{n+1}-f^{n},  \quad n\geq1. 
 	\end{equation}
 	It follows from \eqref{zn} and \eqref{lin} that, for any $n\geq 1,$
 	\begin{equation}\label{dieq}
 		\partial_{t}\zeta^n+v\cdot\partial_{x}\zeta^n 
 	=Q(f^n, \ \zeta^{n})+Q(\zeta^{n-1},\ f^{n}),\quad \zeta^n|_{t=0}=0. 
 	\end{equation} 
 Since $f^n\in\mathcal A_T\cap\mathcal M_T(C_0),$ then 	we use Lemma \ref{lem:j1} and Corollary \ref{cor:upper} to conclude that, for any $n\geq 1,$
 \begin{equation} \label{znfn}
 	 \begin{aligned}
 	&\sup_{t\leq T } \norm{ \omega \zeta^n}_{\mathcal H_x^3 L_v^2} +\bigg[\int_0^T\big (  \normm{ \omega   \zeta^n }_{\mathcal H_x^3 }^2 + \norm{\comi v \omega  \zeta^n}_{\mathcal H_x^3 L_v^2}^2\big )dt\bigg]^{1\over2} \\
 	&\qquad\qquad  \leq \tilde C  \sup_{t\leq T } \norm{ \omega \zeta^{n-1}}_{\mathcal H_x^3 L_v^2} \bigg[\int_0^T\big (  \normm{ \omega   f^n }_{\mathcal H_x^3 }^2 + \norm{\comi v \omega f^n}_{\mathcal H_x^3 L_v^2}^2\big )dt\bigg]^{1\over2},
 \end{aligned}
 \end{equation} 
 where $\tilde C$ only depend on $C_0,C_1$ and $C_2$ in  \eqref{c0} and \eqref{lower}. As a result, for any $n\leq n_0$, we substitute  \eqref{nn11} into the above estimate,   to get 
\begin{equation*}
\begin{aligned}
	&\sup_{t\leq T } \norm{ \omega \zeta^n}_{\mathcal H_x^3 L_v^2} +\bigg[\int_0^T\big (  \normm{ \omega   \zeta^n }_{\mathcal H_x^3 }^2 + \norm{\comi v \omega  \zeta^n}_{\mathcal H_x^3 L_v^2}^2\big )dt\bigg]^{1\over2}\\
	&  \leq \tilde C  \sup_{t\leq T } \norm{ \omega \zeta^{n-1}}_{\mathcal H_x^3 L_v^2}\frac{\Theta}{4}\leq \frac14\sup_{t\leq T } \norm{ \omega \zeta^{n-1}}_{\mathcal H_x^3 L_v^2}, \end{aligned}
\end{equation*}
provided $\tilde C\Theta\leq \frac12.$
This yields that
\begin{equation}\label{znsmall}
	\forall\ n\leq n_0, \quad \sup_{t\leq T } \norm{ \omega \zeta^n}_{\mathcal H_x^3 L_v^2} +\bigg[\int_0^T\big (  \normm{ \omega   \zeta^n }_{\mathcal H_x^3 }^2 + \norm{\comi v \omega  \zeta^n}_{\mathcal H_x^3 L_v^2}^2\big )dt\bigg]^{1\over2}\leq  \hat C4^{-n},
\end{equation}
 where $\hat C:=\sup_{t\leq T}\norm{ \omega (f^1-f^0)}_{\mathcal H_x^3 L_v^2}=\sup_{t\leq T } \norm{ \omega \zeta^0}_{\mathcal H_x^3 L_v^2}$.  
 Consequently, using \eqref{znsmall} and \eqref{nn11} as well as  the estimate
 \begin{multline}\label{trineq}
 	\bigg[\int_0^T\big (  \normm{ \omega   f^{n_0+1} }_{\mathcal H_x^3 }^2 + \norm{\comi v \omega  f^{n_0+1}}_{\mathcal H_x^3 L_v^2}^2\big )dt\bigg]^{1\over2}\\
 	\leq \bigg[\int_0^T\big (  \normm{ \omega   \zeta^{n_0} }_{\mathcal H_x^3 }^2 + \norm{\comi v \omega  \zeta^{n_0}}_{\mathcal H_x^3 L_v^2}^2\big )dt\bigg]^{1\over2}\\
 	+\bigg[\int_0^T\big (  \normm{ \omega   f^{n_0} }_{\mathcal H_x^3 }^2 + \norm{\comi v \omega  f^{n_0}}_{\mathcal H_x^3 L_v^2}^2\big )dt\bigg]^{1\over2},
 \end{multline}
we obtain 
 \begin{equation}\label{fn01}
 	\begin{aligned}
 		\bigg[\int_0^T\big (  \normm{ \omega   f^{n_0+1} }_{\mathcal H_x^3 }^2 + \norm{\comi v \omega  f^{n_0+1}}_{\mathcal H_x^3 L_v^2}^2\big )dt\bigg]^{1\over2}\leq \hat C4^{-n_0}+\frac{\Theta}{4}\leq \big(1+2^{-n_0}\big)\frac{\Theta}{4},
 	\end{aligned}
 \end{equation}
 the last inequality using \eqref{n0}.  This with \eqref{nn0} enables us to apply Theorem \ref{thm:lin} to conclude that
 \begin{equation*}
 	f^{n_0+1}\in  \mathcal A_{T}\cap \mathcal M_{T}(C_0) \cap \mathcal N_{T}\big (\rho, C_*\big),
 \end{equation*}
 and moreover, in view of \eqref{znfn} and \eqref{fn01}, 
 \begin{equation*}
 	\begin{aligned}
 	&\sup_{t\leq T } \norm{ \omega \zeta^{n_0+1}}_{\mathcal H_x^3 L_v^2} +\bigg[\int_0^T\big (  \normm{ \omega   \zeta^{n_0+1} }_{\mathcal H_x^3 }^2 + \norm{\comi v \omega  \zeta^{n_0+1}}_{\mathcal H_x^3 L_v^2}^2\big )dt\bigg]^{1\over2}\\
	&  \leq \tilde C  \sup_{t\leq T } \norm{ \omega \zeta^{n_0}}_{\mathcal H_x^3 L_v^2} \big(1+2^{-n_0}\big) \frac{\Theta}{4}\leq \frac14 \big(1+2^{-n_0}\big)\sup_{t\leq T } \norm{ \omega \zeta^{n_0}}_{\mathcal H_x^3 L_v^2}\\
	&\leq \hat C\big(1+2^{-n_0}\big) 4^{-(n_0+1)},
 	\end{aligned}
 \end{equation*}
 the last line following from \eqref{znsmall}. The above estimate and \eqref{fn01} enable us to apply induction on $n,$ to  conclude that
   \begin{equation}\label{n011}
 \forall \ n\geq n_0+1,\quad    \sup_{t\leq T } \norm{ \omega \zeta^{n}}_{\mathcal H_x^3 L_v^2}\leq \hat C \Big(1+\sum_{j=n_0}^{n-1}2^{-j}\Big) 4^{-n},
\end{equation} 
and that
 \begin{equation}\label{estiman}
 \forall \ n\geq n_0+1,\quad     \bigg[\int_0^T\big (  \normm{ \omega   f^{n} }_{\mathcal H_x^3 }^2 + \norm{\comi v \omega  f^{n}}_{\mathcal H_x^3 L_v^2}^2\big )dt\bigg]^{1\over2}\leq   \Big(1+\sum_{j=n_0}^{n-1}2^{-j}\Big)\frac{\Theta}{4}.
\end{equation} 
In fact,  for any given $n\geq n_0+2,$ we apply a similar estimate to \eqref{trineq}  and the  
inductive assumption;  this yields that
  \begin{equation*}
 	\begin{aligned}
 		& \bigg[\int_0^T\big (  \normm{ \omega   f^{n} }_{\mathcal H_x^3 }^2 + \norm{\comi v \omega  f^{n}}_{\mathcal H_x^3 L_v^2}^2\big )dt\bigg]^{1\over2}\\
 		&\leq \hat C \Big(1+\sum_{j=n_0}^{n-2}2^{-j}\Big) 4^{-(n-1)}+\Big(1+\sum_{j=n_0}^{n-2}2^{-j}\Big)\frac{\Theta}{4}\\
 		&\leq   \Big(1+\sum_{j=n_0}^{n-2}2^{-j}\Big) 2^{-n} \frac{\Theta}{4}+\Big(1+\sum_{j=n_0}^{n-2}2^{-j}\Big)\frac{\Theta}{4}\\
 		&\leq   2^{-n+1} \frac{\Theta}{4}+\Big(1+\sum_{j=n_0}^{n-2}2^{-j}\Big)\frac{\Theta}{4}  \leq   \Big(1+\sum_{j=n_0}^{n-1}2^{-j}\Big)\frac{\Theta}{4},
 	\end{aligned}
 \end{equation*}
 the second inequality using \eqref{n0}.   Moreover, 
using \eqref{znfn} and the  inductive assumption as well as the above estimate, we have, for any $n\geq n_0+2,$
\begin{align*}
	\sup_{t\leq T } \norm{ \omega \zeta^{n}}_{\mathcal H_x^3 L_v^2}&\leq  \tilde C \bigg[ \hat C \Big(1+\sum_{j=n_0}^{n-2}2^{-j}\Big) 4^{-(n-1)} \bigg] \Big(1+\sum_{j=n_0}^{n-1}2^{-j}\Big)\frac{\Theta}{4}\\
	&\leq \bigg[ \hat C \Big(1+\sum_{j=n_0}^{n-2}2^{-j}\Big) 4^{-(n-1)} \bigg] \Big(1+\sum_{j=n_0}^{n-1}2^{-j}\Big)\frac{1}{8}\leq \hat C \Big(1+\sum_{j=n_0}^{n-1}2^{-j}\Big) 4^{-n}.
\end{align*}
 Thus assertions \eqref{n011}  and \eqref{estiman} hold.   Combining  \eqref{estiman} and \eqref{nn11} yields the desired assertion \eqref{thetafn}.

 \underline{\it Step (iv).}  Estimate  \eqref{thetafn} enables us  to apply Theorem \ref{thm:lin}, to conclude that    there exist a time $T$ independent of $n,$  such that for each $n\geq 1$,  we can find a solution $f^n\in L^\infty([0,T]; \mathcal H_x^3 L_v^2)$ to the iteration equation \eqref{lin}, satisfying that   
 \begin{equation}\label{nn00}
   \forall\  n\geq1, \quad f^n\in \mathcal A_{T} \cap \mathcal M_{T}(C_0)\cap \mathcal N_{T}\big (\rho, C_*\big),
\end{equation}
and that
\begin{equation}\label{ddfn}
  \forall\  n\geq1,\quad  \bigg[\int_0^T\big (  \normm{ \omega   f^{n} }_{\mathcal H_x^3 }^2 + \norm{\comi v \omega  f^{n}}_{\mathcal H_x^3 L_v^2}^2\big )dt\bigg]^{1\over2}\leq \Theta.
\end{equation}
Moreover, substituting the above estimate into \eqref{znfn} yields, for any $n\geq 1,$
 \begin{equation} \label{+znfn++}
 	 \begin{aligned}
 	&\sup_{t\leq T } \norm{ \omega \zeta^n}_{\mathcal H_x^3 L_v^2} +\bigg[\int_0^T\big (  \normm{ \omega   \zeta^n }_{\mathcal H_x^3 }^2 + \norm{\comi v \omega  \zeta^n}_{\mathcal H_x^3 L_v^2}^2\big )dt\bigg]^{1\over2} \\
 	&\qquad\qquad  \leq \tilde C \Theta \sup_{t\leq T } \norm{ \omega \zeta^{n-1}}_{\mathcal H_x^3 L_v^2}\leq \frac 12 \sup_{t\leq T } \norm{ \omega \zeta^{n-1}}_{\mathcal H_x^3 L_v^2}\leq \hat C2^{-n}.
 \end{aligned}
 \end{equation} 
  
  \underline{\it Step  (v)  (Existence).} 
 Let $f^n$ be the solutions to the iteration equations \eqref{lin} satisfying \eqref{nn00} and \eqref{ddfn}.    In this step we will show that $\{f^n\}_{n\geq 1}$ is a Cauchy sequence in $X_{\rho, T}\cap Y_{\rho, T}$, that is, 
 \begin{equation}\label{dizn}
 \forall\ n\geq 1,\quad 	\abs{\zeta^n}_{X_{\rho, T}}+	\abs{\zeta^n}_{Y_{\rho, T}} \leq  \tilde C 2^{-n},
 \end{equation} 
recalling $\zeta^n=f^{n+1}-f^n.$   
  
  Without loss of generality   it suffices to prove \eqref{dizn} for the case of  $\frac{\gamma}{2}+2s\geq 1.$
  Similar to \eqref{riindu}, we derive from equation \eqref{dieq} that, for any $k\geq 2, $
 \begin{equation*} 
	 \begin{aligned}
	&  \sup_{t\leq T}  \|\omega  H_\delta^{k} \zeta^n(t)\|_{\mathcal{H}^3_xL^2_v}+ \bigg[\int_0^{T} \Big(\normm{\omega   H_\delta^k  \zeta^n}_{\mathcal H_x^3}^2 + \|\<v\>\omega H_\delta^{k}\zeta^n\|^2_{\mathcal{H}^3_xL^2_v}\Big)dt\bigg]^{1\over2}\\
	&  \leq  C\sup_{0<t\leq T}  \|\omega  H_\delta^{k} f^n(t)\|_{\mathcal{H}^3_xL^2_v} \bigg[\int_0^T \Big(  \normm{\omega \zeta^n}_{\mathcal H_x^{3}}^2+\norm{\comi v\omega \zeta^n}_{\mathcal H_x^{3} L_v^2}^2\Big) dt\bigg]^{1\over2} \\
	&+  C k\bigg[  \int_0^{T}   \Big(\normm{  \omega   H_\delta^{k-1}\zeta^n}^2_{\mathcal{H}^3_x}+\| \comi v  \omega H_\delta^{k-1}\zeta^n\|_{\mathcal{H}^3_xL^2_v}^2\Big) dt\bigg]^{1\over2}\\
	&+  C\sum_{j=1}^{k-1} {k\choose j} \sup_{t\leq T}\norm{\omega H_\delta^j f^n}_{\mathcal{H}^3_xL^2_v}\bigg[\int_0^{T}  \big( \normm{\omega H_\delta^{k-j}\zeta^n}_{\mathcal H_x^2}^2+\norm{\comi v\omega H_\delta^{k-j}\zeta^n}_{\mathcal H_x^3L_v^2}^2\big)dt \bigg]^{1\over2}\\
	&+  C\sum_{j=0}^{k} {k\choose j} \sup_{t\leq T}\norm{\omega H_\delta^j \zeta^{n-1}}_{\mathcal{H}^3_xL^2_v}\bigg[\int_0^{T}  \big( \normm{\omega H_\delta^{k-j}f^n}_{\mathcal H_x^2}^2+\norm{\comi v\omega H_\delta^{k-j}f^n}_{\mathcal H_x^3L_v^2}^2\big)dt \bigg]^{1\over2}.	
	 \end{aligned}
\end{equation*}
Thus for $k\geq 2$, we repeat the argument after \eqref{riindu},  to conclude  that 
 \begin{equation*}
 		 \begin{aligned}
	& L_{\rho,k}  \sup_{t\leq T}  \|\omega  H_\delta^{k} \zeta^n(t)\|_{\mathcal{H}^3_xL^2_v}+ L_{\rho,k} \bigg[ \int_0^{T} \Big(\normm{\omega   H_\delta^k  \zeta^n}_{\mathcal H_x^3}^2 + \|\<v\>\omega H_\delta^{k}\zeta^n\|^2_{\mathcal{H}^3_xL^2_v}\Big)dt\bigg]^{1\over2}\\
	&  \leq    C  |f^n|_{X_{\rho, T}}  \bigg[\int_0^T \Big(  \normm{\omega \zeta^n}_{\mathcal H_x^{3}}^2+\norm{\comi v\omega \zeta^n}_{\mathcal H_x^{3} L_v^2}^2\Big) dt\bigg]^{1\over2}\\
	&\quad+  \frac{  C}{\rho}\big(1+  |f^{n}|_{X_{\rho, T}} \big) |\zeta^{n}|_{Y_{\rho, T}}    + \frac{   C}{\rho}  |\zeta^{n-1}|_{X_{\rho, T}}  |f^{n}|_{Y_{\rho, T}}  +   C |f^{n}|_{Y_{\rho, T}} \sup_{t\leq T} \norm{\omega \zeta^{n-1}}_{\mathcal H_x^3 L_v^2} \\
	&\quad    +    C    |\zeta^{n-1}|_{X_{\rho, T}}   \bigg[\int_0^{T} \big( \normm{\omega   f^n }_{\mathcal H_x^3}^2 + \norm{\comi v \omega f^n }_{\mathcal H_x^3 L_v^2}^2\big)dt\bigg]^{1\over2}\\
	& \leq   C  C_* 2^{-n} +   \frac{   C (1+C_*)}{\rho}  |\zeta^{n}|_{Y_{\rho, T}}   + \frac{   CC_*}{\rho}  |\zeta^{n-1}|_{X_{\rho, T}}+  C C_* 2^{-(n-1)} +  C \Theta  |\zeta^{n-1}|_{X_{\rho, T}},
	 \end{aligned}
 \end{equation*}
where in the last line we use \eqref{nn00}, \eqref{ddfn} and \eqref{+znfn++}.  For $0\leq k\leq 1$,   direct verification shows that
\begin{multline*}
	L_{\rho,k}  \sup_{t\leq T}  \|\omega  H_\delta^{k} \zeta^n(t)\|_{\mathcal{H}^3_xL^2_v}+ L_{\rho,k} \bigg[ \int_0^{T} \Big(\normm{\omega   H_\delta^k  \zeta^n}_{\mathcal H_x^3}^2 + \|\<v\>\omega H_\delta^{k}\zeta^n\|^2_{\mathcal{H}^3_xL^2_v}\Big)dt\bigg]^{1\over2}\\
	\leq  C  C_* 2^{-n} +    C C_* 2^{-(n-1)} +  C \Theta  |\zeta^{n-1}|_{X_{\rho, T}}.
\end{multline*}
 Moreover, the above two estimates still hold true   if we replace  $H_\delta$ by 
  	\begin{equation*}
 		\frac{1}{\delta+1}t^{\delta+1} \partial_{x_j}+ t^{\delta} \partial_{v_j} \textrm{ with } j=2 \textrm{ or } 3.
 	\end{equation*}
  As a result, following the argument after \eqref{mainest+}, we get  
  \begin{equation*}
   |\zeta^{n}|_{X_{\rho, T}}   + |\zeta^{n}|_{Y_{\rho, T}}   \leq 	  C  C_* 2^{-n} +   \frac{1}{2} |\zeta^{n-1}|_{X_{\rho, T}}
  \end{equation*}
  and thus
  \begin{equation*}
  	 |\zeta^{n}|_{X_{\rho, T}}   + |\zeta^{n}|_{Y_{\rho, T}}   \leq 	\tilde  C 2^{-n}. 
  \end{equation*}
 This implies that  $\big\{ f^n\big\}_{n\geq 1}$   is a   Cauchy sequences in   $X_{\rho, T}$ and $Y_{\rho, T}$, and  the limit, denoted by $f$,  solves the nonlinear Boltzmann equation \eqref{Bolt} with initial datum $f_{in}$. Moreover it follows from \eqref{nn00} and \eqref{ddfn} that
\begin{equation}\label{ft}
	f\in \mathcal A_T\cap \mathcal M _T(C_0) \cap \mathcal N_{T}(\rho, C_*)
\end{equation}
and
\begin{equation}\label{thta}
\bigg[\int_0^T\big (  \normm{ \omega  f }_{\mathcal H_x^3 }^2 + \norm{\comi v \omega f}_{\mathcal H_x^3 L_v^2}^2\big )dt\bigg]^{1\over2}\leq \Theta.
\end{equation}

 \underline{\it Step (vi) (uniqueness).} Let $f\in X_{\rho, T}\cap Y_{\rho, T}\subset L^\infty([0,T]; \mathcal H_x^3 L_v^2)$ be the solution to   the nonlinear Boltzmann equation \eqref{Bolt}  constructed in the previous step, and    let $g\in L^\infty([0,T]; \mathcal H_x^3 L_v^2)$ be any  solution  to \eqref{Bolt},  satisfying condition \eqref{aat} and that $g\geq 0.$ 
 	Denote 
	\begin{equation*}
		 \zeta=g-f.
	\end{equation*}
	Then 
	\begin{equation*}
\partial_t 	 \zeta+v\cdot\partial_x	 \zeta=Q(g, 	 \zeta) +Q(	 \zeta, f),\quad 	 \zeta|_{t=0}=0.
	\end{equation*}
 Similar to \eqref{znfn}, we have
 \begin{equation*}
 	 \begin{aligned}
 	&\sup_{t\leq T } \norm{ \omega \zeta}_{\mathcal H_x^3 L_v^2} +\bigg[\int_0^T\big (  \normm{ \omega   \zeta  }_{\mathcal H_x^3 }^2 + \norm{\comi v \omega  \zeta }_{\mathcal H_x^3 L_v^2}^2\big )dt\bigg]^{1\over2} \\
 	&   \leq \tilde C  \sup_{t\leq T } \norm{ \omega \zeta}_{\mathcal H_x^3 L_v^2} \bigg[\int_0^T\big (  \normm{ \omega  f }_{\mathcal H_x^3 }^2 + \norm{\comi v \omega f}_{\mathcal H_x^3 L_v^2}^2\big )dt\bigg]^{1\over2}\\
 	&\leq  \tilde C \Theta  \sup_{t\leq T } \norm{ \omega \zeta}_{\mathcal H_x^3 L_v^2}\leq \frac12   \sup_{t\leq T } \norm{ \omega \zeta}_{\mathcal H_x^3 L_v^2},
 \end{aligned}
 \end{equation*} 
 in the last line we used \eqref{thta} and the fact that $\tilde C\Theta\leq \frac12.$
 Thus $\zeta\equiv 0$ in $L^\infty([0,T]; \mathcal H_x^3 L_v^2)$.  The uniqueness assertion  in Theorem \ref{Bolt} follows.

\underline{\it Step (vii)  (smoothing effect).} 	
It remains to show the quantitative estimate \eqref{alpha1}, and this part is quite similar to that in \cite[Subsection 4.2]{chenlixu2024}.   
For  $j=1$ or $j=2$ it holds that,  in view of \eqref{ft}, 
\begin{equation*}
\forall \ k\in\mathbb Z_+,\quad  L_{\rho, k} \sup_{t \leq T}\norm{   H_{\delta_j}^{k}f(t)}_{H^3_xL^2_v}  
\leq  C_*. 
\end{equation*}
Thus
\begin{equation}\label{hkdel}
\forall \ k\in\mathbb Z_+,\quad 	\sup_{t \leq T}\norm{   H_{\delta_j}^{k}f(t)}_{H^3_xL^2_v}\leq C_*\rho^k(k!)^{\max\{(2\tau)^{-1}, 1\}}.
\end{equation}
Recall   $\delta_1,\delta_2$ are defined in  \eqref{de1de2}. Moreover, observe (cf. \cite[inequality (4.19) in Subsection 4.2]{chenlixu2024})  
  \begin{equation}\label{pse1}
\forall\ k\in\mathbb Z_+,\quad 	\norm{(A_1+A_2)^k f}_{H^3_xL^2_v}\leq 2^{k} \norm{ A_1^k f}_{H^3_xL^2_v}+2^{k} \norm{ A_2^k f}_{H^3_xL^2_v},
 \end{equation}
where $A_j, j=1,2,$ are any   two Fourier multipliers  with symbols $ a_j=a_j(\xi,\zeta)$, that is,
 \begin{eqnarray*}
 \mathcal F_{x,v}	(A_j f)(\xi,\zeta)=a_j(\xi,\zeta)\mathcal F_{x,v}f(\xi,\zeta),
 \end{eqnarray*}
 with $\mathcal F_{x,v} f   $   the full Fourier transform in $(x,v)$.  Then
  we use \eqref{generate} and  then apply \eqref{pse1}  with
   \begin{align*}
   A_1=	\frac{(\delta_2+ 1)(\delta_1+1)}{\delta_2-\delta_1}  H_{\delta_1},\quad A_2=-\frac{(\delta_2+ 1)(\delta_1+1)}{\delta_2-\delta_1} t^{\delta_1-\delta_2}H_{\delta_2},
   \end{align*}
 to compute that, recalling $\delta_1>\delta_2,$
  \begin{align*}
 	\sup_{0<t\leq T}t^{(\lambda+1)k} \norm{\partial_{x_1}^{k}f(t)}_{H^3_xL^2_v}  &=   \sup_{0<t\leq T} \norm{(A_1+A_2)^kf(t)}_{H^3_xL^2_v} \\
 	& \leq 2^{k} \sup_{0<t\leq T} \norm{ A_1^kf(t)}_{H^3_xL^2_v}  +2^{k} \sup_{0<t\leq T} \norm{ A_2^kf(t)}_{H^3_xL^2_v} \\ \
 	&\leq \tilde C^{k}  \sup_{0<t\leq T}\inner{\norm{ H_{\delta_1}^{k}f }_{H^3_xL^2_v} +\norm{ H_{\delta_2}^{k}f }_{H^3_xL^2_v}},
 	 	\end{align*}
 	where $\tilde C$ is a constant depending only on $\delta_1,\delta_2$.  Combining the above estimate   with \eqref{hkdel}, we conclude that
 	\begin{equation*}
 		\sup_{t\leq T}t^{(\lambda+1)k} \norm{\partial_{x_1}^{k}f(t)}_{H^3_xL^2_v}\leq  2C_*(\tilde C\rho)^{k} (k!)^{\max\{(2\tau)^{-1}, 1\}}.
 	\end{equation*}
 Similarly, the above estimate is also true with $\partial_{x_1}$ replaced by $\partial_{x_2}$ or $\partial_{x_3}$. This, with the fact that
\begin{eqnarray*}
\forall\ \alpha\in\mathbb Z_+^3,\quad 	\norm{\partial_x^\alpha f}_{H^3_xL^2_v}\leq \sum_{1\leq j\leq 3}\norm{\partial_{x_j}^{\abs\alpha}f}_{H^3_xL^2_v},
\end{eqnarray*}
gives
\begin{eqnarray*}
	\forall\ \alpha\in\mathbb Z_+^3,\quad 	\sup_{t\leq T}t^{(\lambda+1)\abs\alpha}\norm{\partial_x^\alpha f(t)}_{H^3_xL^2_v}\leq  6C_* (\tilde C\rho)^{\abs\alpha}(\abs\alpha!)^{\max\{(2\tau)^{-1}, 1\}}.
\end{eqnarray*}
In the same way, we have
  \begin{eqnarray*}
\forall\ \beta\in\mathbb Z_+^3, \quad 	\sup_{t\leq T}t^{  \lambda \abs\beta} \norm{\partial_{v}^{\beta}f(t)}_{H^3_xL^2_v}    \leq  6C_*(\tilde C\rho)^{\abs\beta} (\abs\beta!)^{\max\{(2\tau)^{-1}, 1\}}.
 	\end{eqnarray*}
 As a result,  combining the   two estimates above implies that
 (see \cite[Subsection 4.2]{chenlixu2024} for detail)
 \begin{equation*}
 	\forall \ \alpha , \beta \in \mathbb{Z}_{+}^3, \quad 	\sup_{ t\leq T}t^{(\lambda+1)\abs\alpha+ \lambda \abs\beta} \norm{\partial_x^{\alpha}\partial_{v}^{\beta}f(t)}_{H^3_xL^2_v} \leq    C^{|\alpha|+|\beta|+1} \big[(|\alpha|+|\beta|)!\big] ^{\max\{(2\tau)^{-1}, 1\}}
 \end{equation*}
  for some constant $C>0.$  The proof of Theorem \ref{thm:Gevrey} is thus completed. 
\end{proof}

\appendix 

\section{Some auxiliary inequalities}\label{sec:appinter}

\begin{lem} 
For any regular function $h$  and any  $\vep>0$, we have that
\begin{align}\label{tau -j}
\|\<v\>^s\<D_v\>^sh\|_{L^2_v}\leq \vep \|\<D_v\>  h\|_{L^2_v}+C_{\vep}\|\<v\>^{\frac{s}{1-s}} h\|_{L^2_v}.
\end{align}
\end{lem}
\begin{proof}
It follows from \eqref{chara} and Young's inequality that, for any $\varepsilon>0,$
\begin{align*}
\|\<v\>^s\comi{D_v}^s h\|^2_{L^2_v}&\leq C\sum_{k,j=-1}^{+\infty} 2^{2ks} 2^{2js}\|\Delta_j\cP_k h\|^2_{L^2_v}   \\
&\leq \varepsilon  \sum_{k,j=-1}^{+\infty} 2^{2j}\|\Delta_j\cP_k h\|^2_{L^2_v} + C_{\varepsilon}\sum_{k,j=-1}^{+\infty} 2^{2 k s\frac{1}{1-s}} \|\Delta_j\cP_k h\|^2_{L^2_v}\\
&\leq \varepsilon    \|\<D_v\>  h\|^{2}_{L^2_v}+C_{\varepsilon} 
\|\<v\>^{\frac{s}{1-s}} h\|^{2}_{L^2_v}.
\end{align*}
 The proof is thus completed. 
\end{proof}

\begin{lem}\label{lem:A.2}
Let $F:\R^+\rightarrow\R^+$ be a smooth function satisfying $F(x)\leq C x$ and $0\leq F'\leq C$ for some constant $C$,  and let $g:\R^3\rightarrow\R^+$ be a positive function. Then it holds that 
\beno
\forall\  0<s<1,\quad \|F(g)\|_{H^s}\leq \tilde C\|g\|_{H^s},
\eeno
where the constant $\tilde C$ depends only on $C$ and $s.$
In particular, we have that
\begin{equation*}
\|\log(1+g)\|_{H^s}+\Big\|\f g{1+g}\Big\|_{H^s}\leq \tilde C\|g\|_{H^s}.
\end{equation*}
\end{lem}
\begin{proof}
Note that for $s\in(0,1)$ we have the following characterization of Sobolev spaces (cf.\cite[Proposition 1.59]{MR2768550}): 
\beno
\|F(g)\|_{H^s}^2\sim \|F(g)\|^2_{L^2}+\int_{\R^3\times\R^3}\f{\big|F\big (g(x+y)\big )-F\big (g(x)\big )\big|^2}{|y|^{3+2s}}dxdy,
\eeno
recalling that by $A\sim B$ we mean $C^{-1}A\leq B\leq CA$ for some generic constant $C$.

Since $0\leq F(x)\leq C x$, then $\|F(g)\|_{L^2}\leq C\|g\|_{L^2}$. Moreover, by $0<F'\leq C$ we have 
\beno
\big|F\big (g(x+y)\big )-F\big (g(x)\big )\big | \leq C |g(x+y)-g(x)|,
\eeno
and thus
\beno
\int_{\R^3\times\R^3}\f{\big|F\big (g(x+y)\big )-F\big (g(x)\big )\big | ^2}{|y|^{3+2s}}dxdy\leq C \int_{\R^3\times\R^3}\f{|g(x+y)-g(x)|^2}{|y|^{3+2s}}dxdy.
\eeno
This yields the first assertion  in Lemma \ref{lem:A.2}. The second assertion is just a consequence of the first one.  The proof is completed. 
\end{proof}

\section{Properties of the regularization operator}\label{secapp:reguopera}

We list some properties of the  regularization operator  $\Lambda_\vartheta ^{-1}$ which is defined in \eqref{reguoperator}. 

\begin{lem}\label{lem:regular}
The following assertions holds ture.
\begin{enumerate}[label=(\roman*), leftmargin=*, widest=ii]
	\item There exists a constant $C_\vartheta$ depending only on  $\vartheta$ such that 	for any $  p, q\in\mathbb R,$ it holds that
  \begin{equation*} 
  		\norm{\Lambda_\vartheta^{-1}  g}_{H_x^p H_v^q} 
\leq C_{\vartheta}  \norm{g}_{H_x^{p-2} H_v^{q}}\ \textrm{ and }\  \norm{\Lambda_\vartheta^{-1}  g}_{H_x^p H_v^q} 
\leq C_{\vartheta}  \norm{g}_{H_x^{p} H_v^{q-2}}.
	\end{equation*}
\item For  any $g\in H_x^p H_v^q$ with $p,q\in\mathbb R$ we have, for each $1\leq j\leq 3,$
 \begin{equation}\label{ubd}
	\norm{  \Lambda_\vartheta^{-1}  g}_{H_x^p H_v^q}+\norm{\vartheta^{1\over2}  \Lambda_\vartheta^{-1} \partial_{x_j} g}_{H_x^p H_v^q}+\norm{\vartheta  \Lambda_\vartheta^{-1} \partial_{x_j}^2 g}_{H_x^p H_v^q}\leq 3\norm{ g}_{H_x^p H_v^q}.
\end{equation}
The above estimate still holds true if we replace    $\partial_{x_j}$ therein  by   $\partial_{v_j}$.    
\item There exists a constant $C$ independent of $\vartheta, $ such  that  for each $1\leq j\leq 3$ and for any regular $g$ it holds that
\begin{equation}\label{ux}
	  \normm{   \Lambda_\vartheta^{-1}  g}_{H_x^p}+\normm{ \vartheta^{1\over2}    \Lambda_\vartheta^{-1} \partial_{v_j} g}_{H_x^p} +\normm{\vartheta    \Lambda_\vartheta^{-1} \partial_{v_j}^2 g}_{H_x^p}  \leq C\normm{ g}_{H_x^p}.
\end{equation}
The above estimate still holds with $\partial_{v_j}$ replaced by $\partial_{x_j.}$
\end{enumerate}
\end{lem}

\begin{proof}
	Assertions (i) and (ii) just follows from the straightforward  verification.  To prove Assertion (iii) ,  we use the fact that
	\begin{equation*}
	\forall \ \alpha, \beta  \in\mathbb Z_+^3,\quad  \partial_{\xi}^\alpha\partial_{\eta}^\beta \big[   (1+\vartheta| \xi|^2+\vartheta| \eta|  ^2)^{-1} \big]\leq   C_{\alpha, \beta}\inner {1+\abs\alpha+\abs \eta }^{-1}
\end{equation*}
with $  C_{\alpha, \beta}$ constants depending only on $\alpha$ and $\beta$   but independent of $\vartheta. $ This with the help of pseudo-differential calculus (cf. \cite{MR2599384} for instance) implies that, recalling $[\cdot, \cdot]$ stands for the commutator  between two operators, 
\begin{equation}\label{b4}
 \forall\  \ell \in\mathbb R, \quad 	\norm{ [\comi v^{\ell},  \  \Lambda_\vartheta^{-1}]   g}_{L_v^2}\leq C_\ell \norm{ \comi v^{\ell-1}  g}_{ L_v^{2}},
\end{equation}
with $C_\ell $  constants depending only on $\ell $ but independent of $\vartheta.$  
This,  with \eqref{ubd} as well as \eqref{trinorm} and the fact that $[(-\triangle_{\SS^2})^{s/2},   \Lambda_\vartheta^{-1}]=0$,   implies
\begin{align*}
\norm{\comi v^{\frac{\gamma}{2}}(-\triangle_{\SS^2})^{s/2}   \Lambda_\vartheta^{-1} g}_{L_v^2}&\leq \norm{   \comi v^{\frac{\gamma}{2}}(-\triangle_{\SS^2})^{s/2} g}_{L_v^2}+\norm{ [ \Lambda_\vartheta^{-1}, \comi v^{\frac{\gamma}{2}}] (-\triangle_{\SS^2})^{s/2}g}_{L_v^2}\\
&	\leq \norm{   \comi v^{\frac{\gamma}{2}}(-\triangle_{\SS^2})^{s/2}g}_{L_v^2}+C\norm{   \comi v^{\frac{\gamma}{2}-1+s}  \comi{ D_v}^s   g}_{L_v^2}\leq C\normm{g}
\end{align*}
for some constant $C$ independent of $\vartheta.$  Similarly, 
 \begin{equation*}
 	\norm{\comi v^{\frac{\gamma}{2}}\comi{  D_v}^s   \Lambda_\vartheta^{-1} g}_{L_v^2}\leq C\norm{\comi v^{\frac{\gamma}{2}}\comi{  D_v}^s   g}_{L_v^2} \leq C\normm{g}.
 \end{equation*}
The above two estimates   still holds   if we replace   $  \Lambda_\vartheta^{-1}$ by $\vartheta^{1\over2}  \Lambda_\vartheta^{-1} \partial_{v_1} $ and $\vartheta      \Lambda_\vartheta^{-1} \partial_{v_1}^2.$ We have proven  estimate \eqref{ux}. The proof of Lemma \ref{lem:regular} is completed.  
\end{proof}

\medskip
  \noindent{\bf Acknowledgements}.
W.-X. Li was supported by NSFC (Nos.  12325108,  12131017, 12221001),  and  the Natural Science Foundation of Hubei Province (No. 2019CFA007).



\end{document}